\documentclass[11pt,a4paper]{amsart}
\usepackage{amsmath,amsfonts,xargs,microtype,amssymb,bbm}

\usepackage[utf8]{inputenc}
\usepackage{graphicx} % Required for inserting images
\usepackage{amsthm}
\usepackage{tikz}
\usepackage{enumitem}
\usetikzlibrary{calc}
\usepackage{tkz-tab}
\usepackage{theoremref}
\usepackage{thmtools, thm-restate}
\declaretheorem{theorem}
\usepackage{thm-restate}
\usepackage{hyperref}
\usepackage{comment}
\usepackage{color}
\usepackage{mathtools}
\usepackage{subcaption}
\mathtoolsset{showonlyrefs}

\hypersetup{colorlinks=true,urlcolor=black, pdftitle="Lp-zonoids"}
\begin{document}
\newcommand{\MainTheoremExtra}{.}
\newcommand{\MainTheoremExtraTwo}{ }
\newcommand{\MainTheoremExtraThree}{Moreover, 
        in
        % f we define 
        % \begin{align*}
        %     v_+:=\sum_{v_i>0}v_ie_i\qquad \textup{and}\qquad v_-:=\sum_{v_i<0}v_ie_i,
        % \end{align*}
        % then equality holds on the right if and only if either $v\in \mathbb{R}^m_+\cup (-\mathbb{R}^m_+)$, or $v\notin \mathbb{R}^m_+\cup (-\mathbb{R}^m_+)$ and the faces $F(K;v_+)$ and $F(K;v_-)$ intersect. 
        the case where $K$ is the ball $B_q^m$ with $1\leq q<\infty$,  equality holds on the right if and only if $v\in \mathbb{R}^m_+\cup (-\mathbb{R}^m_+)$.  The left inequality is sharp.}
\newcommand{\R}{ {\mathbb R}}
\newcommand{\N}{ {\mathbb N}}
\newcommand{\s}{ {\mathbb S}}
\newcommand{\conv}{{\operatorname{conv}}}
\newcommand{\supp}{{\operatorname{supp}}}
\newcommand\numberthis{\addtocounter{equation}{1}\tag{\theequation}}

\theoremstyle{plain}

\newtheorem{lemma}[theorem]{Lemma}
\newtheorem{corollary}[theorem]{Corollary}
\newtheorem{prop}[theorem]{Proposition}

\newtheorem{conj}{Conjecture}

\theoremstyle{remark}
\newtheorem{remark}{Remark}

\makeatother

\title[$L_p$-Rogers--Shephard inequalities]{$L_p$-Rogers--Shephard type inequalities for $L_p$-zonoids and symmetric bodies}

\author[M. Fradelizi, A. Manui, M. Meyer, C. S. Ndiaye]{Matthieu Fradelizi, Auttawich Manui, Mark Meyer, and Cheikh Saliou Ndiaye}


\date{}

\begin{abstract} 
    We study generalizations of the classical Rogers--Shephard inequalities in the framework of Firey $L_p$-summation. We first consider the class of asymmetric $L_p$-zonoids. In this setting, we show that proving a sharp $L_p$-Rogers--Shephard inequality for asymmetric $L_p$-zonoids in $\R^n$ is equivalent to proving a sharp inequality between the volumes of projections of $B_q^m\cap \R^m_+$ and $B_q^m$ onto an $n$-dimensional subspace $E$, where $q$ is the H\"older conjugate of $p$. We conjecture that the inequality is sharp when the subspace $E$ is a coordinate subspace. We fully establish this inequality along with equality conditions in the case $p =2$. For general $p$, we prove it in the case $n=m-1$, $n=1$, and discuss several particular cases, including an averaged version and a local version of the inequality. We then turn to the setting of convex bodies having a center of symmetry. Rogers and Shephard also proved a sharp version of their inequality for bodies in this class. We conjecture a similar bound for the $L_p$-summation, and we establish our conjecture for the particular case of asymmetric $L_1$-zonoids, which, in particular, proves our conjecture in the planar case.
\end{abstract}

\maketitle
% \setcounter{tocdepth}{2}
% \tableofcontents

\section{Introduction}

\subsection{\texorpdfstring{$L_p$}{Lp}-Rogers--Shephard inequalities}
    We study variations of the classical inequalities of Rogers and Shephard that establish sharp bounds comparing a convex body with its reflection with respect to the origin. Formally, we study convex bodies, which are the sets $K\subset\mathbb{R}^n$ that are convex, compact, and have nonempty interior. An important notion in our study is the Minkowski sum of convex bodies $K$ and $L$, which is given by
    \begin{equation*}
        K+L:=\{k+l:k\in K,l\in L\}.
    \end{equation*}
    If we define $-K:=\{-k:k\in K\}$ to be the reflection of $K$ with respect to the origin, then we discuss the difference body $K-K:=K+(-K)$. We also study the convex hull of $K\cup-K$, denoted by $\textup{conv}(K,-K)$, which is defined to be the smallest convex set containing both $K$ and $-K$. As we will see shortly, the sets $K-K$ and $\textup{conv}(K,-K)$ are related to each other through the concept of Firey summation.

    The first inequality of interest, established by Rogers and Shephard in \cite{RS-57}, is the bound
    \begin{equation}
        \label{eq:Classical-Rogers-shephard-ineq}
        |K -K| \leq \binom{2n}{n}|K|,
    \end{equation}
    where $|\cdot |$ denotes the Lebesgue measure. The bound \eqref{eq:Classical-Rogers-shephard-ineq} is sharp, and equality holds precisely when $K$ is a simplex. Another result of interest comes from the paper \cite{RS-58}, where Rogers and Shephard proved the following sharp bound for $K$ being a convex body containing the origin
    \begin{equation}
        \label{eq:Rogers-Shephard-ineq-infty}
        \left|\conv(K,-K) \right| \leq 2^n |K|,
    \end{equation}
    with equality if and only if $K$ is a simplex with one vertex at the origin.
    
    In the paper \cite{RS-58-2}, Rogers and Shephard also proved that if $K$ has a center of symmetry (i.e., there exists $a\in K$ such that $K-a=a-K$) and contains the origin, then the above inequality can be improved as follows 
     \begin{equation}
        \label{eq:Rogers-Shephard-ineq-infty-sym}
        \left|\conv(K,-K) \right| \leq (n+1) |K|,
    \end{equation}
    with equality if and only if there exist $a,b$ on the boundary of $K$ such that all two dimensional sections of $K$ by a plane containing $a$ and $b$ are triangles with $a$ and $b$ as vertices or quadrilaterals with $a$ and $b$ as opposite vertices.
    %A classical inequality in convex geometry due to the Rogers and Shephards \cite{RS-57} states that for any convex body $K \subset \R^n$, that is, a convex and compact set with nonempty interior,
    %\begin{equation}
        %\label{eq:Classical-Rogers-shephard-ineq}
        %|K + (-K)| \leq \binom{2n}{n}|K|,
    %\end{equation}
    %where $K + L =\{ x + y : x \in K,y \in L \} $ denotes the Minkowski sum, $-K$ the reflection of $K$ with respect to the origin, and $| K |$  denote the Lebesgue measure. Inequality \eqref{eq:Classical-Rogers-shephard-ineq} is sharp, and equality holds if and only if $K$ is a simplex. 

    %Another remarkable inequality due to Rogers and Shephards \cite{RS-58} asserts that for any convex body $K$ containing the origin
    %\begin{equation}
     %   \label{eq:Rogers-Shephard-ineq-infty}
      %  \left|\conv(K,-K) \right| \leq 2^n |K|,
    %\end{equation}
    %where $\conv \{ K,-K \} $ denotes the convex hull of $K$ and its reflection. Equality holds if and only if $K$ is a simplex with one of the vertex is at the origin.
 
    Firey extended the notion of Minkowski summation in the paper \cite{F-62}, defining the $L_p$-sum (or Firey sum) of convex bodies containing the origin. The $L_p$-sum is defined through the concept of the support function. Recall that the support function of a convex set  $K\subset\mathbb{R}^n$ is the function $h_K:\mathbb{R}^n\rightarrow\mathbb{R}$ defined by
    \begin{equation*}
        h_K(x):=\sup_{k\in K}\langle x,k\rangle.
    \end{equation*}
    For fixed $1\leq p<\infty$, there exists a convex set $K\oplus_pL$ whose support function satisfies
    \begin{equation}\label{def:Firey-sums}
        h_{K\oplus_p L}(x)=(h_K^p(x)+h_L^p(x))^{\frac{1}{p}},
    \end{equation}
    which is well-defined as long as both $K$ and $L$ contain the origin. In the case $p=\infty$, we take $h_{K\oplus_{\infty} L}:=\max\{h_K,h_L\}$. The $L_p$-sum of the sets $K$ and $L$ is defined to be precisely the set $K\oplus _p L$. We pay special attention to the case $p=1$, where $K\oplus_1L=K+L$ is the Minkowski summation, and the case $p=\infty$, where $K\oplus_{\infty}L=\textup{conv}(K,L)$ is the convex hull of $K\cup L$. 
    % Lutwak further developed the theory of Minkowski summation, introducing the notion of $L_p$-mixed volumes and deriving a family of inequalities in the expositions \cite{L-93,L-96}. 
    
    %In \cite{F-62}, Firey introduced an extension of the Minkowski sum. Recall that the support function of $K$ is given by $h_K(u)=\sup \{\langle x, u\rangle \mid x \in K\}$. For $p \in [1, \infty]$, the $L_p$-sum (or Firey sum) $K \oplus_p L $ of convex bodies $K,L$ containing the origin is defined via support functions:
    %\begin{equation}
     %   \label{def:Firey-sums}
      %  h_{K\oplus_p L}^p := h_K^p +h_L^p,
    %\end{equation}
    %with the case $ p = \infty$ understood in the limiting sense. Note that the case $ p =1 $ and $p=\infty$ correspond, respectively, to the Minkowski sum and to the convex hull of the union. The theory of $L_p$-sum was further developed by Lutwak, who created the notion of $L_p$-mixed volumes and established a rich family of associated inequalities; see, for instance, \cite{L-93,L-96}.

    % In the paper \cite{BC-08}, Bianchini and Colesanti made initial efforts to generalize Rogers--Shephard inequalities to the setting of $L_p$-summation. To state the problem, fix $1\leq p\leq \infty$, and fix an integer $n\geq 1$. 
    The following general inequality linking inequalities \eqref{eq:Classical-Rogers-shephard-ineq} and \eqref{eq:Rogers-Shephard-ineq-infty}
    has been conjectured: for any convex body $K\subset\mathbb{R}^n$ containing the origin
\begin{equation}\label{eq:Rogers-Shephard-ineq-p}
        |K\oplus_p-K|\leq \kappa_{n,p}|K|\quad \hbox{where}\quad \kappa_{n,p}=\sum_{i=0}^n\binom{n/q}{ i/q}^{-1}\binom{n}{i}^2,
    \end{equation}
     with equality for a simplex having the origin as a vertex, where $q$ is the H\"older conjugate of $p$ ($\frac{1}{p}+\frac{1}{q}=1$) and the binomial coefficient is defined using the gamma function,
     $$\binom{x}{y}:=\frac{\Gamma(x+1)}{\Gamma(y+1) \Gamma(x-y+1)}.$$ The inequality was established for $n=2$ by Bianchini and Colesanti \cite[Theorem 1]{BC-08} and in $\R^n$ for locally anti-blocking bodies in \cite{MNZ-25}, while the equality cases were proved in \cite{FMMN-26-1}. In this paper, we shall consider improvements of this inequality for restricted classes of convex bodies, i.e., asymmetric $L_p$-zonoids on one side and convex bodies having a center of symmetry on the other side.  

\subsection{Asymmetric \texorpdfstring{$L_p$}{Lp}-zonoids}
We first consider improvements of inequality \eqref{eq:Rogers-Shephard-ineq-p} within the class of asymmetric $L_p$-zonoids for $1\leq p \leq +\infty$. An asymmetric $L_p$-zonotope, introduced by Weberndorfer \cite{W-13}, is defined as the $L_p$-sum of finitely many segments containing the origin. For example, if $(e_1,\dots,e_n)$ is the canonical basis of $\R^n$, then 
\[
[0,e_1]\oplus_p\cdots\oplus_p [0,e_n]=B_q^n\cap\R_+^n,
\]
where $q$ is the H\"older conjugate of $p$ and, for $1\le q<+\infty$, $B_q^n=\{x=(x_1,\dots,x_n)\in\R^n:|x_1|^q+\cdots+|x_n|^q\le1\}$ and $B_\infty^n=[-1,1]^n$ denote the $\ell_q$-ball in $\R^n$.
An asymmetric $L_p$-zonoid is the Hausdorff limit of asymmetric $L_p$-zonotopes. We study the following conjecture.

% Since for any $a,b>0$, $p\ge1$ and $u\in \mathbb{S}^{n-1}$ one has $[-au,bu]=[0,-au]\oplus_p[0,bu]$, an asymmetric $L_p$-zonotope $Z$ can be written in the form
% \[
% Z=[0,v_1]\oplus_p\cdots\oplus_p[0,v_m],
% \]
% where $v_1,\dots,v_m\in\R^n$. In this case, we get 
% \[
% Z\oplus_p -Z=[-v_1,v_1]\oplus_p\cdots\oplus_p[-v_m,v_m].
% \]
% In terms of the support function, introducing the vectors $u_i=v_i/|v_i|_2$, this means that there exists $c_i>0$ such that for any $x\in\R$,
% \[
% h_{Z}(x)=\left(\sum_{i=1}^mc_i\langle x,u_i\rangle_+^p\right)^\frac{1}{p}\quad\hbox{and}\quad h_{Z\oplus_p -Z}(x)=\left(\sum_{i=1}^mc_i|\langle x,u_i\rangle|^p\right)^\frac{1}{p}, 
% \]
% where for any $a\in\R$, one denotes $a_+=\max(a,0)$. 
% An asymmetric $L_p$-zonoid $Z$ is the Hausdorff limit of asymmetric $L_p$-zonotopes. In terms of support functions, $Z$ is an asymmetric $L_p$-zonoid if and only if there exists a positive measure $\mu$ on $S^{n-1}$ such that 
% \[
% h_{Z}(x)=\left(\int_{S^{n-1}}\langle x,u\rangle_+^pd\mu(u)\right)^\frac{1}{p}.
% \]
% Notice also that the support function of $Z \oplus_p -Z$ is simply
% \[
% h_{Z\oplus -Z}(x)=\left(\int_{S^{n-1}}|\langle x,u\rangle|^pd\mu(u)\right)^\frac{1}{p}.
% \]

    \begin{conj} \label{conj:Rogers-Shephard-ineq-p-zonoids}
        For $p\ge 1$, let $Z \subset \R^n$ be an asymmetric $L_p$-zonoid. Then
        \begin{equation}\label{eq:Rogers-Shephard-ineq-p-zonoids-conj}
            |Z \oplus_p -Z| \leq 2^n\,|Z|.
        \end{equation}
        Moreover, if $p>1$, equality holds if and only if $Z$ is contained in a hyperplane or is the $L_p$-sum of $n$ segments with the origin as an endpoint, or equivalently, $Z$ is the image of $B_q^n\cap\R_+^n$ under a nonsingular linear transform, where $q$ satisfies $1/p+1/q=1$.
    \end{conj}
    We note that Conjecture~\ref{conj:Rogers-Shephard-ineq-p-zonoids} is true when $p=1$, and in fact is an equality due to the translation invariance of volume under Minkowski summation. On the other hand, when $p = \infty$, one can approximate an $L_\infty$-zonoid by $L_\infty$-zonotopes, and the $L_\infty$-sum corresponds to taking convex hulls, hence an $ L_\infty$-zonotope is simply a polytope containing the origin, and an $L_\infty$-zonoid is simply any convex body containing the origin. Therefore, Conjecture~\ref{conj:Rogers-Shephard-ineq-p-zonoids} reduces to \eqref{eq:Rogers-Shephard-ineq-infty}. Moreover, Conjecture~\ref{conj:Rogers-Shephard-ineq-p-zonoids} can equivalently be expressed as a sharp quantitative comparison between the volumes of projections of $B_q^m \cap \mathbb{R}^m_+$ and those of $B_q^m$ onto the same subspace, as in the conjecture below.  We let $P_EK$ denote the orthogonal projection of a set $K$ onto the subspace $E$.
    We refer to the survey \cite{NT-23} for closely related topics.
    \begin{conj} \label{conj:projection-cap-conj}
        Fix $q \ge 1$ and  $m\ge n \geq 1$. Then, for any $n$-dimensional subspace $E$ of $\mathbb{R}^m$, one has
        \begin{equation} \label{eq:projection-cap-conj}
            | P_E B_q^m|_n\leq 2^n|P_E(B_q^m \cap \mathbb{R}^m_+)|_n .
        \end{equation}
    Moreover, equality holds if and only if there exists a partition  $U_0, U_1, \dots, U_n $ of $[m]$ such that $U_j \neq \emptyset$ for every $ 1 \leq j \leq n$, $U_0 = \{ i\in [m] :P_E e_i = 0 \}$, and 
    % there is a partition of $\{1,\dots,m\}$ into $n$ non empty subsets $U_1,\dots,U_n$, and 
    for every $1\le j\le n$, there exist unit vectors $w_j\in\sum_{i\in U_j}\R_{>0} e_i$ such that $E=\textup{span}(w_1,\dots,w_n)$.
    \end{conj}

    We first show in Theorem~\ref{thm: equivalent-RS-Proj} that Conjecture ~\ref{conj:Rogers-Shephard-ineq-p-zonoids}, when restricted to the setting of asymmetric $L_p$-zonotopes, is equivalent to Conjecture~\ref{conj:projection-cap-conj}. Note that this is sufficient if we wish to study the inequality in Conjecture ~\ref{conj:Rogers-Shephard-ineq-p-zonoids}, since any asymmetric $L_p$-zonoid can be approximated by asymmetric $L_p$-zonotopes. Our argument is based on the technique of Madiman, Mathieu Meyer, Zvavitch, and the first-named author in \cite{FMMZ-24}. We then give an affirmative answer to Conjecture~\ref{conj:projection-cap-conj} in the case $q=2$ in the following Theorem~\ref{thm:Rogers-Shephard-ineq-2-zonoids}, using a method introduced by Barthe \cite{B-01} based on the reverse Brascamp--Lieb inequality \cite{B-98, B-04, BKX-23}. 
    \begin{restatable}{theorem}{ProjectionCaptwo}
        \label{thm:projection-cap-conj-2}
    Conjecture~\ref{conj:projection-cap-conj} holds for $ q= 2$% 
    \MainTheoremExtra
    \end{restatable}
    The approach above only yields an equality characterization of Conjecture~\ref{conj:Rogers-Shephard-ineq-p-zonoids} for asymmetric $L_2$-zonotopes. To overcome this, we give an alternative proof of Conjecture~\ref{conj:Rogers-Shephard-ineq-p-zonoids} directly using the continuous version of the reverse Brascamp--Lieb inequality \cite{B-04}.
    
    \begin{restatable}{theorem}{RSInqTwo}
    \label{thm:Rogers-Shephard-ineq-2-zonoids}%
        Conjecture~\ref{conj:Rogers-Shephard-ineq-p-zonoids} holds for $ p =2 $% 
    \MainTheoremExtra
    \end{restatable}
    We further investigate Conjecture~\ref{conj:projection-cap-conj} in the more general case where $B_q^m$ is replaced by an unconditional convex body $K$. We recall that a convex body $K$ is unconditional if for any $x = (x_1,\ldots,x_n) \in K$, one has $(\varepsilon_1 x_{1},\ldots,\varepsilon_n x_{n}) \in K$ for 
    % every permutation $ \sigma$ and 
    every $ \varepsilon \in \{-1,1\}^n$.

    \begin{conj} \label{conj:projection-cap-conj-uncon}
        Let  $m\ge n \geq 1$, let $E$ be an $n$-dimensional subspace of $\mathbb{R}^m$ and $K$ be an unconditional convex body in $\R^m$. Then 
        \begin{equation} \label{eq:projection-cap-conj-uncon}
            | P_E K|_n\leq 2^n|P_E(K \cap \mathbb{R}^m_+)|_n .
            % \geq \frac{1}{2^n} | P_E K|_n.
        \end{equation}
    \end{conj}
    In the following proposition, we give an affirmative answer to Conjecture~\ref{conj:projection-cap-conj-uncon} when $E$ is a hyperplane in $\R^m $, and we also establish a sharp lower bound. 
    \begin{restatable}{prop}{ProjectionIneqHyperplane} \label{prop:Projection-ineq-hyperplane}
    % \begin{prop} \label{prop:Projection-ineq-hyperplane}
    For any $u \in \s^n$ and for any unconditional convex body $K \subset \R^{n+1}$, we have
    \begin{equation}
        \label{eq:Projection-ineq-hyperplane}
        \binom{n}{\lfloor \frac{n+1}{2} \rfloor} |P_{u^\perp} ( K \cap \R^{n+1}_+ )|_n \leq |P_{u^\perp} K|_n \leq 2^n |P_{u^\perp} ( K \cap \R^{n+1}_+ )|_n.
    \end{equation}
    \MainTheoremExtraTwo
    % \end{prop}
    \end{restatable}
    For $K=B_q^m$, the upper bound gives the inequality of Conjecture~\ref{conj:projection-cap-conj}, and we also establish the equality case, while the lower bound is improved with a constant depending on $q$ in Proposition~\ref{prop:projection_hyperplane_lq_ball}. Using Theorem~\ref{thm: equivalent-RS-Proj}, this also establishes Conjecture 
   ~\ref{conj:Rogers-Shephard-ineq-p-zonoids} in the case where $Z$ is the $L_p$-sum of $n+1$ one-sided segments. See also Corollary~\ref{cor:Rogers-Shephard-ineq-p-zonotope_hyperplane}.

    In the following proposition, we prove that Conjecture ~\ref{conj:Rogers-Shephard-ineq-p-zonoids} is true when the subspace $E$ has dimension $1$. %To state the equality cases, we will use the notation
   % \begin{equation*}     F(K;v):=\{k\in K:\langle k,v\rangle =h_K(v)\}  \end{equation*}
   % to denote the facet of a convex body $K\subset\mathbb{R}^m$ in the direction $v\in \mathbb{S}^{m-1}$.

    \begin{restatable}{prop}{ProjectionOne}\label{prop:projection_ineq_one_dimension}
        Let $K\subset\mathbb{R}^m$ be an unconditional convex body, and $v\in \mathbb{S}^{m-1}$ be a direction so that $E:=\textup{span}(v)$ is a subspace of dimension $1$, then 
        \begin{equation*}
            |P_E(K\cap \mathbb{R}^m_+)|_1\leq|P_EK|_1\leq 2|P_E(K\cap \mathbb{R}^m_+)|_1.
        \end{equation*}
        \MainTheoremExtraThree
        % Moreover, 
        % in
        % % f we define 
        % % \begin{align*}
        % %     v_+:=\sum_{v_i>0}v_ie_i\qquad \textup{and}\qquad v_-:=\sum_{v_i<0}v_ie_i,
        % % \end{align*}
        % % then equality holds on the right if and only if either $v\in \mathbb{R}^m_+\cup (-\mathbb{R}^m_+)$, or $v\notin \mathbb{R}^m_+\cup (-\mathbb{R}^m_+)$ and the faces $F(K;v_+)$ and $F(K;v_-)$ intersect. 
        % the case where $K$ is the ball $B_q^m$ with $1\leq q<\infty$, we see that equality holds if and only if $v\in \mathbb{R}^m_+\cup (-\mathbb{R}^m_+)$.  The left inequality is sharp.
    \end{restatable}
     % Notice that there is equality on the left inequality, for example, for $K=B_1^m$ and $v=\frac{1}{\sqrt{2}}(e_i-e_j)$, for some $i\neq j$. {\color{red} Put some computation? in the appendix or something}\\

    Next, we study local maximality in Conjecture~\ref{conj:Rogers-Shephard-ineq-p-zonoids} when $Z$ is a perturbation of a linear image of $B^n_q \cap \R^n_+$.
    \begin{restatable}{theorem}{LocalMax}
    \label{thm:local-max}
    For $p >1$, let $\zeta$ be an asymmetric $L_p$-zonoid, and let $u_1, \ldots,u_n$ be  linear independent vectors of $\R^n$. For $\varepsilon\geq 0$, denote $Z_{\varepsilon} = [0,u_1] \oplus_p \ldots \oplus_p [0,u_n] \oplus_p \left(\varepsilon^{\frac{1}{p}} \zeta\right) $. There exists $\varepsilon_0 = \varepsilon_0 (\zeta,u_1,\ldots,u_n)$ such that for any $0 < \varepsilon \leq \varepsilon_0$, we have
    \begin{equation}
        \label{eq:local-max}
        | Z_{\varepsilon} \oplus_p - Z_{\varepsilon} | \leq 2^n | Z_{\varepsilon} |.
    \end{equation}
    Equality holds if and only if $\zeta$ is of the form $\zeta = [0,a_1u_1] \oplus_p \ldots \oplus_p [0,a_nu_n]$ for some $a_i \geq 0$.
    \end{restatable}
    In the next result, we study an average version of Conjectures~\ref{conj:projection-cap-conj} and~\ref{conj:projection-cap-conj-uncon}. To be more precise, if we integrate inequality \eqref{eq:projection-cap-conj-uncon} with respect to the Haar probability measure on the Grassmannian $G(m,n)$ of $n$-dimensional subspaces of $\R^m$, we get an inequality for the $n$th intrinsic volume. We recall that the $n$th intrinsic volume of $K \subset \R^m$ is defined by
    \begin{equation}
        V_n (K) = \frac{\tbinom{m}{n}\omega_{m}}{\omega_{m-n}\omega_n} \int_{G(m,n)} |P_E K|_n d \mu(E),
    \end{equation}
    where $\mu$ is the Haar probability measure on the Grassmannian $\mathrm{G}(m, n)$ of $n$-dimensional subspaces of $\mathbb{R}^m$. We prove that the average version of \eqref{eq:projection-cap-conj-uncon} is true when $K$ is a 1-unconditional zonoid, in particular, this gives an average version of Conjecture~\ref{conj:projection-cap-conj} for $1\le p\le2$. Recall that $K$ is 1-unconditional if for any $x = (x_1,\ldots,x_n) \in K$, permutation $\sigma$, and $\varepsilon\in \{-1,1\}^n$, we have $(\varepsilon_1 x_{\sigma(1)},\ldots,\varepsilon_n x_{\sigma(n)}) \in K$.
    \begin{restatable}{prop}{ProjectionIneqWeighted}
    \label{prop:Projection-ineq-weighted}
    Let $m\ge n$. If $K\subset\mathbb{R}^m$ is a $1$-unconditional zonoid, then
    \begin{equation}
        \label{eq:Projection-ineq-weighted}
        V_n \left(K  \right)\leq 2^n V_n \left(K \cap \R^m_+ \right).
        % \geq \frac{1}{2^{n}} .
    \end{equation}
    % In particular, we obtain
    % \begin{equation}
    %     \label{eq:Projection-ineq-weighted}
    %         V(B_q^m [n],B_2^m[m-n]) < 2^n V \left(B_q^{m} \cap \R^m_+ [n],B_2^m[m-n]\right).
    % \end{equation}
    Equality holds if and only if there exists $a\ge0$ such that $K=aB_\infty^m$.
\end{restatable}

\subsection{Symmetric bodies}
We study inequality \eqref{eq:Rogers-Shephard-ineq-p} for the class of convex bodies with a center of symmetry (i.e., the convex bodies $K$ such that $K-a=a-K$ for some $a\in K$). We conjecture the following sharp bound.
\begin{conj}
    \label{conj:Rogers-Shephard-ineq_p-sym}
    Let $K\subset \R ^n$ be a convex body that contains the origin and has a center of symmetry, and let $p>1$. Then
  \[
  |K \oplus_p -K|\le \kappa^{(s)}_{n,p}|K|,\quad\hbox{with}\quad  \kappa^{(s)}_{n,p}=\sum_{i=0}^n\frac{\binom{n}{i}}{\binom{n/q}{i/q}}.
  \]
  Equality holds if $K=[0,1]^n$.
\end{conj}
Observe that Conjecture~\ref{conj:Rogers-Shephard-ineq_p-sym} is obvious when $p = 1$, since volume is translation invariant with respect to Minkowski summation, and the inequality is in fact equality for convex bodies symmetric about the origin. Conjecture~\ref{conj:Rogers-Shephard-ineq_p-sym} is true when $p=\infty$, thanks to the Rogers--Shephard inequality \eqref{eq:Rogers-Shephard-ineq-infty-sym}. We prove this conjecture for asymmetric $L_1$-zonoids, using the idea of shadow systems, introduced in this context by Campi and Gronchi \cite{CG-06}.
\begin{restatable}{theorem}{RGsymZonoid}
\label{thm:RG-sym-zonoid}
    Let $p \ge 1. $ If $Z$ is an asymmetric $L_1$-zonoid containing the origin, we have 
    \begin{equation} 
        \label{eq:RG-sym-zonoid}
        | Z \oplus_p -Z| \leq \kappa^{(s)}_{n,p} |Z|, \quad \text{with }\qquad \kappa^{(s)}_{n,p} :=\sum_{i=0}^n\frac{\binom{n}{i}}{\binom{n/q}{i/q}},
    \end{equation}
    where $q$ satisfies $\frac{1}{p}+\frac{1}{q}=1$. Equality holds if $Z$ is a parallelotope with a vertex at the origin.
\end{restatable}
Note that Theorem~\ref{thm:RG-sym-zonoid} implies Conjecture~\ref{conj:Rogers-Shephard-ineq_p-sym} in the planar case since the convex bodies in the plane having a center of symmetry are exactly the asymmetric $L_1$-zonoids.

The paper is organized as follows. Section~\ref{sec:pre} contains background material on asymmetric $L_p$-zonoids. In Section~\ref{subsec:equivalent}, we prove the equivalence between Conjecture~\ref{conj:Rogers-Shephard-ineq-p-zonoids} and Conjecture~\ref{conj:projection-cap-conj}. Section~\ref{sec:proj-hyp} and Section~\ref{sec:proj-line} handle the cases $n =m-1$ and $n=1$, respectively. Next, we dedicate Section~\ref{sec:p=2} to the proofs of Theorem~\ref{thm:projection-cap-conj-2} and Theorem~\ref{thm:Rogers-Shephard-ineq-2-zonoids}. Moving on to Section~\ref{sec: local maximum} and Section~\ref{sec: average}, we discuss the local maximizer and an averaged version of the inequality. We conclude the paper by studying symmetric bodies in Section~\ref{sec:sym-bodies}.

\subsection{Related literature}

The reader may consult the papers \cite{L-93,L-96,LYZ-12} for an overview of the $L_p$-Brunn--Minkowski theory, and the papers \cite{A-10,CG-06,FMMZ-24,W-13} for the developments concerning $L_p$-zonoids/zonotopes. Some authors refer to asymmetric $L_p$-zonoids as $L_p$-projection bodies instead. For example, see \cite{HS-09,HS-09-2,L-05,LYZ-20}. 

We also suggest the papers \cite{BC-08, CCG-99,CG-02,CG-06,CG-06-02} for background on shadow systems. 

Aside from the problems considered in the present paper, there is some interest in the study of the Rogers--Shephard inequality in a more general setting. For example, the papers \cite{AHRYZ-21} and  \cite{A-10, AAGJV-19, AGJV-16,C-06, MMZZ-26} focus on functional analogues.

\subsection{Funding \& Acknowledgment} 
The second-named author is supported by the Chateaubriand Fellowship of the Office for Science \& Technology of the Embassy of France in the United States, the U.S. National Science Foundation Grant DMS-2247771, and the United States-Israel Binational Science Foundation (BSF) Grant 2018115.

The third-named author is supported by the
National Science Foundation through the MSPRF program (award number: 2502794).

The authors would like to thank the laboratories LAMA, at University Gustave Eiffel, and the Institut Mathématique de Jussieu for their hospitality.

\section{Preliminaries} \label{sec:pre}
    We refer to \cite{S-93} for standard background material on convex geometry. We denote the inner product of $x,y\in \mathbb{R}^n$ by $\langle x, y \rangle$. For $p \geq 1$ and for any $x\in \mathbb{R}^n $, $\|x\|_p $ denotes the usual $\ell_p$-norm of $x$.
    Denote the $\ell_p$-ball of $\mathbb{R}^n$ by $B^n_p$. For a set $K \subset \mathbb{R}^n$, we let $|K|_m$ denote the $m$-dimensional Lebesgue measure of $K$ restricted to the affine hull of $K$, which we denote by $\textup{aff} (K)$. For convenience, we simply write $|K|$.
    % {\color{red} AM: I think we didn't use this def at all}
% \subsection{\texorpdfstring{$L_p$}{Lp}-Zonoids} \label{sec:pre-lp-zonoids}
An asymmetric $L_p$-zonotope $Z$ is defined to be the $L_p$-sum of finitely many segments containing the origin. For notational convenience, we write
\[
    \bigoplus_{i\in [m]}^{(p)} K_i :=K_1 \oplus_p \ldots \oplus_p K_m,
\]
for any convex sets $K_1,\ldots,K_m$ containing the origin, where $[m] :=\{1,\ldots,m\}$.
Since for any $a,b>0$, $p\ge1$ and $u\in \mathbb{S}^{n-1}$ one has $[-au,bu]=[0,-au]\oplus_p[0,bu]$, an asymmetric $L_p$-zonotope $Z$ can be written in the form
\[
Z= \bigoplus_{i\in [m]}^{(p)} [0,v_i],
% [0,v_1]\oplus_p\cdots\oplus_p[0,v_m],
\]
where $v_1,\dots,v_m\in\R^n$. In this case, we get 
\(
Z\oplus_p -Z= \bigoplus_{i\in [m]}^{(p)} [-v_i,v_i].
% [-v_1,v_1]\oplus_p\cdots\oplus_p[-v_m,v_m].
\)
If we introduce the vectors $u_i=v_i/\|v_i\|_2$,  then we find from the above representation that there exist numbers $c_i>0$ such that for any $x\in\R^n$ 
\[
h_{Z}(x)=\left(\sum_{i=1}^mc_i\langle x,u_i\rangle_+^p\right)^\frac{1}{p}\quad\hbox{and}\quad h_{Z\oplus_p -Z}(x)=\left(\sum_{i=1}^mc_i|\langle x,u_i\rangle|^p\right)^\frac{1}{p}, 
\]
where we define $a_+ :=\max(a,0)$ for any $a\in \mathbb{R}$.
An asymmetric $L_p$-zonoid $Z$ is the Hausdorff limit of asymmetric $L_p$-zonotopes. In terms of its support function, $Z$ is an asymmetric $L_p$-zonoid if and only if there exists a positive measure $\mu$ on $\s^{n-1}$ such that 
\[
h_{Z}(x)=\left(\int_{\s^{n-1}}\langle x,u\rangle_+^pd\mu(u)\right)^\frac{1}{p}.
\]
Notice also that the support function of $Z \oplus_p -Z$ is simply
\[
h_{Z\oplus -Z}(x)=\left(\int_{\s^{n-1}}|\langle x,u\rangle|^pd\mu(u)\right)^\frac{1}{p}.
\]
% \subsection{Polytope}

\section{Inequalities for asymmetric \texorpdfstring{$L_p$}{Lp}-zonoids} \label{sec:asymmetric-zonoid-inequality}

\subsection{The equivalence between projections and zonoids} \label{subsec:equivalent}

The following lemma justifies the equality statement in Conjecture~\ref{conj:Rogers-Shephard-ineq-p-zonoids}, establishing that the image of $B_q^n \cap \R^n_+$ under a nonsingular linear transform is an $L_p$-sum of $n$ segments with the origin as an endpoint.

\begin{lemma}\label{lem:positive_orthant_support_function}
    Let $1\leq p\leq \infty$, and let $q$ satisfy $1/p+1/q=1$. Then, if $\{e_1,\dots,e_m\}$ is the standard basis of $\mathbb{R}^m$, we have
    \begin{equation*}
        B_q^m\cap \mathbb{R}^m_+=\bigoplus_{i\in [m]}^{(p)} [0,e_i].
        % [0,e_1]\oplus_p\dots \oplus_p[0,e_m].
    \end{equation*}
    Equivalently, in the case where $p<\infty$, we have
    % \begin{equation*}
        $h_{B_q^m\cap\mathbb{R}^m_+}^p(x)=\sum_{i=1}^{m}\langle e_i,x\rangle^p_+$.
    % \end{equation*}
    % where we define $\langle e_i,x\rangle_+ :=\max\{0,\langle e_i,x\rangle\}$. 
    If we take the $L_p$-sum of the segments $[-e_i,e_i]$ instead, we have
    \begin{equation*}
    % $
        B_q^m= \bigoplus_{i\in [m]}^{(p)} [-e_i,e_i],
        % $
        % [-e_1,e_1]\oplus_p\dots\oplus_p[-e_m,e_m],
    \end{equation*}
    and 
    % \begin{equation*}
        $h_{B_q^m}^p(x)=\sum_{i=1}^{m}|\langle e_i,x\rangle|^p,$
    % \end{equation*}
    for
    $p<\infty$.
\end{lemma}

\begin{proof}
    It is straightforward to prove this lemma when either $p=1$ or $p=\infty$, so we will concentrate on the case $1<p<\infty$. First, observe that
    \begin{equation*}
        h_{B_q^m\cap \mathbb{R}^m_+}(x)
        =\sup_{y\in B_q^m\cap \mathbb{R}^m_+}\langle y,x\rangle 
        =\sup_{y\in B_q^m\cap \mathbb{R}^m_+}\sum_{i=1}^{m}\langle e_i,y\rangle \langle e_i,x\rangle
        \leq \sup_{y\in B_q^m\cap \mathbb{R}^m_+} \sum_{i=1}^{m}\langle e_i,y\rangle \langle e_i,x\rangle_+.
    \end{equation*}
    By H\"older's inequality and the fact that $y\in B_q^m\cap \mathbb{R}^m_+$, we have
    \begin{equation*}
        \sum_{i=1}^{m}\langle e_i,y\rangle \langle e_i,x\rangle_+\leq \left(\sum_{i=1}^{m}\langle e_i,y\rangle^q\right)^{\frac{1}{q}}\left(\sum_{i=1}^{m}\langle e_i,x\rangle_+^p\right)^{\frac{1}{p}}\leq \left(\sum_{i=1}^{m}\langle e_i,x\rangle_+^p\right)^{\frac{1}{p}},
    \end{equation*}
    which gives
    \begin{equation*}
         h_{B_q^m\cap \mathbb{R}^m_+}^p(x)\leq \sum_{i=1}^{m}\langle e_i,x\rangle_+^p.
    \end{equation*}
    This is clearly equality when $x=0$. To show that this is equality when $x\neq 0$, define $y_x\in B_q^m\cap  \mathbb{R}^m_+$ by
    \begin{equation*}
        \langle e_j,y_x \rangle^q =\frac{\langle e_j,x \rangle_+^p}{\sum_{i=1}^{m}\langle e_i,x\rangle _+^p}
    \end{equation*}
    for $1\leq j\leq m$. Then, 
    \begin{equation*}
        \langle y_x,x \rangle=\frac{\sum_{j=1}^{m}\langle e_j,x \rangle _+^p}{\left(\sum_{i=1}^{m}\langle e_i,x\rangle^p_+\right)^{\frac{1}{q}}} = \left(\sum_{i=1}^{m}\langle e_i,x\rangle_+^p\right)^{\frac{1}{p}},
    \end{equation*}
    so we must have $h_{B_q^m\cap \mathbb{R}^m_+}(x)=\langle y_x,x\rangle$, proving the required formula. The case where we take the $L_p$-sum of segments $[-e_i,e_i]$ is handled in a nearly identical way, so we omit the proof.
\end{proof}

\begin{lemma} \label{lem:equ-lp-ball}
    Let $1\leq p \leq \infty$, and let $q$ satisfy $\frac{1}{p} + \frac{1}{q} = 1.$ If $Z$ is contained in a hyperplane or is an $L_p$-sum of $n$-segments with the origin as an endpoint, then $|Z \oplus_p -Z| = 2^n|Z|$.
\end{lemma}

\begin{proof}
    First, if $Z$ is contained in a hyperplane, then the inequality is trivial. Indeed, $Z\oplus_p Z$ is also contained in the same hyperplane and hence $|Z \oplus_p -Z| = 0 = |Z| $.

    Now, if $Z$ is the $L_p$-sum of $n$ segments with the origin as an endpoint, then there exists a nonsingular linear transform $T:\mathbb{R}^{n}\rightarrow\mathbb{R}^{n}$ such that 
    % $TZ=B_q^{n}\cap\mathbb{R}^{n}_+$. From Lemma~\ref{lem:positive_orthant_support_function}, 
    % \begin{equation*}
        $TZ=\bigoplus_{i \in [n]}^{(p)} [0,e_i]$.
    % \end{equation*}
    Thus,
    \begin{equation*}
        TZ\oplus_p-TZ=\bigoplus_{i \in [n]}^{(p)}[-e_i,e_i]=B_q^{n}.
    \end{equation*}
    It follows immediately that 
    \begin{equation*}
        |TZ\oplus_p -TZ|=2^n|TZ|.
    \end{equation*}
    Using the fact that $TZ\oplus_p-TZ=T(Z\oplus_p-Z)$, we have
    \begin{equation*}
        |Z\oplus_p-Z|=\frac{1}{|\textup{det}(T)|}|TZ\oplus_p-TZ|=\frac{2^n}{|\textup{det}(T)|}|TZ|=2^n|Z|,
    \end{equation*}
    which is exactly what we need.
\end{proof}
We recall that it is possible to extend the definition of the $L_p$-sum \eqref{def:Firey-sums} to non-convex sets, as proved by Lutwak, Yang, and Zhang in \cite{LYZ-12}. For $p>1$ and $K,L\subset\mathbb{R}^n$, define 
    \begin{equation}
        \label{def:Firey-sums-non-convex}
        K \oplus_p L:=\left\{(1-t)^{1 / q} x+t^{1 / q} y : x \in K, y \in L, t \in[0,1]\right\},
    \end{equation}
    where $q$ is the conjugate of $p$, that is $\frac{1}{q}+\frac{1}{p}=1$. The definition \eqref{def:Firey-sums-non-convex} coincides with \eqref{def:Firey-sums} when  $K,L$ are convex bodies containing the origin, see \cite[Lemma 2]{LYZ-12}. This equivalent definition allows us to prove a form of uniqueness of the $L_p$ representation of a nonsingular linear image of $B_q^n\cap\R^n_+$ as an $L_p$-sum of segments.
\begin{lemma} \label{lem:uniqueness-lq}
    Let $1 < p < \infty$, and let $w_1,\ldots, w_n \in \R^n$ be independent vectors. If there exist $v_1,\ldots,v_m \in\R^n $ with $ m \geq n$ such that 
    \begin{equation} 
        \bigoplus_{j\in [n]}^{(p)} [0,w_j] 
        % \oplus_p \ldots \oplus_p [0,w_n] 
        = \bigoplus_{i\in [m]}^{(p)}[0,v_i],
        % \oplus_p \ldots \oplus_p [0,v_m],    
    \end{equation}
    then there is a partition $U_0, \ldots,U_n$ of $[m]$ such that for any $1 \leq j \leq n$
    \[
        U_j := \{ i  \in [m] : \alpha_{ij} w_j =  v_i \text{ for some }  \alpha_{ij} > 0 \},
    \]
    with $ \sum_{i \in U_j} \alpha_{ij}^{p} = 1 $ and $U_0 = \{ i \in [m]: v_i = 0\} $.
\end{lemma}

\begin{proof}
    By applying a linear transform, it is enough to prove the result when $w_j = e_j$ for all $j \in [n]$. For any $j \in [n]$, we denote
    \[
    U_j := \{ i  \in [m] : \alpha_{ij} e_j =  v_i \text{ for some }  \alpha_{ij} > 0 \}.
    \]
    We show that $U_0,\ldots, U_n$ form a partition of $[m]$. By definition, $U_j \subset [m]$ for every $0\leq j \leq n $, and the sets $U_0,\ldots, U_n$ are pairwise-disjoint. It remains to show that they cover $[m]$.
    
    Using Lemma~\ref{lem:positive_orthant_support_function}, we get that $v_i \in \bigoplus_{j \in [n]}^{(p)}[0,e_j] =B_q^n\cap\R^n_+$ for all $i \in [m]$.
    Let $j \in [n]$. Since $e_j \in \bigoplus_{i \in [m]}^{(p)} [0,v_i] $, it follows from \eqref{def:Firey-sums-non-convex} that
    \begin{equation}
        \label{eq:LYZ-ej}
        e_j = \sum_{i=1}^m t_i^{1/q}\lambda_i v_i,
    \end{equation}
    for some $t_i,\lambda_i \in [0,1]$ with $\sum_{i=1}^m t_i
    =1$, where $q$ is the H\"older conjugate of $p$. 
    Thus, for any $i \not\in U_j \cup U_0$,  $t_i = 0$ or $\lambda_i =0$.
    %there exist $\alpha_{ij}$ such that $\alpha_{ij} e_j =  v_i$, 
    Hence,
    \[
    e_j = \sum_{i \in U_j} t_i^{1/q}\lambda_i v_i =  \sum_{i \in U_j} \left(\frac{t_i}{\sum_{i \in U_j}t_i}\right)^{1/q}\left(\sum_{i \in U_j}t_i\right)^{1/q}\lambda_i v_i \in \bigoplus_{i \in U_j}^{(p)}[0,v_i],
    \]
    where in the last step, we use that $\left(\sum_{i \in U_j}t_i\right)^{1/q}\lambda_i v_i \in [0,v_i]$ and \eqref{def:Firey-sums-non-convex}.
    This step is justified since $\sum_{i\in U_j}t_i\neq 0$. Indeed, if $\sum_{i\in U_j}t_i=0$, then, we would have $t_i=0$ for all $i\in U_j$, which gives a contradiction since $e_j$ is not the origin. Since $\bigoplus_{i \in U_j}^{(p)}[0,v_i]$ is a convex set containing the origin, we get that $[0,e_j] \subset  \bigoplus_{i \in U_j}^{(p)}[0,v_i]$ and so
    \begin{equation} 
    \label{eq:reduced-v}
        \bigoplus_{j\in [n]}^{(p)}[0,e_j] = \bigoplus_{j\in [n]}^{(p)} \bigoplus_{i \in U_j}^{(p)} [0,v_i]=:\mathcal{V}.
    \end{equation}
    For any $ k \in [m]$ such that $ k \not \in \bigcup_{1\leq j \leq n} U_j$, we get
    \[
     h_{B_q^n\cap \R^n_+}^p  \leq h_{\mathcal{V}}^p + h_{[0,v_{k}]}^p  \leq  h_{B_q^n\cap \R^n_+}^p ,
    \]
    where the first inequality follows from \eqref{eq:reduced-v}, and the second is from the assumption $ \bigoplus_{i \in [m]}^{(p)}[0,v_i] = B_q^n\cap \R^n_+$.
    Therefore, $h_{[0,v_{k}]}^p  \equiv 0$, and so $v_k = 0$. Then, $k \in U_0$, verifying that $U_0,\ldots,U_n$ is a partition of $[m]$.
    % Thus,
    % \[
    %     \bigoplus_{j\in [n]}[0,e_j] = \bigoplus_{j\in [n]} \bigoplus_{i \in U_j} [0,v_i].
    % \]
    Finally, we observe that for any $k \in [n]$,
    \[
        [0,e_k]=P_{\R e_k} \left( \bigoplus_{j\in [n]}^{(p)}[0,e_j]\right) 
        = P_{\R e_k} \left(\bigoplus_{j\in [n]}^{(p)} \bigoplus_{i \in U_j}^{(p)} [0,v_i] \right)= \bigoplus_{i \in U_k}^{(p)} [0,v_i].
    \]
    Thus, $h_{[0,e_k]}^p = \sum_{i\in U_k} h_{\alpha_{ik}[0,e_k]}^p = \sum_{i \in U_k} \alpha_{ik}^{p} h_{[0,e_k]}^p$ and so $\sum_{i \in U_k} \alpha_{ik}^{p} = 1$.
\end{proof}

\begin{lemma} \label{lem:sets-projection}
    Let $E$ be an $n$-dimensional subspace of $\R^m$ and 
    let 
    \[
        U_0 = \{ i \in [m]: P_Ee_i =0\} \quad \text{and} \quad u_ i = \frac{P_E e_i}{\|P_E e_i\|_2}, \quad i \in [m] \setminus U_0.
    \]
    Suppose that $[m] \setminus U_0$ can be partitioned into $n$ nonempty sets $U_1,\ldots,U_n$ such that 
    \[
        u_i = u_k \quad \text{if and only if} \quad i,k \text{ belong to the same } U_j,1 \leq j\leq n.
    \]
    Then there exist unit vectors $w_j \in \sum_{i \in U_j} \R_{>0} e_i$ such that $E = \textup{span}(w_1,\dots,w_n)$.
    
    % $U_0, \ldots,U_n$ be a partition of $[m]$ such that
    % Let $v_1,\ldots,v_n \in \R^m$ be the basis of an $n$-dimensional subspace $E$ of $\R^m$, and let
    % $e_1,\ldots, e_m$ be the standard basis of $\R^m$. If there exists a partition $U_0, \ldots, U_n$ of $[m]$ such that $U_0 = \{ i \in [m]: P_Ee_i =0\}$ and for any $1\leq j \leq n$,
    % \[
    %     U_j = \{i \in [m] : \alpha_{ij} v_j = P_E e_i \text{ for some } \alpha_{ij} >0\} \neq \emptyset,
    % \]
    % then there exist unit vectors $w_j \in \sum_{i \in U_j} \R_{\{>0\}} e_i$ with $E = \textup{span}(w_1,\dots,w_n)$.
\end{lemma}

\begin{proof}
    For any $i \not\in U_0$, let $c_i = \|P_E e_i\|_2^2$. 
    Since $U_1,\ldots,U_n$ are nonempty, we can assume that $j \in U_j$ for any $1\leq j \leq n$, relabeling the indices if necessary. Thus, for $1 \leq j \leq n $, and for any $i\in U_j$, we get that $u_i =u_j$. Hence,
    \[
        U_j =\{i \in [m]: u_i = u_j\}.
    \]
    % Now, let us see what this condition gives for the space $E$. 
    % Since for any $1\le j\le n$, $u_{j}=P_Ee_{j}/\sqrt{c_{j}}$, the fact that $u_i=u_{j}$, for all $i\in U_{j}$ implies that 
    For any $1 \leq j \leq n$, and for any $i \in U_j \setminus \{j\}$, we have
    \[
        P_E\left(\frac{e_j}{\sqrt{c_j}}-\frac{e_i}{\sqrt{c_i}}\right)=\frac{P_E e_j}{\sqrt{c_j}}-\frac{P_E e_i}{\sqrt{c_i}}=u_j-u_i=0.
    \]
    Thus, the vectors $\frac{e_{j}}{\sqrt{c_{j}}}-\frac{e_i}{\sqrt{c_i}}\in E^{\bot}. $ Also, it follows from the definition of $U_0$ that for any $i \in U_0$, $e_i \in E^\perp$.
    We claim that
    \[
    E^\bot=\textup{span}\left( \left\{\frac{e_{j}}{\sqrt{c_{j}}}-\frac{e_i}{\sqrt{c_i}} :\ i\in U_{j}\setminus\{j\}, 1\le j\le n\right\}\cup\{ e_i: i \in U_0 \} \right).
    \]
    Note that with $j$ fixed, the vectors $\frac{e_{j}}{\sqrt{c_{j}}}-\frac{e_i}{\sqrt{c_i}}$ 
    % \[
    %     \frac{e_{j}}{\sqrt{c_{j}}}-\frac{e_i}{\sqrt{c_i}}
    % \]
    are linearly independent since each vector contains an $e_i$ which does not appear in the others. Since $U_0,\ldots,U_n$ are pairwise disjoint, all vectors are linearly independent. The total number of vectors is
    \[
        \sum_{j=1}^n\left(\#U_j-1\right)+\#U_0=\left(\sum_{j=1}^n\#U_j\right)-n+\#U_0=m-n ,
    \]
    where $\# U$ denotes the cardinality of the set $U$.
    Since the dimension of $E^\perp $ is $m-n$, the claim follows. Now, for $1\le j\le n$, we define
    \[
    w_j=\sum_{i\in U_{j}}\sqrt{c_i}e_i.
    \]
    A direct computation shows that $w_k\in (E^\bot)^\bot=E$. Finally, the vectors $w_k$, $1\le k\le n$ are linearly independent since they have pairwise disjoint supports and each of them is nonzero. Hence, $E=\textup{span}(w_k: 1\le k\le n).$
\end{proof}

\begin{lemma} \label{lem:equ-projection}
    Let $E$ be an $n$-dimensional subspace of $\R^m$. Suppose there exists a partition  $U_0, U_1, \dots, U_n $ of $[m]$ such that $U_j \neq \emptyset$ for every $ 1 \leq j \leq n$, $U_0 = \{ i\in [m] :P_E e_i = 0 \}$, and 
    for every $1\le j\le n$, there exist unit vectors $w_j\in\sum_{i\in U_j}\R_{>0} e_i$ such that $E=\textup{span}(w_1,\dots,w_n)$. Then, for any $1 \leq q \leq \infty$,
    $|P_E B_q^m| = 2^n |P_E(B_q^m \cap \R^m_+)|$.
\end{lemma}

\begin{proof}
    Notice that the vectors $(w_1,\dots,w_n)$ form an orthonormal basis of $E$. 
% By renormalizing the vectors $w_j$, we can assume, without loss of generality, that they form an orthonormal basis of $E$. 
There exists $\lambda_i>0$ for $i \in [m] \setminus U_0$ such that  
\[
w_j=\sum_{i\in U_j}\lambda_ie_i.
% \quad\hbox{and}\quad \sum_{i\in U_j}\lambda_i^2=1.
\]
Then, for any $1\le j\le n$ and for any $i\in U_j$ one has $\langle e_i,w_k\rangle=0$ for $k\neq j$ and $\langle e_i,w_j\rangle=\lambda_i$. Hence $P_E(e_i)=\lambda_iw_j$. Therefore, for any $1\le j\le n$, 
%if we denote $n_j=\textup{card}(U_j)$ and $U_j=\{i_{j,i_1},\dots,i_{j,n_j}\}$\[P_E\left(([0,e_{j,i_1}]\oplus_2\cdots\oplus_2[0,e_{i_{j,n_j}}]\right)\]
\[
P_E\left(\bigoplus^{(p)}_{i\in U_j}[0,e_i]\right)=\bigoplus^{(p)}_{i\in U_j}[0,P_Ee_i]=\bigoplus^{(p)}_{i\in U_j}[0,\lambda_iw_j]=\left(\sum_{i\in U_j} \lambda_i^p\right)^\frac{1}{p}[0,w_j].
\]
Taking the $L_p$-sum and using Lemma~\ref{lem:positive_orthant_support_function}, we conclude that 
\[
    P_E(B_q^m\cap \mathbb{R}^m_+) = P_E\left(\bigoplus^{(p)}_{i\in [m]}[0,e_i]\right) = \bigoplus^{(p)}_{j\in [n]}\left(\sum_{i\in U_j} \lambda_i^p\right)^\frac{1}{p}[0,w_j] .
\]
Similarly, we have
\[
    P_EB_q^m = P_E\left(\bigoplus^{(p)}_{i\in [m]}[-e_i,e_i]\right) = \bigoplus^{(p)}_{j\in [n]}\left(\sum_{i\in U_j} \lambda_i^p\right)^\frac{1}{p}[-w_j,w_j].
\]
Since $(w_1,\dots, w_n)$ is an orthonormal basis of $E$,  
\begin{align}
&|P_E(B_q^m\cap \mathbb{R}^m_+)|_n= \prod_{j\in[n]} \left(\sum_{i\in U_j} \lambda_i^p\right)^\frac{1}{p} \left|\bigoplus^{(p)}_{i\in [n]}[0,w_i]\right|_n
= \prod_{j\in[n]} \left(\sum_{i\in U_j} \lambda_i^p\right)^\frac{1}{p}\left|\bigoplus^{(p)}_{i\in [n]}[0,b_i]\right|
\\
&=\prod_{j\in[n]} \left(\sum_{i\in U_j} \lambda_i^p\right)^\frac{1}{p}\frac{1}{2^n}\,\left|\bigoplus^{(p)}_{i\in [n]}[-b_i,b_i]\right|
=\prod_{j\in[n]} \left(\sum_{i\in U_j} \lambda_i^p\right)^\frac{1}{p}\frac{1}{2^n}\,\left|\bigoplus^{(p)}_{i\in [n]}[-w_i,w_i]\right|_n
=\frac{1}{2^n}\,|P_E B_q^m|_n,
\end{align}
where we denote $\{b_1,\ldots,b_n\}$ for the standard basis of $\R^n$
\end{proof}

\begin{theorem}
    \label{thm: equivalent-RS-Proj}
    Restricted to the class of asymmetric $L_p$-zonotopes, Conjecture~\ref{conj:Rogers-Shephard-ineq-p-zonoids} 
      is equivalent to Conjecture~\ref{conj:projection-cap-conj}. More precisely, for fixed $1 < p< \infty$ and for $q$ satisfying $\frac{1}{p} +\frac{1}{q} = 1$, the following statements are equivalent:
    \begin{enumerate}
        \item[(i)] If $Z\subset \R^n$ is an asymmetric $L_p$-zonotope, then
        $$
            |Z \oplus_p -Z| \leq 2^n|Z|.
        $$
        Equality holds if and only if $Z$ is contained in a hyperplane or is an $L_p$-sum of $n$-segments with the origin as an endpoint.
        \item[(ii)] If $ m\geq n$, and if $E$ is an $n$-dimensional subspace of $\R^m$, then 
        $$
         |P_E B^m_q|_n\le 2^n   |P_E(B_q^m \cap \R^m_+)|_n .
        $$ 
        Equality holds if and only if there exists a partition  $U_0, U_1, \dots, U_n $ of $[m]$ such that $U_j \neq \emptyset$ for every $ 1 \leq j \leq n$, $U_0 = \{ i\in [m] :P_E e_i = 0 \}$, and 
    % there is a partition of $\{1,\dots,m\}$ into $n$ non empty subsets $U_1,\dots,U_n$, and 
    for every $1\le j\le n$, there exist unit vectors $w_j\in\sum_{i\in U_j}\R_{>0} e_i$ such that $E=\textup{span}(w_1,\dots,w_n)$.
    \end{enumerate}
    % Specifically, if $Z$ is a full-dimensional asymmetric $L_p$-zonotope that is the $L_p$-sum of $m\geq n$ segments that each contain the origin, then there exists a subspace $E\subset \mathbb{R}^m$ of dimension $n$, and a positive constant $C$ that depends on $Z$ such that 
    % \begin{align*}
    %     |P_EB_q^m|_n=C|Z\oplus_p-Z|\qquad \textup{and}\qquad |P_E(B_q^m\cap \mathbb{R}^m_+)|_n=C|Z|.
    % \end{align*}
    % Moreover, $Z$ is the image of $B_q^n\cap \mathbb{R}^n_+$ under a non-singular linear transform if and only if there exist non-empty, pairwise-disjoint sets $U_1,\dots, U_n \subset \{1,\dots,m\}$ and vectors $w_j\in \sum_{i\in U_j}\mathbb{R}_+e_i$ for $1\leq j\leq n$ such that $E=\textup{span}(w_1,\dots,w_n)$. {\color{purple} MM: Need to prove the equality condition part} 
\end{theorem}

\begin{proof}
    We begin with the implication $(ii) \Rightarrow (i)$. Assume that $(ii)$ holds. 
    % By approximation, it is enough to consider the case where $Z$ is a full-dimensional asymmetric $L_p$-zonotope. That is, $Z$ is the $L_p$-sum of finitely many nondegenerate segments containing the origin. 
    We write $$Z=\bigoplus_{i\in [m]}^{(p)}[0,u_i],$$ where $u_1,\dots,u_m\in\mathbb{R}^n$ are distinct nonzero vectors. It is enough to consider the case when $Z$ is full-dimensional. Thus, we have $\textup{span}(u_1,\dots,u_m)=\mathbb{R}^n$, so that $m\geq n$.
    Let $\{e_1,\dots,e_m\}$ denote the canonical basis for $\mathbb{R}^m$, and define the linear operator $U:\mathbb{R}^m\rightarrow\mathbb{R}^n$ by $Ue_i=u_i$, so that $U$ is an $n\times m$ matrix whose columns are $u_1,\dots,u_m$. Using Lemma~\ref{lem:positive_orthant_support_function}, we have that
    \begin{align*}
        Z=\bigoplus_{i\in [m]}^{(p)}[0,Ue_i]=U\left(\bigoplus_{i\in [m]}^{(p)}[0,e_i]\right)=U(B_q^m\cap \mathbb{R}^m_+),
    \end{align*}
    and that 
    \begin{align*}
        Z\oplus-Z=\bigoplus_{i\in [m]}^{(p)}[-Ue_i,Ue_i]=U\left(\bigoplus_{i\in [m]}^{(p)}[-e_i,e_i]\right)=U(B_q^m),
    \end{align*}
    so that 
    \begin{align*}
        Z=U(B_q^m\cap \mathbb{R}^m_+)\qquad \textup{and}\qquad Z\oplus_p-Z=U(B_q^m).
    \end{align*}
    %{\color{red}MF: the ten preceding lines could be simplified a bit if we would notice before that for any linear $T:\R^m\to\R^n$ and any sets $A,B$ in $\R^m$ one has $T(A\oplus_pB)=T(A)\oplus_p T(B)$. Using this, we directly have 
    %\begin{equation*}
    %    Z\oplus_p-Z=[-U(e_1),U(e_1)]\oplus_p\dots\oplus_p [-U(e_m),U(e_m)]=U([-%e_1,e_1]\oplus_p\dots\oplus_p [-e_m,e_m])=U(B_q^m).
    %\end{equation*}}
    The matrix $U$ has a singular value decomposition $U=R_1\Sigma R_2$, where $R_1$ is an $n\times n$ orthogonal matrix, $\Sigma$ is an $n\times m$ square diagonal matrix with nonnegative diagonal entries $\lambda_1,\dots,\lambda_n$, and $R_2$ is an $m\times m$ orthogonal matrix. The fact that $U$ has rank $n$ forces the values $\lambda_i$ to be positive. Let $\{w_1,\dots,w_n\}$ denote the canonical basis for $\mathbb{R}^n$, and define the linear operator $T:\mathbb{R}^n\rightarrow\mathbb{R}^m$ by $T(R_1w_i)=\lambda_i^{-1}R_2^{-1}e_i$.
    % , so that $T$ is an $m\times n$ matrix whose $i$th column is $\lambda_i^{-1}R_2^{-1}e_i$. 
    For $x\in \mathbb{R}^m$, we compute
    \begin{align*}
        TR_1\Sigma R_2x&=\sum_{i=1}^{m}\langle R_2x,e_i\rangle TR_1\Sigma e_i
        =\sum_{i=1}^n\langle x,R_2^{-1}e_i\rangle \lambda_i TR_1w_i
        \\
        &=\sum_{i=1}^n\langle x,R_2^{-1}e_i\rangle R_2^{-1}e_i=P_Ex,
    \end{align*}
    where $E = \textup{span} (R^{-1}_2e_1,\ldots, R^{-1}_2e_n)$, so we have that $TR_1\Sigma R_2=P_E$. 
    Since $m\geq n$, the volume of the image under $T$ of a convex body $K\subset \mathbb{R}^n$ changes according to the formula
    \begin{equation*}
        |TK|_n=\sqrt{\textup{det}(T^{*}T)}\cdot |K|_n.
    \end{equation*}
    Then we find that the $n$-dimensional volume of $P_E B_q^{m}$ is given by
    \begin{equation}\label{eq:volume_lq_ball_projection}
        |P_EB_q^{m}|_n=|TR_1\Sigma R_2B_q^{m}|_n=\sqrt{\textup{det}(T^{*}T)}\cdot|UB_q^{m}|_n=\sqrt{\textup{det}(T^{*}T)}\cdot |Z\oplus_p-Z|.
    \end{equation}
    We used the fact that the volume is invariant under the rotation $R_1$. Similarly, we find that 
    \begin{equation}\label{eq:volume_orthant_lq_ball_projection}
        |P_E(B_q^{m}\cap \mathbb{R}^m_+)|_n=\sqrt{\textup{det}(T^{*}T)}\cdot |U(B_q^m\cap \mathbb{R}^m_+)|_n=\sqrt{\textup{det}(T^{*}T)}\cdot |Z|.
    \end{equation}
    From \eqref{eq:volume_lq_ball_projection}, \eqref{eq:volume_orthant_lq_ball_projection}, and the inequality in $(ii)$, we have
    \begin{equation}
        \label{eq:equality-check-ii-i}
        |Z\oplus_p-Z| = \frac{1}{\sqrt{\textup{det}(T^{*}T)}}\cdot |P_EB_q^m|_n \leq \frac{1}{\sqrt{\textup{det}(T^{*}T)}}\cdot 2^n|P_E(B_q^{m}\cap \mathbb{R}^m_+)|_n = 
        2^n |Z|
        % 2^n |Z|=\frac{1}{\sqrt{\textup{det}(T^{*}T)}}\cdot 2^n|P_E(B_q^{m}\cap \mathbb{R}^m_+)|_n\geq \frac{1}{\sqrt{\textup{det}(T^{*}T)}}\cdot |P_EB_q^m|_n=|Z\oplus_p-Z|,
    \end{equation}
    which establishes the inequality in $(i)$. 
    
    Now, we focus on the equality characterization in $(i)$. 
    % Equality holds if and only if $Z$ is contained in a hyperplane or it is an $L_p$-sums of $n$-segments with the origin as an endpoint. 
    The sufficient condition follows from Lemma~\ref{lem:equ-lp-ball}. Then we will prove the other direction is true, assuming that $Z$ is full-dimensional and $|Z \oplus_p -Z | = 2^n |Z|$. Thus, it follows from  \eqref{eq:equality-check-ii-i} that
    \[
        |P_EB_q^m|_n = 2^n|P_E(B_q^{m}\cap \mathbb{R}^m_+)|_n.
    \]
    Using the equality case in $(ii)$, there exists a partition  $U_0, U_1, \dots, U_n $ of $[m]$ such that $U_j \neq \emptyset$ for every $ 1 \leq j \leq n$, $U_0 = \{ i\in [m] :P_E e_i = 0 \}$, and 
    for every $1\le j\le n$, there exist unit vectors $w_j\in\sum_{i\in U_j}\R_{>0} e_i$ such that $E=\textup{span}(w_1,\dots,w_n)$. Let $\lambda_1,\ldots, \lambda_m \geq 0$ be such that for every $ 1 \leq j \leq n$,
    \[
        w_j = \sum_{i \in U_j} \lambda_i e_i.
    \]
    % We follows the same computation in Lemma~\ref{lem:equ-projection} to obtain 
    Observe that
    \[
        P_E \left( \bigoplus_{i \in U_j}^{(p)} [0,e_i]\right) 
        = \bigoplus_{i \in U_j}^{(p)} [0,P_E e_i] =\bigoplus_{i \in U_j}^{(p)} \left[0,\sum_{j=1}^n \langle w_j,e_i \rangle w_j\right] .
    \]
    Since $U_1,\ldots,U_n$ are pairwise-disjoint, we obtain that
    \begin{equation} \label{eq:projection-001}
        P_E \left( \bigoplus_{i \in U_j}^{(p)} [0,e_i]\right)  
        % =\bigoplus_{i \in U_j}^{(p)} \left[0,\lambda_i w_j\right] 
        = \left(\sum_{i\in U_j} \lambda_i^p\right)^{\frac{1}{p}} [0,w_j],
    \end{equation}
    where the last equality follows from a direct computation using the definition of $L_p$-summation.
    % \[
    %     P_E (B_q^m \cap \R^m_+)
    %     = \bigoplus^{(p)}_{j \in [n]} \left(\sum_{i\in U_j} \lambda_i^p\right)^{\frac{1}{p}} [0,w_j]
    % \]
    Using Lemma~\ref{lem:positive_orthant_support_function}, we get that
    \[
        TZ = P_E (B_q^m \cap \R^m_+) = 
        % P_E \left( \bigoplus_{i \in [m]}^{(p)} [0,e_i]\right) = 
        % \bigoplus^{(p)}_{j \in [n]} P_E \left( \bigoplus_{i \in U_j}^{(p)} [0,e_i]\right) = 
        \bigoplus^{(p)}_{j \in [n]} \left(\sum_{i\in U_j} \lambda_i^p\right)^{\frac{1}{p}} [0,w_j].
    \]
    Therefore, $Z$ is image of $B_q^n\cap \R^n_+$.
    % $Z= \bigoplus_{j\in[n]}^{(p)} \left(\sum_{i\in U_j}\lambda_i^p\right)^{\frac{1}{p}} [0,T^* w_j]$.

    % Moreover, if the subspace $E$ is a coordinate subspace, then the orthogonal matrix $R_2$ is a permutation on the standard basis $\{e_1,\dots,e_m\}$. It follows that $u_i=Ue_i$ is nonzero for exactly $n$ vectors $e_i$. But, the segments $[0,u_i]$ are assumed to be nondegenerate, so the only choice we have is that $m=n$. Then $Z$ is the $L_p$-sum of exactly $n$ segments $[0,u_i]$ that span $\mathbb{R}^n$, which by Lemma~\ref{lem:positive_orthant_support_function} is equivalent to saying that $Z$ is the image of the positive orthant of the $l_q$ ball $B_q^n$ under a nonsingular linear transform.
    
    Now, we prove the other direction $(i) \Rightarrow (ii)$. Suppose that statement $(i)$ holds. 
    % Let $ m \geq n \geq 1$. Let $E$ be an $n$-dimensional subspace of $\R^m.$ 
    % Let $w_1,\ldots,w_n$ be an orthonormal basis of $E$.
    % an $m\times m$ orthogonal matrix $R_2$ such that 
    % \begin{equation*}
    %     E=\textup{span}(w_1,\dots,w_n).
    % \end{equation*}
    Using Lemma~\ref{lem:positive_orthant_support_function}, we get that
    \begin{align}
        P_E (B_q^m \cap \R^m_+) = P_E \left( \bigoplus_{i \in [m]}^{(p)} [0,e_i] \right) = \bigoplus_{i \in [m]}^{(p)} [0,P_E e_i].
    \end{align}
    Similarly, we have $P_E B_q^m = \bigoplus_{i \in [m]}^{(p)} [-P_Ee_i , P_E e_i]$.
    Let 
    $
        Z = \bigoplus_{i \in [m]}^{(p)} [0,P_E e_i]. 
        % = \bigoplus_{i \in [m]}^{(p)} [0,R_2P_E e_i] \subset \R^n \times \{0\} .
    $
    Note that $Z$ is an asymmetric $L_p$-zonotope in the subspace $E$. Using the inequality in $(i)$, we obtain
    \begin{equation} \label{eq:project-pf}
       |P_E B_q^m|_n = |Z\oplus_p -Z|_n \leq 2^n |Z|_n = 2^n |P_E (B_q^m \cap \R^m_+)|.
    \end{equation}
    Now, we justify the equality conditions. Note that the verification of the sufficient condition follows from Lemma~\ref{lem:equ-projection}. Assume that $|P_EB_q^m|_n = 2^n|P_E(B_q^{m}\cap \mathbb{R}^m_+)|_n.$
    Using the equality condition in $(i)$, 
    % \begin{equation}
    %     \label{eq:projection-eq-case}
    %     |P_EB_q^m|_n = 2^n|P_E(B_q^m\cap \R^m_+)|_n
    % \end{equation}
    we obtain that 
    % $Z$ is an asymmetric $L_p$-zonotope generated by an $L_p$-sum of $n$ one-sided segments, that is,
    $
        {Z} = \bigoplus_{i\in [n]}^{(p)}[0,v_i]
    $
    for linearly independent vectors $v_1,\ldots,v_n$ in $E$. 
    % is linear image of $B_q^n\cap \R^n_+$ on $\R^n$, that is, there exists a linear map $\Tilde{L}$ on $\R^n$ such that
    % \[
    %     \Tilde{Z}_p = \Tilde{L} (B_q^n \cap \R^n_+).
    % \]
    Thus,
    $$
        \bigoplus_{i \in [m]}^{(p)} [0,P_E e_i]  =  \bigoplus_{j \in [n]}^{(p)}[0,v_j].
    $$
    Using Lemma~\ref{lem:uniqueness-lq}, there exists a partition $U_0 = \{i:P_Ee_i = 0\},U_1,\ldots,U_n$ of $[m]$ such that for any $1 \leq j \leq n$,
    \[
        U_j=\left\{i \in[m]: \alpha_{i j}  v_j= P_Ee_i \text { for some } \alpha_{i j} > 0\right\},
    \]
    with $\sum_{i \in U_j} \alpha_{ij}^p = 1$. For any $ i \not \in U_0$, we denote $u_i = P_Ee_i /\|P_Ee_i\|_2$ and $c_i = \|P_E e_i\|_2^2$. It follows from  $\sum_{i \in U_j} \alpha_{ij}^p = 1$ that $U_1,\ldots,U_n$ are non-empty sets. Without loss of generality, for $1\leq j \leq n$, let $j \in U_j$ and write $U_j $ as follows:
    \[
        U_j =\{i \in [m]: u_i = u_j\}.
    \]
    The proof is complete using Lemma~\ref{lem:sets-projection}. \qedhere
\end{proof}
From equations \eqref{eq:volume_lq_ball_projection} and \eqref{eq:volume_orthant_lq_ball_projection}
in the preceding proof, we deduce the following corollary. 
\begin{corollary} \label{Cor:equality-zp-lq-vol}
    Let $1\leq p \leq \infty$. If $Z$ is a full-dimensional asymmetric $L_p$-zonotope that is the $L_p$-sum of $m\geq n$ segments that each contain the origin, then there exists a subspace $E\subset \mathbb{R}^m$ of dimension $n$, and a positive constant $C$ that depends on $Z$ such that 
    \begin{align*}
        |P_EB_q^m|_n=C|Z\oplus_p-Z|\qquad \textup{and}\qquad |P_E(B_q^m\cap \mathbb{R}^m_+)|_n=C|Z|.
    \end{align*}
\end{corollary}

\begin{remark}
    Theorem~\ref{thm: equivalent-RS-Proj} also shows, by a standard approximation argument, that Conjecture~\ref{conj:projection-cap-conj} implies the inequality in Conjecture~\ref{conj:Rogers-Shephard-ineq-p-zonoids} for asymmetric $L_p$-zonoids. However, it is not clear whether the equality characterization also extends to this infinite-dimensional setting.
\end{remark}

Next, we present a particular case of Conjecture~\ref{conj:Rogers-Shephard-ineq-p-zonoids} when $Z$ is the $L_p$-sum of $n$ segments, and we characterize the conditions for equality.
% , and it will help us to characterize the equality conditions.
\begin{prop} \label{prop: RS-n-seg}
    Let $1< p \leq\infty$. For any asymmetric $L_p$-zonotope $Z$ in $\R^n$ which is the $L_p$-sum of $n$ segments containing the origin, we have 
    \[
        |Z \oplus_p -Z | \leq 2^n |Z|,
    \]
    where equality holds if and only if $Z$ is contained in a hyperplane or is an $L_p$-sum of $n$-segments with the origin as an endpoint.
\end{prop}

\begin{proof}
    If $Z$ is not full-dimensional, then the result is trivial, since $|Z| = |Z \oplus_p -Z| = 0$. Thus, we assume that $Z$ is full-dimensional.
    Let $\{e_1,\dots,e_n\}$ denote the canonical basis for $\mathbb{R}^n$. By applying a linear transform, it is enough to prove the result when 
    $
        Z = \bigoplus_{i\in [n]}^{(p)} [-a_ie_i,b_ie_i]
    $
    for some $a_i,b_i \geq 0$. Since 
    $$
        [-a_ie_i,b_ie_i] \oplus_p [-b_ie_i,a_ie_i] = [-\|(a_i,b_i)\|_pe_i,\|(a_i,b_i)\|_pe_i],
    $$
    we have
    $$
        Z \oplus_p -Z = \bigoplus_{i\in [n]}^{(p)} [-\|(a_i,b_i)\|_pe_i,\|(a_i,b_i)\|_pe_i].
    $$
    It follows from Lemma~\ref{lem:positive_orthant_support_function} that $|Z \oplus_p -Z | = \prod_{i\in [n]} \|(a_i,b_i)\|_p |B_q^n|$ where $q$ satisfies $\frac{1}{p} + \frac{1}{q} = 1$. Now, we compute the volume of $Z$. Recall the decomposition proved in \cite[proof of Lemma 23]{MNZ-25}: for any locally anti-blocking convex bodies $K$ and $L$, one has
    \begin{equation}
        \label{eq:decomp}
        |K\oplus_p L| = \sum_{I \subset [n]} |K_I \oplus_p L_I|,
    \end{equation}
    where $K_I = \{x \in K: x_i \geq 0 \text{ for all } i \in I \text{ and } x_i \leq 0 \text{ for all } i \not\in I\}$. Here, the locally anti-blocking convex body is a convex body for which projections onto and intersections with the coordinate subspaces are identical. Using \eqref{eq:decomp} and Lemma~\ref{lem:positive_orthant_support_function}, we obtain
    \begin{align}
        |Z| &= \sum_{I \subset [n]} \left(\prod_{i \in I} a_i \prod_{j \not\in I} b_j |B_q^n\cap\R^n_+|\right)
        = \left(\prod_{i=1}^n (a_i+b_i)\right) |B_q^n\cap\R^n_+|
        \\
        &\geq \left(\prod_{i=1}^n \|(a_i,b_i)\|_p\right) |B_q^n\cap\R^n_+| = \left(\prod_{i=1}^n \|(a_i,b_i)\|_p\right) \frac{|B_q^n|}{2^n} = \frac{|Z \oplus_p -Z |}{2^n}. \label{eq: RS-n-seg}
    \end{align}
    
    Now, we characterize the equality condition. Using Lemma~\ref{lem:equ-lp-ball}, it is enough to prove only the necessary direction for full-dimensional $Z$.
    % First, if $Z$ is the linear image of $B_q^n\cap\R^n_+$, then there exists a non-singular linear transform $T:\mathbb{R}^{n}\rightarrow\mathbb{R}^{n}$ such that $TZ=B_q^{n}\cap\mathbb{R}^{n}_+$. From Lemma~\ref{lem:positive_orthant_support_function}, we have
    % % \begin{equation*}
    %     $TZ=\bigoplus_{i \in [n]}^{(p)} [0,e_i],$
    % % \end{equation*}
    % and that 
    % \begin{equation*}
    %     TZ\oplus_p-TZ=\bigoplus_{i \in [n]}^{(p)}[-e_i,e_i]=B_q^{n}.
    % \end{equation*}
    % It follows immediately that 
    % \begin{equation*}
    %     |TZ\oplus_p -TZ|=2^n|TZ|.
    % \end{equation*}
    % Using the fact that $TZ\oplus_p-TZ=T(Z\oplus_p-Z)$, we have
    % \begin{equation*}
    %     |Z\oplus_p-Z|=\frac{1}{|\textup{det}(T)|}|TZ\oplus_p-TZ|=\frac{2^n}{|\textup{det}(T)|}|TZ|=2^n|Z|,
    % \end{equation*}
    % which is exactly what we need.
    % Conversely, a
    Assume that $|Z \oplus_p -Z | = 2^n |Z|$. Thus, we have an equality in \eqref{eq: RS-n-seg}: for any $i \in [n]$, $\|(a_i,b_i)\|_1 = \|(a_i,b_i)\|_p$. Therefore, for every $i\in [n]$, $a_i = 0$ or $b_i = 0$. Moreover, since $Z$ is full-dimensional, $a_i$ and $b_i$ can not both be zero. Hence, for each $i\in [n]$, exactly one of $a_i$ and $b_i$ is nonzero, which completes the proof.
\end{proof}

    \subsection{The case $p=2$} \label{sec:p=2}
In the first part of this section, we will prove Conjecture~\ref{conj:projection-cap-conj} using the reverse Brascamp--Lieb inequality proved by Barthe \cite{B-98, B-04,BKX-23}, and using the equality characterization from the paper of Boroczky, Kalantzopoulos and Xi \cite{BKX-23}. We first recall the definition of independent subspaces introduced by Valdimarsson \cite{V-08}. Let $E_1,\ldots,E_k$ be proper subspaces of $\R^n$. For any $\varepsilon \in \{0,1\}^k$, denote $$F_\varepsilon = \bigcap_{i=1}^k E_i ^{\varepsilon_i},$$ where $E_i^{0} = E_i$ and $E_i^{1} = E_i ^\perp$ for any $i =1,\ldots,k$. The subspaces $F_\varepsilon$ with positive dimension are called independent subspaces. We denote $F_{dep}$ the orthogonal component of $\oplus_{\varepsilon \in J} F_\varepsilon$. We now state a part of \cite[Theorem 3 and Theorem 4]{BKX-23}.
\begin{theorem}[\cite{BKX-23}]
For any non-trivial linear subspaces $E_1, \ldots,E_k$ of $\R^n$, for positive real numbers $ c_1,\ldots,c_k >0$ satisfying $ \sum_{i=1}^m c_i P_{E_i} =I_n $, and for any non-negative $f_i \in L_1 (E_i)$, we have 
\begin{equation}
    \int_{\mathbb{R}^n} \sup _{x=\sum_{i=1}^k c_i x_i, \, x_i \in E_i} \prod_{i=1}^k f_i\left(x_i\right)^{c_i} d x \geq \prod_{i=1}^k\left(\int_{E_i} f_i\right)^{c_i} .
    \label{eq:reverse-BL}
\end{equation}
If $F_{dep} \neq \R^n$, then let $F_1,\ldots,F_l$ be independent subspaces, and if $F_{dep}= \R^n$, then set $l = 1$ and $F_1 = \{0\}$. If equality holds in \eqref{eq:reverse-BL} for $f_i$ with $\int_{E_i} f_i(x) dx >0$ for any $i =1,\ldots,k$, then
\begin{equation}
    \label{eq:equality-condition-BL}
    f_i(x) = \theta_i e^{-\langle AP_{F_{dep}} x, P_{F_{dep}} x -b_i \rangle} \prod_{F_j \subset E_i} h_j (P_{F_j}(x-w_i)), \quad \text{ a.e. } x\in E_i,
\end{equation}
where 
\begin{itemize}
    \item $\theta_i >0, b_i \in E_i \cap F_{dep}$ for $w_i \in E_i $ for any $ i  = 1,\ldots,k$,
    \item $h_j \in L_1 (F_j)$ is non-negative for $j=1,\ldots,l$.
\end{itemize}
\end{theorem}
\begingroup
\renewcommand{\MainTheoremExtra}{%
: if $m\ge n \geq 1$, and if $E$ is an $n$-dimensional subspace of $\mathbb{R}^m$, then
        \begin{equation} \label{eq:projection-cap-conj}
            | P_E B_2^m|_n\leq 2^n|P_E(B_2^m \cap \mathbb{R}^m_+)|_n .
        \end{equation}
    Moreover, equality holds if and only if there exists a partition  $U_0 , U_1, \dots, U_n $ of $[m]$ such that $U_j \neq \emptyset$ for every $ 1 \leq j \leq n$, $U_0 = \{ i\in [m] :P_E e_i = 0 \}$, and 
    % there is a partition of $\{1,\dots,m\}$ into $n$ non empty subsets $U_1,\dots,U_n$, and 
    for every $1\le j\le n$, there exist unit vectors $w_j\in\sum_{i\in U_j}\R_{>0} e_i$ such that $E=\textup{span}(w_1,\dots,w_n)$.
}
\ProjectionCaptwo*
\endgroup
% \ProjectionCaptwo*
% \begin{theorem}
%     \label{thm:projection-cap-conj-2}
%     Conjecture~\ref{conj:projection-cap-conj} holds for $ q= 2$.
% \end{theorem}

\begin{proof}
%{\color{purple} MM: We should also mention somewhere (maybe in the introduction) that this method comes from Franck Barthe, and that while Barthe's result can be solved by using sections, we cannot do the same for this result.} 
    % Using Theorem~\ref{thm: equivalent-RS-Proj} to prove the inequality, it is enough to check that if $m\ge n$ and $E$ is an $n$-dimensional subspace $\mathbb{R}^m$, then
    % \begin{equation}
    %     |P_E(B_2^m\cap \mathbb{R}^m_+)|_n \ge \frac{1}{2^n}\,|P_E B_2^m|_n.
    % \end{equation}
    If $K$ is an arbitrary compact convex subset of $E$ that contains the origin, then we have the formula
    \begin{equation}\label{eq:volume_subspace_formula}
        |K|_n=\frac{1}{\Gamma\left(1+\frac{n}{q}\right)}\int_{E}e^{-\|x\|_K^q}dx.
    \end{equation}  
    We use \eqref{eq:volume_subspace_formula} with $K=P_E(B_q^m\cap \mathbb{R}^m_+)$. Let $x\in E$ and compute
    \begin{equation*}
        \|x\|_{P_E(B_q^m\cap \mathbb{R}^m_+)}=\inf_{x=P_Ez}\|z\|_{B_q^m\cap \mathbb{R}^m_+}=\inf_{\substack{x=P_Ez\\z\in \mathbb{R}^m_+}}\left(\sum_{i=1}^{m}|z_i|^q\right)^{\frac{1}{q}}. 
    \end{equation*}
    Let $\{e_1,\dots,e_m\}$ denote the canonical basis for $\mathbb{R}^m$ and set $U_0 = \{ i: P_Ee_i = 0 \}$.
    % For any $i \not \in U_0$, we denote $c_i:=\|P_Ee_i\|_2^2 >0$ and $u_i:=P_Ee_i/\|P_Ee_i\|_2$. 
    For any $i \in [m]$, we denote $c_i:=\|P_Ee_i\|_2^2$ and
    \[
        u_i := \begin{cases}
            P_Ee_i/\|P_Ee_i\|_2 & i \not\in U_0,
            \\
            0 &i \in U_0.
        \end{cases}
    \]
    Then, we express the orthogonal projection $P_E$ onto $E$ as
    \begin{equation*}
     % P_E=\sum_{i=1}^{m}e_i\otimes P_Ee_i=\sum_{i=1}^{m}\sqrt{c_i}e_i\otimes u_i\quad   
     P_E=\sum_{i=1}^{m}P_Ee_i\otimes P_Ee_i=\sum_{i=1}^{m} c_iu_i\otimes u_i = \sum_{i=1}^{m} c_iP_{E_i} ,
    \end{equation*}
    where $(a \otimes b)(x) =\langle x, b\rangle a$ is a rank-one operator and $E_i := \R u_i$. 
    Taking the trace, we deduce that $\sum_{i=1}^{m}c_i=n.$
    % \begin{equation}\label{eq_trace_equals_n}
    %     \sum_{i=1}^{m}c_i=n.
    % \end{equation}
    For the following computation, we will express an arbitrary $z\in \mathbb{R}^m$ as $z=\sum_{i=1}^{m}z_ie_i$. With this notation, we have $P_Ez=\sum_{i=1}^m\sqrt{c_i}z_iu_i$. Then, we compute
    \begin{equation*}
        \begin{split}
            |P_E(B_q^m\cap \mathbb{R}^m_+)|_n&=\frac{1}{\Gamma\left(1+\frac{n}{q}\right)}\int_Ee^{-\inf\{\sum_{i=1}^{m}|z_i|^q:\ z\in \mathbb{R}^m_+,\ x=P_Ez\}}dx\\
            &=\frac{1}{\Gamma\left(1+\frac{n}{q}\right)}\int_E\sup\left\{\prod_{i=1}^me^{-z_i^q}:z\in\mathbb{R}^m_+,x=P_Ez\right\}dx\\
            &=\frac{1}{\Gamma\left(1+\frac{n}{q}\right)}\int_E\sup\left\{\prod_{i=1}^me^{-z_i^q}: z_i\geq0,x=\sum_{i=1}^m\sqrt{c_i}z_iu_i\right\}dx
            \\
            &=\frac{1}{\Gamma\left(1+\frac{n}{q}\right)}\int_E\sup\left\{\prod_{i\not \in U_0}e^{-z_i^q}: z_i\geq0,x=\sum_{i=1}^m\sqrt{c_i}z_iu_i\right\}dx.
        \end{split}
    \end{equation*}
    With the change of variables $t_i:=z_i/\sqrt{c_i}$ for any $i \not \in U_0$, the last integral in the above computation becomes
    \begin{equation*}
      %  \frac{1}{\Gamma\left(1+\frac{n}{q}\right)}
      \int_E\sup\left\{\prod_{i\not \in U_0}\left(e^{-c_i^{\frac{q}{2}-1}t_i^q}\right)^{c_i}:t_i\geq0,x=\sum_{i=1}^mc_it_iu_i\right\}dx
      .
    \end{equation*}
    Now, for each $i$, define the function $f_i:E_i\rightarrow [0,+\infty)$ by
    \begin{equation*}
        f_i(t_iu_i):=e^{-c_i^{\frac{q}{2}-1}t_i^q}\mathbbm{1}_{[0,+\infty)}(t_i).
    \end{equation*}
    From the definition of $f_i$ and the reverse Brascamp--Lieb inequality \eqref{eq:reverse-BL}, we get the bound 
    \begin{equation*}
        %\frac{1}{\Gamma\left(1+\frac{n}{q}\right)}
        \int_{E}\sup_{x= 
        \sum_{i=1}^m c_it_iu_i}\prod_{i\not \in U_0}f_i^{c_i}(t_iu_i) dx\geq 
        %\frac{1}{\Gamma\left(1+\frac{n}{q}\right)}
        \prod_{i\not \in U_0}\left(\int_{0}^{+\infty}f_i(t_iu_i)dt_i\right)^{c_i}.
    \end{equation*}
    To estimate the product on the right-hand side of the above inequality, set $\alpha_i:=c_i^{1/2-1/q}$, and make the change of variables $x_i:=\alpha_it_i$ to get
    \begin{equation*}
        \int_{0}^{+\infty}f_i(t_iu_i)dt_i=\int_{0}^{+\infty}e^{-(\alpha_it_i)^q}dt_i=\frac{1}{\alpha_i}\int_{0}^{+\infty}e^{-x_i^q}dx_i.
    \end{equation*}
    Making one more change of variable $y_i:=x_i^q$, we compute
    \begin{equation*}
        \frac{1}{\alpha_i}\int_{0}^{+\infty}e^{-x_i^q}dx_i=\frac{1}{\alpha_i}\cdot \frac{1}{q}\cdot\Gamma\left(\frac{1}{q}\right)=\frac{1}{\alpha_i}\cdot\Gamma\left(1+\frac{1}{q}\right).
    \end{equation*}
    Using the fact that $\sum_{i\not \in U_0}c_i=n$, we find that
    \begin{equation*}
        \frac{1}{\Gamma\left(1+\frac{n}{q}\right)}\prod_{i\not \in U_0}
        \left(\int_{0}^{+\infty}f_i(t_iu_i)dt_i\right)^{c_i}
        = 
        \frac{\Gamma\left(1+\frac{1}{q}\right)^n}{\Gamma\left(1+\frac{n}{q}\right)}\cdot\frac{1}{\left({\prod}_{i\not \in U_0}c_i^{c_i}\right)^{\frac{1}{2}-\frac{1}{q}}}.
    \end{equation*}
    Now, the condition $c_i=\|P_Ee_i\|_2^2$ implies that $c_i\leq 1$, so that if $q \geq 2$, then
    \begin{equation*}
        \left(\prod_{i\not \in U_0}c_i^{c_i}\right)^{\frac{1}{2}-\frac{1}{q}}\leq 1.
    \end{equation*}
    Putting everything together, we conclude that
    \begin{equation}
    \label{eq:proj-cap-Barthe}
        |P_E(B_q^m\cap \mathbb{R}^m_+)|_n\geq \frac{\Gamma\left(1+\frac{1}{q}\right)^n}{\Gamma\left(1+\frac{n}{q}\right)}=\frac{1}{2^n}|B_q^n|_n,
    \end{equation}
    which is the desired bound when $q =2$.\\

Now, we verify the equality conditions. First, the sufficient condition follows from Lemma~\ref{lem:equ-projection} and it remains to prove the necessary condition holds. Assume that 
\begin{equation}\label{eq:proj-cap-q2}
        |P_E(B_2^m\cap \mathbb{R}^m_+)|= \frac{1}{2^n}\,|P_E B_2^m|.
    \end{equation}
Recall that we want to show that there is a partition of $[m] \setminus U_0$ into $n$ nonempty subsets $U_1,\dots, U_n$, and for every $1\le j\le n$, there exists a unit vector $w_j\in\sum_{i\in U_j}\R_{ \{>0\} } e_i$ such that $E=\textup{span}(w_1,\dots,w_n)$.

First, notice that there is equality in the reverse Brascamp--Lieb inequality for the functions $f_i(tu_i)=e^{-t^2}\mathbbm{1}_{[0,+\infty)}(t)$. Fix ${i\not \in U_0}$. Since $f_i$ can be written as the form \eqref{eq:equality-condition-BL}, there exists $j_i$ such that $F_{j_i} = E_i$. It follows from the definition of $F_{j_i}$ that for every $k \not \in U_0$, either $u_i = \pm u_k$ or $u_i \perp u_k$. Since $i \not\in U_0$ was arbitrary, it follows that for any $i,k \not \in U_0$, either $u_i = \pm u_k$ or $u_i \perp u_k$.
Observe that $\{u_i: i\not \in U_0\}$ spans $E$, so we get that $F_{dep} = \{0\}$ and so for any $i\not \in U_0$, we have
\[
    f_i(tu_i) = \theta_i h_{j_i} (tu_i-w_i) = e^{-t^2}\mathbbm{1}_{[0,+\infty)}(t) \quad \text{ for a.e. } t \in \R,
\]
for some $w_i \in E_i,\theta_i >0$ and $h_{j_i} \in L_1(E_i)$.
Without loss of generality, after changing indices, we may assume that $\{u_{k} : k = 1,\ldots,n\}$ forms an orthonormal basis, and for each $1\leq k \leq n$, we introduce the set $U_k = \{ i : u_i = \pm u_{k} \}$. Note that $(U_k)_{k=1}^n$ forms a partition of $[m] \setminus U_0$ and for any $i \in U_k$, $F_{j_i} = E_i = E_{k}$. 
% We shall prove below that, in fact, $u_{l_k}=u_i$, for all $i\in U_k$. 
% For any $i \in U_k$, observe that
Thus,
\[
    \theta_i h_{j_{k}} (tu_i-w_i)  = f_i(tu_i) =e^{-t^2}\mathbbm{1}_{[0,+\infty)}(t)= f_{k} (tu_{k}) =  \theta_{k} h_{j_{k}} (tu_{k}-w_{k}).
\]
Then, in order for the support of $h_{j_k}$ to match up with the support of $\mathbbm{1}_{[0,+\infty)}$, we must have that $u_i = u_{k}$. The proof is complete using Lemma~\ref{lem:sets-projection}. \qedhere

\end{proof}

% where we use Lemma~\ref{lem:positive_orthant_support_function} in the third equality.
% The following corollary can be proved using the same argument as above.
% \begin{corollary}
%     Let $\{e_1,\dots,e_m\}$ denote the canonical basis for $\mathbb{R}^m$.  For any $q>1$, and  for any $E=\textup{span}(w_j: 1\le j\le n)$, with $w_j\in\sum_{i\in U_j}\R_{\{>0\}}e_i$, where $U_1,\dots, U_n \subset \{1,\dots,m\}$ are non-empty, pairwise-disjoint sets, we have $|P_E(B_q^m \cap \R_+^m)|_n = \frac{1}{2^n}\,|P_E B_q^m|_n$.
% \end{corollary}

\begin{remark}
    We emphasize that this approach applies only for $q =2$, and in particular, it does not yield a proof of Conjecture~\ref{conj:projection-cap-conj} beyond this case. The obstruction is that \eqref{eq:proj-cap-Barthe} requires that $q\ge 2$, and in this range Mathieu Meyer and Pajor proved \cite{MP-88} that $|P_E B_q^m| \geq |B_q^n|$.
\end{remark}
Combining Theorem~\ref{thm: equivalent-RS-Proj} with Theorem~\ref{thm:projection-cap-conj-2}, we have a proof of Conjecture~\ref{conj:Rogers-Shephard-ineq-p-zonoids} in the class of asymmetric $L_2$-zonotopes.
\begin{corollary}
    Let $Z \subset \R^n$ be an asymmetric $L_2$-zonotope. Then
        \begin{equation}
            |Z \oplus_2 -Z| \leq 2^n|Z|.
        \end{equation}
    Equality holds if and only if $Z$ is contained in a hyperplane or is the $L_2$-sum of $n$ segments with the origin as an endpoint.
\end{corollary}
Theorem~\ref{thm: equivalent-RS-Proj} only yields the equality characterization for Conjecture~\ref{conj:Rogers-Shephard-ineq-p-zonoids} for asymmetric $L_p$-zonotopes, rather than for any asymmetric $L_p$-zonoids. Therefore, as mentioned in the introduction, we give an alternative proof of the inequality below in order to overcome this obstacle.
Recall that a Borel measure $\mu$ on $\s^{n-1}$ is isotropic if 
\begin{align} \label{eq:isotropic_zonoid_support_function_2}
        \|x\|_2^2=\int_{\s^{n-1}}\langle u,x\rangle^2 d\mu(u)
\end{align}
for every $x\in \mathbb{R}^n$. We refer to the paper by Giannopoulos and V. Milman \cite{GM-2000} for an overview of isotropic measures.
\begin{lemma}\label{lem:isotropic_zonoid}
    For every full-dimensional asymmetric $L_2$-zonoid $Z\subset\mathbb{R}^n$ there exist $T\in \textup{GL}_n(\mathbb{R})$ and an isotropic Borel measure $\mu_{TZ}$ on $\s^{n-1}$ such that
    \begin{align}
        h_{TZ}^2(x)=\int_{\s^{n-1}}\langle x,u\rangle_+^2 d\mu_{TZ} (u).
    \end{align}
    % and moreover that 
    % \begin{align} \label{eq:isotropic_zonoid_support_function_2}
    %     \|x\|_2^2=\int_{\s^{n-1}}\langle u,x\rangle^2 d\mu_T(u)
    % \end{align}
    % holds for each $x\in \mathbb{R}^n$. 
\end{lemma}

\begin{proof}
    % There exist a positive measure $\nu$ on $S^{n-1}$ such that
    % \[
    %     h_{Z}(x)=\left(\int_{S^{n-1}}\langle x,u\rangle_+^pd\nu(u)\right)^\frac{1}{p}.
    % \]
    % Thus,
    % \begin{align}
    %     h_{Z \oplus_p -Z}^p(x)
    %     &= h_{Z}^p(x) + h_{-Z}^p(x) = h_{Z}^p(x) + h_{Z}^p(-x)
    %     \\
    %     &=\int_{S^{n-1}}\langle x,u\rangle_+^p \,d\nu(u) +\int_{S^{n-1}}\langle -x,u\rangle_+^p\,d\nu(u)
    %     \\
    %     &=\int_{S^{n-1}}|\langle x,u\rangle|^p \, d\nu(u).
    % \end{align}
    Since $Z \oplus_2 -Z$ is origin-symmetric, its support function $h_{Z \oplus_2 -Z} (x)$ is a norm. It was proved Lewis \cite{L-78} and also in \cite[Theorem 8.2]{LYZ-05} by Lutwak, Yang, and Zhang that there exists $ T \in \textup{GL}_n(\R)$ and a Borel measure $\mu$ satisfying 
    \begin{align}
    \label{eq:measure-TzpTz}
        h_{T(Z \oplus_2 -Z)}^2 (x) =h_{Z \oplus_2 -Z}^2 (T^* x) = \int_{\s^{n-1}} | \langle x,u \rangle |^2 \,d \mu (u) = \|x\|_2^2.
    \end{align}
    By definition, there exists a Borel measure $\nu$ on $\s^{n-1}$ such that
    \[
        h_{Z}(x)=\left(\int_{\s^{n-1}}\langle x,u\rangle_+^2 d\nu(u)\right)^\frac{1}{2}.
    \]
    Therefore,
    \begin{align}
        \label{eq:support-tz}
        h_{TZ}^2 (x) = h_{Z}^2(T^*x) = \int_{\s^{n-1}} \langle T^*x,u \rangle_+^2 \, d\nu(u)  = \int_{\s^{n-1}} \langle x,Tu \rangle_+^2 \, d\nu(u).
    \end{align} 
    Define a measure $\mu_{TZ}$ to be the pushforward of a measure $\nu_{TZ}$ on $\s^{n-1}$ given by $u \mapsto \frac{Tu}{\|Tu\|_2}$ where $d \nu_{TZ} (u)= \|Tu\|_2^2 d \nu(u)$. Then,
    \begin{equation}
        \label{eq:measure-Tz}
        \int_{\s^{n-1}} f(u) d\mu_{TZ}(u) =\int_{\s^{n-1}} f\left(\frac{Tu}{\|Tu\|_2}\right) \, d\nu_{TZ}(u) = \int_{\s^{n-1}} f\left(\frac{Tu}{\|Tu\|_2}\right) \|Tu\|_2^2\, d\nu(u).
    \end{equation}
    Substituting $f = \langle x, \cdot \rangle_+^2$ in \eqref{eq:measure-Tz} together with \eqref{eq:support-tz}, we obtain
    \[
        h_{TZ}^2 (x) = \int_{\s^{n-1}} \langle x, u \rangle_+^2 d\mu_{TZ}(u).
    \]
    It follows from \eqref{eq:measure-TzpTz} that
    \begin{align}
        \|x\|_2^2 = \|x\|_{TZ\oplus_2 -TZ}^2
        &= h_{TZ}^2(x) + h_{-TZ}^2(x) = h_{TZ}^2(x) + h_{TZ}^2(-x)
        \\
        &=\int_{\s^{n-1}}\langle x,u\rangle_+^2 \,d\mu_{TZ}(u) +\int_{\s^{n-1}}\langle -x,u\rangle_+^2\,d\mu_{TZ}(u)
        \\
        &=\int_{\s^{n-1}}|\langle x,u\rangle|^2 \, d\mu_{TZ}(u). \qedhere
    \end{align}
\end{proof}

We use the continuous version of the reverse Brascamp-Lieb inequality  proved by Barthe in \cite[Theorem 1 and Theorem 2]{B-04} presented in the following theorem. We recall that a family of functions $(f_u)_{u \in \s^{n-1}}$ satisfies hypothesis (H) if it satisfies the following two conditions:
\begin{itemize}[wide =8pt, labelwidth= .4cm]
    \item There exists a continuous function $F:\s^{n-1}\times \R \to (0,+\infty)$ 
    % and two functions $a,b$ on $\s^{n-1}$ with $a<b$ pointwise and $a$ (respectively $b$) is either real-valued continuous or constant with value in $\{-\infty,\infty\}$ 
    such that for all $(u,t)\in \s^{n-1} \times \R$,
    \[
        f_u(x) = \mathbbm{1}_{[0,+\infty)}(x) F(u,x).
    \]
    \item There exists a function $\mathcal{F} \in L_1(\R) \cap L_\infty (\R)$ such that for all $u \in \s^{n-1}$, one has $ f_u \leq \mathcal{F}.$
\end{itemize}

\begin{theorem}[\cite{B-04}] \label{thm:cont-BL}
    Let $\mu$ be an isotropic Borel measure on $\s^{n-1}$ and for $u\in \s^{n-1}$ let $f_{u}:\mathbb{R}\rightarrow \mathbb{R}_+$ be such that $(f_{u})_{u\in S^{n-1}}$ satisfies the hypothesis (H). If $h:\mathbb{R}^n\rightarrow \mathbb{R}$ is a measurable function such that for every integrable function $g$ on the sphere
    \begin{align*}
        h\left(\int_{S^{n-1}}ug(u)d\mu(u)\right)\geq \exp\left(\int_{S^{n-1}}\log f_u(g(u))d\mu(u)\right),
    \end{align*}
    then 
    \begin{align*}
        \int_{\mathbb{R}^n}h(x)dx\geq \exp\left(\int_{S^{n-1}}\log\left(\int_{\mathbb{R}}f_u(x)dx\right)d\mu(u)\right).
    \end{align*}
    Moreover, if equality holds and none of the $f_u$ are Gaussian, then there exists an orthonormal basis $(u_1,\dots,u_n)$ of $\mathbb{R}^n$ such that $\textup{supp}(\mu)\subset \{\pm u_1,\dots,\pm u_n\}$.
\end{theorem}
We note that the logarithmic terms take values in $[-\infty,+\infty)$ and that their integral makes sense also in $[-\infty,+\infty)$, since by the hypothesis (H) their positive parts converge.
\begingroup
\renewcommand{\MainTheoremExtra}{%
: Let $Z \subset \R^n$ be an asymmetric $L_2$-zonoid. Then
        \begin{equation}\label{eq:Rogers-Shephard-ineq-p-zonoids-conj}
            |Z \oplus_2 -Z| \leq 2^n|Z|.
        \end{equation}
        Equality holds if and only if $Z$ is contained in a hyperplane or is the $L_2$-sum of $n$ segments with the origin as an endpoint
}
\RSInqTwo*
\endgroup
% \RSInqTwo*
% \begin{theorem}
% Let $Z\subset \mathbb{R}^n$ be a full-dimensional $L_2$-zonoid. Then
%     \begin{equation*}
%         |Z\oplus_2(-Z)|\leq 2^n|Z|.
%     \end{equation*}
%     Equality holds if and only if $Z$ is the the $L_2$-sum of $n$ linearly independent segments that each have the origin as an endpoint. 
% \end{theorem}

\begin{proof}
    % Let $Z\subset\mathbb{R}^n$ be an asymmetric $L_2$-zonoid. It suffices to prove that
    % \begin{equation*}
    %     |Z\oplus_2(-Z)|\leq 2^n|Z|,
    % \end{equation*}
    % with equality holds if and only if $Z$ is the $L_2$-sum of $n$ linearly independent segments that each have the origin as an endpoint. 
    
    Using Lemma~\ref{lem:isotropic_zonoid}, we may assume that the generating measure $\mu$ is an isotropic Borel measure. That is,
     \begin{align}\label{eq:isotropic_zonoid_support_function}
        h_{Z}^2(x)=\int_{\s^{n-1}}\langle x,u\rangle_+^2 d\mu (u) \quad \text { and } \quad  \|x\|_2^2=\int_{\s^{n-1}}\langle u,x\rangle^2 d\mu(u).
    \end{align}
    We will show that if $x\in \mathbb{R}^n$, then
    \begin{equation}
        \|x\|_{Z}^2\leq \inf\left\{\int_{\s^{n-1}}g(u)^2d\mu(u):x=\int_{\s^{n-1}}ug(u)d\mu (u),\hspace{0.5mm} g\in L^1(\s^{n-1}),\hspace{0.5mm} g\geq 0\right\}.
    \end{equation}
    If $g\in L^1(\s^{n-1})$ is a nonnegative real-valued function for which 
    \begin{equation*}
        x=\int_{\s^{n-1}}ug(u)d\mu(u),
    \end{equation*}
    then for $y\in \mathbb{R}^n$ we use H\"{o}lder's inequality to get
    \begin{align*}
        \langle x,y\rangle 
        &=\int_{\s^{n-1}}\langle u,y\rangle g(u)d\mu(u)\leq \int_{\s^{n-1}}\langle u,y\rangle_+g(u)d\mu(u)
        \\
        &\leq \left(\int_{\s^{n-1}}\langle u,y\rangle_+^2d\mu(u)\right)^{\frac{1}{2}}\left(\int_{\s^{n-1}}g(u)^2d\mu(u)\right)^{\frac{1}{2}}
        \\
        &\leq \|y\|_{Z^\circ}\left(\int_{\s^{n-1}}g(u)^2d\mu(u)\right)^{\frac{1}{2}},
    \end{align*}
    where $K^\circ = \{ x: \langle x,y \rangle \leq 1 \}$ denotes the polar body of $K$.
    It follows that 
    \begin{equation} \label{eq:suff-con-for-BL}
        \|x\|_{Z}=\sup_{\|y\|_{Z^\circ}\leq 1}\langle x,y\rangle \leq \left(\int_{\s^{n-1}}g(u)^2d\mu
        (u)\right)^{\frac{1}{2}}.
    \end{equation}
    Define the functions $f_u:\mathbb{R}\rightarrow \mathbb{R}_+$ by $f_u(x)=\exp(-x^2)\mathbbm{1}_{\mathbb{R}_+}(x)$ and define the measurable function $h:\mathbb{R}^n\rightarrow \mathbb{R}$ by $h(x):=\exp(-\|x\|_{Z}^2)$. It is clear that $(f_u)_{u \in \s^{n-1}}$ satisfies the hypothesis (H), so we only need to verify that for any $g \in L^1(\s^{n-1})$, one has 
    \begin{align*}
        h\left(\int_{\s^{n-1}}ug(u)d\mu(u)\right)\geq \exp\left(\int_{S^{n-1}}
        \log(f_u(g(u)))
        % -g(u)^2 + \log \mathbbm{1}_{\mathbb{R}_+}
        %(g(u)) 
        d\mu(u)\right),
    \end{align*}
    or equivalently,
    \begin{equation} \label{eq:condition-cont-BL}
        \left\|\int_{\s^{n-1}}ug(u)d\mu(u)\right\|_{Z}^2 \leq \int_{\s^{n-1}} \left(g(u)^2 - \log \mathbbm{1}_{\mathbb{R}_+}(g(u))\right) d\mu(u). 
    \end{equation}
    By \eqref{eq:suff-con-for-BL}, the above condition is true for any $ g $ that is $\mu$-almost everywhere positive. In the case where $g\leq0$ on a set of positive measure, the right hand side is $+\infty$, and the inequality is trivial.
    Now, apply the continuous version of the reverse Brascamp-Lieb inequality in Theorem~\ref{thm:cont-BL} with $f_u(x)=\exp(-x^2)\mathbbm{1}_{\mathbb{R}_+}(x)$ to get 
    \begin{align*}
        \Gamma\left(1+\frac{n}{2}\right)|Z|
        &=\int_{\mathbb{R}^n}h(x)dx
        \geq \exp\left(\int_{\s^{n-1}}\log\left(\int_{\mathbb{R}}f_u(x)dx\right)d\mu(u)\right)
        \\
        &=\exp\left(\mu(\s^{n-1})\log \Gamma \left( 1+\frac{1}{2}\right)\right) = \Gamma \left( 1+\frac{1}{2}\right)^n,
    \end{align*}
    where in the last step, we use that $\mu$ is isotropic and so $ \mu(\s^{n-1}) = n$. Since $\frac{\Gamma\left(1+\frac{1}{2}\right)^n}{\Gamma\left(1+\frac{n}{2}\right)}=\frac{1}{2^n}|B_2^n|_n$, we obtain that
    % \begin{equation}
        $|Z| \geq \frac{1}{2^n}|B_2^n|_n.$
    % \end{equation}
    Since $\mu$ is isotropic and satisfies \eqref{eq:isotropic_zonoid_support_function}, we have
    $$
        \|x\|_2^2 = \int_{\s^{n-1}} \langle u,x \rangle^2 d\mu(u) = h_{Z\oplus_2 -Z}^2 (x),
    $$
    and so $Z\oplus_2 -Z = B_2^n$. Thus, $ \|x\|_{Z\oplus_2 -Z}^2 = \|x\|_2^2$ and then
    \begin{align}
        |Z \oplus_2 -Z| = \frac{1}{\Gamma (1+\frac{n}{2})} \int_{\R^n} \exp{(-\|x\|_2^2)} dx = |B_2^n|_n,
    \end{align}
    which verifies the inequality.

    Now, we characterize the equality case. The sufficiency follows from
    % If $Z$ is the $L_2$-sum of $n$ linearly independent segments that each have the origin as an endpoint. By applying the linear transform, we may assume that $ Z = \bigoplus_{i\in[n]}^{(2)}[0,e_i]$ and then the equality follows from 
    Lemma~\ref{lem:positive_orthant_support_function} and Lemma~\ref{lem:equ-lp-ball}.
    For the other direction, using Lemma~\ref{lem:isotropic_zonoid}, we may assume that the generating measure $\mu$ is an isotropic Borel measure. Using the equality characterization of the continuous version of the reverse Brascamp-Lieb inequality in Theorem~\ref{thm:cont-BL}, since $f_u$ is not Gaussian, we get that there exists an orthonormal basis $(u_1,\dots,u_n)$ of $\mathbb{R}^n$ such that $\textup{supp}(\mu)\subset \{\pm u_1,\dots,\pm u_n\}$. Then, $\mu = \sum_{i=1}^n \left( a_i^2 \delta_{-u_i} +b_i^2 \delta_{u_i}\right)$ for some $a_i,b_i \geq 0$.
    Using \eqref{eq:isotropic_zonoid_support_function}, we obtain
    \[
        h_{Z}^2(x)=\int_{\s^{n-1}}\langle x,u\rangle_+^2 d\mu (u) =  \sum_{i=1}^n \left( a_i^2 \langle x,-u_i\rangle_+^2+ b_i^2 \langle x,u_i\rangle_+^2 \right),
    \]
    implying 
    $$
        Z = \bigoplus_{i\in[n]}^{(2)} [0,-a_iu_i] \oplus_2 \bigoplus_{j\in[n]}^{(2)} [0,b_ju_j] = \bigoplus_{i\in[n]}^{(2)} [-a_iu_i,b_iu_i] .
    $$
    We use Proposition~\ref{prop: RS-n-seg} to conclude that $Z$ is the $L_2$-sum of $n$ segments with the origin as an endpoint.
\end{proof}

\subsection{Projections onto hyperplanes} \label{sec:proj-hyp}
We begin by establishing an auxiliary lemma. We first recall the definition of the \emph{Rademacher averages} of a finite sequence. Let $a=\left(a_1, \ldots, a_m\right) \in \mathbb{R}^m$, and let $\left(\sigma_i\right)_{i=1}^m$ be i.i.d.\ Rademacher random variables (symmetric Bernoulli random variables), that is,
$
\mathbb{P}\left(\sigma_i=1\right)=\mathbb{P}\left(\sigma_i=-1\right)=\frac{1}{2} .
$
The Rademacher average of the sequence $a$ is defined to be the expectation
$$
\mathbb{E}_\sigma|\langle a, \sigma\rangle|
=
% \frac{1}{2^m} \sum_{\sigma \in\{-1,1\}^m}\left|\sum_{i=1}^m \sigma_i a_i\right| 
\frac{1}{2^m} \sum_{\sigma \in\{-1,1\}^m}\left|\langle a, \sigma  \rangle\right|.
$$
% We refer to \cite[Chapter~4]{LT-91} for further information about the Rademacher average.
\begin{lemma}
    \label{lem:Weighted-Expectation}
    Let $ m \geq 2$ be an integer and  $a \in \R^m \setminus \{0\}$. Then 
    \begin{equation}
        \label{eq:Weighted-Expectation-up}
        % \frac{1}{2^m} \sum_{\sigma \in \{ -1,1\}^m} | \langle a , \sigma \rangle |
        \mathbb{E}_\sigma|\langle a, \sigma\rangle|
        \leq 
        \frac{1}{2} \left( 
        \|a\|_1
        % \max_{\sigma \in \{ -1,1\}^m}  | \langle a , \sigma \rangle |
        + \min_{\tau \in \{ -1,1\}^m}  | \langle a , \tau \rangle |\right).
    \end{equation}
    Equality holds if and only if there exists $\varepsilon\in\{-1,1\}^{m}$ for which $|\langle a,\varepsilon \rangle|$ is minimized, and either $a$ has at most two nonzero coordinates and the corresponding coordinates of $\varepsilon$ have different signs from each other, or $a$ has at least three nonzero coordinates and over all coordinates of $\varepsilon$ corresponding to nonzero coordinates of $a$ there exists a coordinate $\varepsilon_k$ whose sign is different from the other coordinates, and the bound
    \begin{equation*}
        2|a_k|\geq \sum_{i=1}^{m}|a_i|
    \end{equation*}
    is satisfied.
\end{lemma}

\begin{proof}
    Let $a$ be arbitrary and choose $  \delta\in \{-1,1\}^m$ so that $\delta a\in\mathbb{R}^m_+$. Then, the right and left sides of (\ref{eq:Weighted-Expectation-up}) are unchanged if we take $\delta a$ in place of $a$. Accordingly, it is enough to prove this lemma for any $a\in \mathbb{R}^m_+$, which we will do as follows.

    Choose $\varepsilon\in \{-1,1\}^m$ such that 
    \begin{equation*}
        |\langle a,\varepsilon\rangle|=\min_{\tau\in \{-1,1\}^m}|\langle a,\tau \rangle|.
    \end{equation*}
    %Let $
     %   \varepsilon
        % = (\varepsilon_1,\ldots, \varepsilon_m) 
    %    \in \{-1,1\}^m
    %$
    %be chosen so that $|\langle a, \varepsilon\rangle|$ attains the minimum among all sign choices. 
    % The case $ m =2$ is immediate, since both sides coincide. 
     If every coordinate of $\varepsilon$ has the same sign, then $|\langle a,\varepsilon\rangle|=\|a\|_1$. But then the right side of (\ref{eq:Weighted-Expectation-up}) is exactly $\|a\|_1$. By the triangle inequality, the left side of (\ref{eq:Weighted-Expectation-up}) is bounded above by $\|a\|_1$, which verifies the required bound. 
     
     Now, assume that there exist coordinates of $\varepsilon$ with different signs. Reordering the coordinates if needed, we may assume that there exists $1\leq k\leq m-1$ such that 
    $$
        \varepsilon_1=\ldots=\varepsilon_k = 1, \quad \text { and } \quad \varepsilon_{k+1}= \ldots= \varepsilon_m =-1.
    $$
    Using the triangle inequality and the relation $|\alpha +\beta|+|\alpha-\beta|=2\max\{|\alpha|,|\beta|\}$, which holds for any real numbers $\alpha$ and $\beta$, we find for any $\sigma\in \{-1,1\}^m$ that
    % define
    % $$
    %     \sigma^{\prime}=\left(\sigma_1, \ldots, \sigma_k,-\sigma_{k+1}, \ldots,-\sigma_m\right) .
    % $$
    % This defines a unique pairing of sign vectors. For such a pair $\left(\sigma, \sigma^{\prime}\right)$, 
    \begin{equation}
        \label{eq:Weighted-Expectation-up-1}
        \begin{split}
            | \langle a , \sigma \rangle | + |\langle a , \varepsilon \sigma \rangle | 
            &= 2\max\left\{\left|\sum_{i=1}^{k}a_i\sigma_i\right|, \left|\sum_{i=k+1}^{m}a_i\sigma_i\right|\right\}
            \\
            &\leq 2\max\left\{\sum_{i=1}^{k}a_i,\sum_{i=k+1}^{m}a_i\right\}
            \\
            &=
            \| a \|_1 + \min_{\tau \in \{ -1,1\}^m}  | \langle a , \tau \rangle |.
        \end{split}
    \end{equation}
    Sum both sides of \eqref{eq:Weighted-Expectation-up-1} over all $2^m$ elements of $\sigma\in \{-1,1\}^m$ to get
    \begin{equation*}
        \sum_{\sigma\in \{-1,1\}^m}|\langle a,\sigma\rangle|+\sum_{\sigma\in\{-1,1\}^m}|\langle a,\varepsilon\sigma\rangle|\leq 2^m\left(\|a\|_1+\min_{\tau\in \{-1,1\}^m}|\langle a,\tau\rangle|\right).
    \end{equation*}
    Since both terms on the left side are the same, we have proved the bound (\ref{eq:Weighted-Expectation-up}).

    We now characterize the conditions for equality. Without loss of generality, we assume that each coordinate $a_i$ is positive and that $a$ has at least two nonzero coordinates. Suppose that (\ref{eq:Weighted-Expectation-up}) is equality, and let $\varepsilon\in\{-1,1\}^{m}$ be chosen such that $|\langle a,\varepsilon\rangle|$ is minimized. From the fact that
    \begin{equation*}
        \left|\sum_{i=1}^{m}a_i\right|>\left|-a_1+\sum_{i=2}^{m}a_i\right|,
    \end{equation*}
    we see that not all coordinates of $\varepsilon$ have the same sign. Then by following the above proof, there exists $1\leq k\leq m-1$ such that
    \begin{equation*}
        \max\left\{\left|\sum_{i=1}^{k}a_i\sigma_i\right|, \left|\sum_{i=k+1}^{m}a_i\sigma_i\right|\right\}=\max\left\{\sum_{i=1}^{k}a_i,\sum_{i=k+1}^{m}a_i\right\}
    \end{equation*}
    for each $\sigma\in \{-1,1\}^m$. Taking 
    \begin{equation*}
        \sigma=(\underbrace{1,-1,\dots,-1}_{\textup{k terms}},\underbrace{1,-1,\dots,-1}_{\textup{m-k terms}}),
    \end{equation*}
    we find that
    \begin{equation*}
        \max\left\{\left|a_1-\sum_{i=2}^{k}a_i\right|, \left|a_{k+1}-\sum_{i=k+2}^{m}a_i\right|\right\}=\max\left\{a_1+\sum_{i=2}^{k}a_i,a_{k+1}+\sum_{i=k+2}^{m}a_i\right\}.
    \end{equation*}
    In general, if $\alpha_1\leq \alpha_2$ and $\beta_1\leq \beta_2$ are positive real numbers that satisfy $\max\{\alpha_1,\beta_1\}=\max\{\alpha_2,\beta_2\}$, then either $\alpha_1=\alpha_2\geq \beta_2$, or $\beta_1=\beta_2\geq\alpha_2$. We will apply this principle to the above equality. If 
    \begin{equation*}
        \left|a_1-\sum_{i=2}^{k}a_i\right|=a_1+\sum_{i=2}^{k}a_i\geq a_{k+1}+\sum_{i=k+2}^{m}a_i,
    \end{equation*}
    then we must have $k=1$. It follows that we can take $\varepsilon=(1,-1,\dots,-1)$, and compute 
    \begin{equation*}
        2a_1=a_1+a_1\geq a_1+\sum_{i=2}^{m}a_i=\sum_{i=1}^{m}a_i.
    \end{equation*}
    On the other hand, if
    \begin{equation*}
        \left|a_{k+1}-\sum_{i=k+2}^{m}a_i\right|=a_{k+1}+\sum_{i=k+2}^{m}a_i\geq a_1+\sum_{i=2}^{k}a_i,
    \end{equation*}
    then we deduce that $k=m-1$. Similar to the first case, we can set $\varepsilon =(1,\dots,1,-1)$ and we have the lower bound
    \begin{equation*}
        2a_m\geq\sum_{i=1}^{m-1}a_i.
    \end{equation*}
    In either case, we have verified that the required condition is satisfied. 

    To prove the other direction, suppose that there exists $\varepsilon \in \{-1,1\}^m$ such that $|\langle a,\varepsilon\rangle|$ is minimized, there is a coordinate $\varepsilon_k$ whose sign is different from the signs of the other coordinates of $\varepsilon$, and for which the bound
    \begin{equation*}
        2a_k\geq \sum_{i=1}^{m}a_i
    \end{equation*}
    is satisfied. Following the proof of the bound (\ref{eq:Weighted-Expectation-up}), and rearranging coordinates so that $k=1$, we need to show that
    \begin{equation*}
               \max\left\{a_1, \left|\sum_{i=2}^{m}a_i\sigma_i\right|\right\}=\max\left\{a_1,\sum_{i=2}^{m}a_i\right\}
    \end{equation*}
    is true for each $\sigma\in \{-1,1\}^m$. But, this must be true because by assumption
    \begin{equation*}
        a_1\geq \sum_{i=2}^{m}a_i\geq \left|\sum_{i=2}^{m}a_i\sigma_i\right|.
    \end{equation*}
    This verifies that (\ref{eq:Weighted-Expectation-up}) is an equality, and the proof is complete.
    %{\color{red} Equality attempt:} For the equality case, first observe that for $m=2$, the inequality is an equality since both side are coincide. If $a$ yields the equality, then $\sigma a$ is so for every $\sigma \in \{-1,1\}^m$. Also, if some coordinates of $a$ are zero, then it reduces to the lower dimensional case, hence it suffices to consider the case where all coordinates of $a$ are strictly positive. 
    %We note that if all coordinates of $\varepsilon$ can not have the same sign, otherwise the equality condition will force $a = o$. Thus, $k, m-k \geq 1$.
    %To have an equality, \eqref{eq:Weighted-Expectation-up-1} has to be an equality for every $\sigma \in \{-1,1\}^m$, that is
    %\[
     %   \max \{ |a_1 \sigma_1 + \ldots a_k \sigma_k|, |a_{k+1} \sigma_{k+1} + \ldots + a_m \sigma_m |\}
      %  =
       % \max \{ a_1 + \ldots + a_k,  a_{k+1} + \ldots + a_m \}.
    %\]
    %It follows from algebra that $a_1 + \ldots + a_k = |a_1 \sigma_1 + \ldots a_k \sigma_k| \geq a_{k+1} + \ldots + a_m $ or $a_{k+1} + \ldots + a_m  = |a_{k+1} \sigma_{k+1} + \ldots + a_m \sigma_m | \geq a_1 + \ldots + a_k$. Thus, $k = 1$ or $ m -k =1$.
    % Taking $\sigma=(1,\ldots,1)$ and $\sigma=(-1,\ldots,-1)$, we obtain
    % \[
    %     \sum_{i=1}^k a_i = - \sum_{i=k+1}^m a_i .
    % \]
    % Since $a \in \mathbb{R}^m_+$, both sums are nonnegative, and therefore each must be equal to zero.
\end{proof}

\begin{lemma} \label{lem:Weighted-Expectation-lo-1}
    If $m\geq 1$ is an integer, then
    \begin{equation}
        \label{eq:Weighted-Expectation-lo-1}
        \sum_{\sigma \in \{-1,1\}^m} \left| \sum_{i=1}^m \sigma_i \right| = 2m \binom{m-1}{\lfloor \frac{m}{2} \rfloor}.
    \end{equation}
\end{lemma}

\begin{proof}
    We will prove the lemma by induction on $m$. 
    % that \eqref{eq:Weighted-Expectation-lo-1}.
    % $$
    %     \sum_{\sigma \in \{-1,1\}^m} \left| \sum_{i=1}^m \sigma_i \right| = 2m \binom{m-1}{\lfloor m/2 \rfloor}.
    % $$
    
    If $m=1$, then 
    \begin{equation*}
        \sum_{\sigma\in \{-1,1\}}\left|\sigma\right|=2=2\binom{0}{0},
    \end{equation*}
    as claimed.
    Next, assume by induction that the claim is true for an integer $m\geq 1$. It is straightforward to check that if $i$ is an integer, then the identity
    \begin{equation*}
        \frac{|i+1|+|i-1|}{2}=|i|+\mathbbm{1}_{\{0\}}(i)
    \end{equation*}
    holds. Using this identity, we compute
    \begin{equation*}
        \begin{split}
            \sum_{\sigma\in\{-1,1\}^{m+1}}\left|\sum_{i=1}^{m+1}\sigma_i\right|&=\sum_{\sigma\in \{-1,1\}^m}\left(\left|\sum_{i=1}^{m}\sigma_i+1\right|+\left|\sum_{i=1}^{m}\sigma_i-1\right|\right)\\
            &=2\sum_{\sigma\in\{-1,1\}^{m}}\left|\sum_{i=1}^{m}\sigma_i\right|+2\cdot\#\left\{\sigma\in\{-1,1\}^m:\sum_{i=1}^{m}\sigma_i=0\right\},
        \end{split}
    \end{equation*}
    where $\#A$ denotes the cardinality of the set $A.$
    If $m$ is odd, then for every $\sigma\in\{-1,1\}^m$ the sum $\sum_{i=1}^{m}\sigma_i$ is nonzero. On the other hand, if $m$ is even, then an element $\sigma\in\{-1,1\}^{m}$ satisfies $\sum_{i=1}^{m}\sigma_i=0$ if and only if exactly $m/2$ of the components $\sigma_i$ take on the value $1$. This implies the formula
    \begin{equation*}
        \#\left\{\sigma\in\{-1,1\}^m:\sum_{i=1}^{m}\sigma_i=0\right\}=\mathbbm{1}_{2\mathbb{N}}(m)\cdot \binom{m}{\frac{m}{2}}.
    \end{equation*}
    Putting everything together with the inductive assumption, we have
    \begin{equation*}
        \begin{split}
        \sum_{\sigma\in\{-1,1\}^{m+1}}\left|\sum_{i=1}^{m+1}\sigma_i\right|&=2\left(2m\binom{m-1}{\lfloor \frac{m}{2}\rfloor}+\mathbbm{1}_{2\mathbb{N}}(m)\cdot \binom{m}{\frac{m}{2}}\right)\\
        &=2\left(\mathbbm{1}_{2\mathbb{N}}(m)\cdot \left(2m\binom{m-1}{\frac{m}{2}}+\binom{m}{\frac{m}{2}}\right)+\mathbbm{1}_{2\mathbb{N}}(m+1)\cdot 2m\binom{m-1}{\frac{m-1}{2}}\right)\\
        &=2\left(\mathbbm{1}_{2\mathbb{N}}(m)\cdot (m+1)\binom{m}{\frac{m}{2}}+\mathbbm{1}_{2\mathbb{N}}(m+1)\cdot (m+1)\binom{m}{\frac{m+1}{2}}\right)\\
        &=2(m+1)\binom{m}{\lfloor \frac{m+1}{2}\rfloor},
        \end{split}
    \end{equation*}
    which is what we need.
    %Using the identity for any integer $i$,
    %\(
    %    \frac{|i+1|+|i-1|}{2}=|i|+\mathbf{1}_{\{i=0\}},
    %\)
    %we write
    %\begin{align*}
     %   R_m 
      %  &=\sum_{\sigma \in \{-1,1\}^{m-1}} \left(\left| \sum_{i=1}^{m-1} \sigma_i +1\right|+\left| \sum_{i=1}^{m-1} \sigma_i -1\right|\right) 
       % \\
       % &=2\sum_{\sigma \in\{-1,1\}^{m-1}}\left| \sum_{i=1}^{m-1} \sigma_i \right|
        %+2 \cdot \#\left\{\sigma \in\{-1,1\}^{m-1}: \sum_{i=1}^{m-1} \sigma_i=0\right\}. 
    %\end{align*}
    %If $m-1$ is odd, then the second term vanishes. If $m-1$ is even, then the condition
    %\(\sum_{i=1}^{m-1}\sigma_i=0\) means that $\sigma$ has exactly $(m-1)/2$ entries equal to $1$, and hence the second term equals
    %\(2 \binom{m-1}{(m-1)/2}\).
    %Using the inductive assumption, the claimed formula of $R_m$ follows by a direct computation.
\end{proof}

\begin{lemma} \label{lem:Weighted-Expectation-lo}
    For $a \in \R^m$, set $ a^* := (\| a \|_1 /m, \ldots, \| a \|_1 /m)$. Then,
    \begin{equation}
        \label{eq:Weighted-Expectation-lo}
        \mathbb{E}_\sigma|\langle a, \sigma\rangle|\geq \mathbb{E}_\sigma|\langle a^*, \sigma\rangle| = \frac{\|a\|_1}{2^{m-1}} \binom{m-1}{\lfloor \frac{m}{2} \rfloor}.
    \end{equation}
\end{lemma}

\begin{proof}
    Just as we did in the proof of Lemma~\ref{lem:Weighted-Expectation}, we assume that $a\in\mathbb{R}^m_+$. Set $S:=\|a\|_1/m$, and for $x\in\mathbb{R}^m$ define
    \[
        g(x) = \sum_{\sigma \in \{-1,1\}^m} \left| \sum_{i=1}^m \sigma_i x_i \right|.
    \]
    We want to show that $g(a)\geq g(a^*)$. If $a=a^*$, then we are done, so assume that $a\neq a^*$. Then, there exist indices $i$ and $j$ such that $a_i<S<a_j$. Since the function $g$ is invariant under permutations of coordinates, no generality is lost if we assume that $i=1$ and $j=2$. Next, choose $t\in (0,1)$ such that $S=ta_1+(1-t)a_2$, and write
    \begin{equation*}
        (S,ta_2+(1-t)a_1,a_3,\dots,a_m)=t(a_1,a_2,a_3,\dots,a_m)+(1-t)(a_2,a_1,a_3,\dots,a_m).
    \end{equation*}
    Using the convexity and permutation invariance of $g$, we compute
    \begin{equation*}
        g(S,ta_2+(1-t)a_1,a_3,\dots,a_m)\leq tg(a)+(1-t)g(a_2,a_1,a_3,\dots,a_m)=g(a).
    \end{equation*}
    
    In summary, we have found a vector 
    \begin{equation*}
        \alpha^{1}:=(S,ta_2+(1-t)a_1,a_3,\dots,a_m)
    \end{equation*}
    that satisfies the bound $g(\alpha^1)\leq g(a)$, whose first coordinate is equal to $S$, and $\|\alpha^1\|_1=\|a\|_1$. If $\alpha^1=a^*$, then we are done. Otherwise, repeat the same process to arrive at another vector $\alpha^2$ that satisfies the bound $g(\alpha^2)\leq g(\alpha^1)$, whose first two coordinates are equal to $S$, and $\|\alpha^2\|_1=\|a\|_1$. Repeating this process for at most $m$ iterations, we eventually arrive at the vector $\alpha^k=a^*$ that satisfies the bound $g(\alpha^k)\leq g(\alpha^{k-1})$. We see that
    \begin{equation*}
        g(a)\geq g(\alpha^{1})\geq \dots\geq g(\alpha^{k-1})\geq g(a^*),
    \end{equation*}
    which is what we want. Use Lemma~\ref{eq:Weighted-Expectation-lo-1} to compute
    \begin{equation}
        g(a)\geq g(a^*) = \frac{\|a\|_1}{m} g(1,\dots,1) = 2\|a\|_1 \binom{m-1}{\lfloor \frac{m}{2} \rfloor}. \qedhere
    \end{equation}
    %Without loss of generality, we assume that $a \in \mathbb{R}^m_+$ and set $S = %\|a\|_1 /m$. Define
   % \[
    %    g(x) = \sum_{\sigma \in \{-1,1\}^m} \left| \sum_{i=1}^m \sigma_i x_i \right|.
   % \]
    %If $a \neq a^*$, then, after relabeling the coordinates, we may assume that $ a_1 < S < a_2 .$ Choose $t \in (0,1)$ so that $t a_1 +  (1-t) a_2 = S $.
    %Using convexity of $g$ and its symmetry under permutations of the coordinates, we obtain
    %\begin{align*}
    %    g(S, ta_2 + (1-t)a_1, \ldots,a_m  ) 
    %    &\leq 
    %    t g(a) +(1-t)g(a_2, a_1,  \ldots, a_m) =g(a).
    %\end{align*}
    %Note that this replacement $ a\mapsto (S, ta_2 + (1-t)a_1, \ldots,a_m  )$ preserves the sum of the coordinates. Iterating this step yields 
    % \(
    %     a^*=\left(S,\ldots,S\right),
    % \)
    % and therefore 
    %\begin{align*}
    %    g(a)\geq g(a^*) &= S \sum_{\sigma \in \{-1,1\}^m} \left| \sum_{i=1}^m \sigma_i \right| = 2\|a\|_1 \binom{m-1}{\lfloor \frac{m}{2} \rfloor},
    %\end{align*}
    %where we provide the last computation in Lemma~\ref{eq:Weighted-Expectation-lo-1}.
\end{proof}

We recall that the surface area measure $S_K$ of a nonempty compact convex set $K$ in $\mathbb{R}^m$ is a Borel measure on the unit sphere $\s^{n-1}$ defined as follows: for $E \subset \s^{n-1}, S_K(E)$ equals the volume of the part of the boundary $\partial K$ where normal vectors belong to $E$ ($S_K$ is the pushforward of the $(n-1)$-dimensional Hausdorff measure on $\partial K$ via the Gauss map $\left.\nu_K: \partial K \rightarrow \s^{n-1}\right)$. We recall the \emph{Cauchy formula}. For a convex body $K \subset \R^m$ and for any $ u\in \s^{m-1} $, we have
    \begin{equation}
        \label{eq:Cauchy-polytope-projection-hyperplane-smooth}
        \left|{P}_{u^{\perp}}K\right|= \frac{1}{2} \int_{\s^{m-1}}|\langle u, \xi\rangle| \;{d} \,S_K(\xi).
    \end{equation}
    The volume of the orthogonal projection of a polytope onto a hyperplane can be expressed explicitly as a sum taken over the facets of the polytope. More precisely, let $\mathcal{F}_K$ be the set of facets of $K$ and $n(F)$ be the unit outer normal vector to $F \in \mathcal{F}_K$. Then,
    \begin{equation}
        \label{eq:Cauchy-polytope-projection-hyperplane}
        |P_{u^\perp} K| =\frac{1}{2} \sum_{F \in \mathcal{F}_K} \left|F\right| | \langle u , n(F) \rangle |.
    \end{equation}

\begingroup
\renewcommand{\MainTheoremExtraTwo}{%
If $K$ is a convex polytope whose unit facet normals in $\mathbb{R}^{n+1}$ are given by $v_1,\dots,v_N$, choose $\varepsilon\in\{-1,1\}^{n+1}$ so that $u=\varepsilon |u|$, where $|u|$ is the vector with coordinates $|u_j|$, and set $a_i:=(|u_1|(v_i)_{1},\dots,|u_{n+1}|(v_i)_{n+1})$ for each $i$. Then, there is equality in the upper bound if and only if for each $i$ the quantity $|\langle a_i,\varepsilon\rangle|$ is minimized, and either $a_i$ has at most two positive coordinates and the corresponding coordinates of $\varepsilon$ have different signs from each other, or $a_i$ has at least three positive coordinates, and among the corresponding coordinates of $\varepsilon$ there is $\varepsilon_{n_i}$ whose sign is different from the signs of the other such coordinates, and the bound
    \begin{equation*}
        2|u_{n_i}|(v_i)_{n_i}\geq \sum_{j=1}^{n+1}|u_j|(v_i)_j
    \end{equation*}
    is satisfied. The left inequality is sharp.
}
\ProjectionIneqHyperplane*
\endgroup

% \begin{prop} \label{prop:Projection-ineq-hyperplane}
%     For any $u \in \s^n$ and for any 1-unconditional convex body $K \subset \R^{n+1}$, we have
%     \begin{equation}
%         \label{eq:Projection-ineq-hyperplane}
%         \binom{n}{\lfloor \frac{n+1}{2} \rfloor} |P_{u^\perp} ( K \cap \R^{n+1}_+ )| \leq |P_{u^\perp} K| \leq 2^n |P_{u^\perp} ( K \cap \R^{n+1}_+ )|.
%     \end{equation}
%     If $K$ is a convex polytope whose unit facet normals in $\mathbb{R}^{n+1}$ are given by $v_1,\dots,v_N$, choose $\varepsilon\in\{-1,1\}^{n+1}$ so that $u=\varepsilon |u|$, where $|u|$ is the vector with coordinates $|u_j|$, and set $a_i:=(|u_1|(v_i)_{1},\dots,|u_{n+1}|(v_i)_{n+1})$ for each $i$. Then the upper bound is equality if and only if for each $i$ the quantity $|\langle a_i,\varepsilon\rangle|$ is minimized, and either $a_i$ has at most two positive coordinates and the corresponding coordinates of $\varepsilon$ have different signs from each other, or $a_i$ has at least three positive coordinates, and among the corresponding coordinates of $\varepsilon$ there is $\varepsilon_{n_i}$ whose sign is different from the signs of the other such coordinates, and the bound
%     \begin{equation*}
%         2|u_{n_i}|(v_i)_{n_i}\geq \sum_{j=1}^{n+1}|u_j|(v_i)_j
%     \end{equation*}
%     is satisfied.
    
% \end{prop}

\begin{proof}
    We will prove the result for $1$-unconditional convex polytopes, and the general case will follow by approximation. Suppose that $K\subset\mathbb{R}^{n+1}$ is an unconditional convex polytope. Then, $v$ is a unit facet normal of $K$ if and only if  $|v|$ is a unit facet normal of $K$. It follows that $K$ is completely determined by its unit facet normals $v_1,\dots,v_N$ that lie in $\mathbb{R}^{n+1}_+$, and their reflections into the other orthants of $\mathbb{R}^{n+1}$.

    For any $1\le i\le N $, we denote by $F_i = \{x \in K\cap \R^{n+1}_+: \langle x,v_i \rangle = h_K(v_i)\}$ the intersection of the facets of $K$ with the positive orthant. An arbitrary direction of $\mathbb{S}^n$ is given by $\varepsilon u :=(\varepsilon_1u_1, \ldots, \varepsilon_{n+1}u_{n+1})$, where $\varepsilon\in \{-1,1\}^{n+1}$, and $u\in \mathbb{S}^n\cap \mathbb{R}^{n+1}_+$. For each $1\leq i\leq N$, define the vector $a_i:=(u_1(v_i)_1,\dots,u_{n+1}(v_i)_{n+1})$. For $\sigma\in \{-1,1\}^{n+1}$, we have $\langle \varepsilon u,\sigma v_i\rangle=\langle \varepsilon\sigma,a_i\rangle$. Then, using \eqref{eq:Cauchy-polytope-projection-hyperplane}, \eqref{eq:Weighted-Expectation-up}, and the fact that the vectors $\varepsilon \sigma$ range over $\{-1,1\}^{n+1}$ as $\sigma$ ranges over $\{-1,1\}^{n+1}$, we compute
    %Let $F_1,\ldots,F_N$ be the sets of the form $F \cap \R^{n+1}_+$ where $F$ is a facet of $K$ whose relative interiors intersect $\mathbb{R}^{n+1}_+$, and let $v_1,\ldots,v_N$ be the corresponding outer unit normal vectors of these facet of $K$.
    %Since $K$ is unconditional, $v_1, \dots, v_N \in \mathbb{R}_{+}^{n+1}$. Let $u \in \R^{n+1}_+$ and $\varepsilon \in \{-1,1\}^{n+1} $. Using \eqref{eq:Cauchy-polytope-projection-hyperplane} and \eqref{eq:Weighted-Expectation-up}, we obtain
    \begin{equation}
        \label{eq:Rogers-Shephard-ineq-p-zonotope-3}
        \begin{split}
            |P_{(\varepsilon u)^\perp} K| 
            &= 
            \frac{1}{2} \sum_{\sigma \in \{-1,1\}^{n+1}} \sum_{i=1}^N |F_i| | \langle \varepsilon u , \sigma v_i \rangle |
            =\frac{1}{2}\sum_{i=1}^{N}\left(|F_i|\sum_{\sigma\in \{-1,1\}^{n+1}}|\langle a_i,\varepsilon\sigma\rangle| \right)
            % = \frac{1}{2} \sum_{i=1}^N \left(|F_i| \sum_{\sigma}  | u \cdot \sigma v_i| \right) 
            \\
            &=2^n\sum_{i=1}^{N}|F_i|\cdot\mathbb{E}_{\sigma}|\langle a_i,\sigma\rangle|\leq2^{n-1}\sum_{i=1}^{N}|F_i|\left(\|a_i\|_1+\min_{\tau\in \{-1,1\}^{n+1}}|\langle a_i,\tau\rangle|\right)\\
            &= 
            2^{n-1}\sum_{i=1}^N |F_i| \left( \langle u, v_i \rangle +\min_{\tau \in \{-1,1\}^{n+1}} | \langle \varepsilon u, \tau v_i \rangle | \right).
        \end{split}
    \end{equation}
    The polytope $K\cap \mathbb{R}^{n+1}_+$ is the intersection of $K$ with the closed half-spaces $\{x:x_i\geq 0\}$. That is, $K\cap \mathbb{R}^{n+1}_+$ is the convex polytope generated by the outer unit  normals $v_1,\dots,v_N$ as well as $-e_1,\dots,-e_{n+1}$. Since $u\in \mathbb{S}^n\cap\mathbb{R}^{n+1}_+$, it follows from a classical formula, see for example \cite[Lemma 5.1.1]{S-93}, that
    %On the other hand, the facet of $K \cap \mathbb{R}_{+}^{n+1}$ can be described into two kinds. First, there are facets inherited from $K$, precisely $v_1, \ldots, v_N$. Second, new facets arise from the intersection with the coordinate half-spaces $\left\{x_i \geq 0\right\}$. These facets lie in the coordinate hyperplanes and have outer unit normals \(-e_i\) for $i=1,\ldots,n+1$. Thus, 
    \begin{equation}
    \label{eq:Rogers-Shephard-ineq-p-zonotope-2}
        \sum_{i=1}^N |F_i|  \langle  u , v_i \rangle   
        = 
        % |P_{u^\perp} (K \cap \R^n_+)| = 
        \sum_{i=1}^{n+1} u_i |P_{e_i^{\perp}} (K \cap \R^{n+1}_+)|.
    \end{equation}
    Using \eqref{eq:Cauchy-polytope-projection-hyperplane} and \eqref{eq:Rogers-Shephard-ineq-p-zonotope-2}, we obtain
    \begin{equation}
        \label{eq:Rogers-Shephard-ineq-p-zonotope-4}
        \begin{split}
            |P_{(\varepsilon u)^\perp} (K \cap \R^{n+1}_+)| 
            &= 
            \frac{1}{2} \left( \sum_{i=1}^N |F_i| | \langle \varepsilon u ,  v_i \rangle |  +  \sum_{i=1}^{n+1} u_i |P_{e_i^{\perp}} (K \cap \R^{n+1}_+)| \right)
            \\
            &=
            \frac{1}{2} \left( \sum_{i=1}^N |F_i| | \langle \varepsilon u ,  v_i \rangle |  + \sum_{i=1}^N |F_i|  \langle  u , v_i \rangle  \right).
        \end{split}
    \end{equation}
    Combine \eqref{eq:Rogers-Shephard-ineq-p-zonotope-3}, \eqref{eq:Rogers-Shephard-ineq-p-zonotope-4}, and the fact that 
        \begin{equation*}
            |\langle \varepsilon u,v_i\rangle|\geq\min_{\tau\in\{-1,1\}^{n+1}}|\langle\varepsilon u,\tau v_i\rangle|
        \end{equation*}
    to complete the proof of the upper bound.
    
    % To clarify why it is $L_p$- sum of $n+1 $ segments,  observe that $$Z = P_{e_n^\perp} (B^{n+1}_q \cap \R^{n+1}_{+}) = B_q^{n} \cap \R^n_+ .$$
    To prove the lower bound, define $a_i$ as before and use \eqref{eq:Cauchy-polytope-projection-hyperplane}, \eqref{eq:Weighted-Expectation-lo}, and \eqref{eq:Rogers-Shephard-ineq-p-zonotope-4} to get
    \begin{equation}
        \begin{split}
            |P_{(\varepsilon u)^\perp} K| 
            &= 
            2^{n}\sum_{i=1}^{N}|F_i|\cdot\mathbb{E}_{\sigma}|\langle a_i,\sigma\rangle|\geq \binom{n}{\lfloor\frac{n+1}{2}\rfloor}\sum_{i=1}^{N}|F_i|\cdot\|a_i\|_1 \\
            &
            = \binom{n}{\lfloor \frac{n+1}{2} \rfloor} \sum_{i=1}^N |F_i| \langle u, v_i \rangle
            \\
            &\geq \binom{n}{\lfloor \frac{n+1}{2} \rfloor}
            |P_{(\varepsilon u)^\perp} (K \cap \R^{n+1}_+)|.
        \end{split}
    \end{equation}
    Note that the left inequality is sharp when $K = B_1^{n+1}$ with direction $\varepsilon u= (\frac{1}{\sqrt{n+1}},\ldots,\frac{1}{\sqrt{n+1}})$. Indeed, since there is only one facet, all components of $a$ are the same. This implies that the first inequality is equality. The second is also equality since $\varepsilon = (1,\ldots,1)$.
    % We note that equality can occur if and only if above inequalities are in fact equalities. In particular, this forces $\varepsilon \in\mathbb{R}^n_+$ and, by the equality conditions in \eqref{eq:Weighted-Expectation-up}, we must have $ u_j (v_i)_j =S $ for some constant $S$ that is independent of both $i$ and $j$. Consequently, all $v_i$ coincide, so $N=1$. In particular, $u$ is proportional to the coordinate-wise reciprocal of $v_1$  and $K$ is a dilation of the crosspolytope.

    Next, we will characterize the equality conditions for the upper bound. Suppose that the upper bound is an equality. Then, carefully observing the proof of the upper bound, we see from \eqref{eq:Rogers-Shephard-ineq-p-zonotope-3} that
    \begin{equation}\label{eq:rademacher_lp_hyperplane_theorem_equality}
        \mathbb{E}_{\sigma}|\langle a_i,\sigma\rangle|=\frac{1}{2}\left(\|a_i\|_1+\min_{\tau\in\{-1,1\}^{n+1}}|\langle a_i,\tau\rangle|\right)
    \end{equation}
    holds for each $1\leq i\leq N$, and from \eqref{eq:Rogers-Shephard-ineq-p-zonotope-3} and \eqref{eq:Rogers-Shephard-ineq-p-zonotope-4} we find that 
    \begin{equation}\label{eq:minimizing_reflection}
        |\langle a_i,\varepsilon\rangle|=\min_{\tau\in\{-1,1\}^{n+1}}|\langle a_i,\tau\rangle|
    \end{equation}
    is satisfied for each $i$. We showed in Lemma~\ref{lem:Weighted-Expectation} that if $\varepsilon$ satisfies \eqref{eq:rademacher_lp_hyperplane_theorem_equality}, and \eqref{eq:minimizing_reflection} is satisfied as well, then either $a_i$ has at most two positive coordinates and the corresponding coordinates of $\varepsilon$ have different signs, or $a_i$ has at least three positive coordinates and among the corresponding coordinates of $\varepsilon$ there is $\varepsilon_{n_i}$ whose sign is different from the signs of the other coordinates, and the bound
    \begin{equation*}
        2|u_{n_i}|(v_i)_{n_i}\geq \sum_{j=1}^{n+1}|u_j|(v_i)_j
    \end{equation*}
    is satisfied. These are exactly the conditions we need.

    We move on to prove the other direction. If we assume that the conditions are true, then we automatically have \eqref{eq:minimizing_reflection} for each $i$. We also notice that the conditions along with \eqref{eq:minimizing_reflection} are exactly the hypothesis of Lemma~\ref{lem:Weighted-Expectation} with $a_i$ in place of $a$, so we find immediately that \eqref{eq:rademacher_lp_hyperplane_theorem_equality} is true. We have established that both \eqref{eq:minimizing_reflection} and \eqref{eq:rademacher_lp_hyperplane_theorem_equality} are satisfied, which implies that the upper bound is an equality.
\end{proof}

Using Theorem~\ref{thm: equivalent-RS-Proj}, we obtain a particular case of the $L_p$-Rogers--Shephard inequality. For the reverse inequality, we make an improvement that depends on $p$. The following lemma will help us to characterize the equality conditions.
% The constant appearing below is sharp in the endpoint case $p=\infty$. We suspect that the constant in the left-hand inequality can be improved when one allows it to depend on
% $p$. However, our method is specifically designed for the $1$-unconditional setting in the Conjecture~\ref{conj:projection-cap-conj}, and for this reason it does not yield a $p$-dependent constant.

\begin{lemma}\label{lemma:at_most_two_postive_components}
    Let $u\in \mathbb{S}^n\cap \mathbb{R}^{n+1}_+$ be a direction, and suppose that there exists $\varepsilon\in \{-1,1\}^{n+1}$ such that 
    \begin{equation*}
        |\langle \varepsilon u, x\rangle|=\min_{\tau\in \{-1,1\}^{n+1}}|\langle \varepsilon u,\tau x\rangle|
    \end{equation*}
    holds for almost every $x\in\mathbb{R}^{n+1}_+$. Then, there exist at most two indices $i$ and $j$ such that the components $u_i$ and $u_j$ are positive.
\end{lemma}

\begin{proof}
    Assume for contradiction that there exist distinct indices $i$ and $j$ for which $u_i$ and $u_j$ are positive, and $\varepsilon_i=\varepsilon_j$. For $t\in(0,1)$ define
    \begin{equation}
        \label{eq:rs-zp-0001}
        \delta(t) := \frac{(u_i +u_j)t}{\|u\|_1}.
    \end{equation}
     If $x \in \R^{n+1}_+$ satisfies $\|x - e_i -e_j \|_\infty \leq \delta(t)$, then 
     \begin{align}\label{eq:delta_bounds}
        |x_i-1|\leq \delta(t), \qquad |x_j-1|\leq \delta(t),\qquad \textup{and} \qquad |x_k|\leq \delta(t),
     \end{align} 
     for any $k\notin \{i,j\}$. Use the triangle inequality and \eqref{eq:delta_bounds} to compute
     \begin{equation*}
         \begin{split}
             |\langle \varepsilon u,x\rangle|&=\left|u_ix_i+u_jx_j+\sum_{k\notin \{i,j\}}\varepsilon_i \varepsilon_kx_k\right|\geq|u_ix_i+u_jx_j|-\left|\sum_{k\notin \{i,j\}}\varepsilon_i \varepsilon_ku_kx_k\right|\\
             &\geq u_ix_i+u_jx_j-\sum_{k\notin \{i,j\}}u_k|x_k|\geq (u_i+u_j)(1-\delta(t))-\delta(t)\sum_{k\notin \{i,j\}}u_k\\
             &=(u_i+u_j)-\delta(t)\sum_{k=1}^{n+1}u_k=(1-t)(u_i+u_j).
         \end{split}
     \end{equation*}
     where the last equality we used the definition \eqref{eq:rs-zp-0001}.
      %\begin{align}
    %    |\langle \varepsilon u ,x \rangle| = |u_i x_i +u_j x_j + \sum_{k \notin\{ i,j\}} \varepsilon_k u_k x_k| \geq (u_i + u_j) (1-\delta) - \delta\sum_{k \neq i,j}  u_k  \overset{\eqref{eq:rs-zp-0001}}{=} (u_i + u_j)(1-t).
    %\end{align}
    %This immediately implies that \(u\) has at most two nonzero coordinates. Indeed, if \(u\) had three nonzero coordinates, say \(u_i,u_j,u_k>0\), then \(\varepsilon_i,\varepsilon_j,\varepsilon_k\in\{-1,1\}\) would have to be pairwise distinct, which is impossible. Now we prove the claim. 
    % $ t = \frac{1}{2} \left(1- \frac{|u_i-u_j|}{u_i+u_j}\right) \in (0,1/2)$ and 
    %and $t$ will be chosen later.
    Next, choose $\tau \in \{-1,1\}^{n+1}$ such that $\tau_i = 1$ and $\tau_j =-1$. Then, using the bounds \eqref{eq:delta_bounds} we get
    \begin{align}
        |\langle \tau u ,x \rangle| &= \left| u_i x_i -u_j x_j + \sum_{k \notin\{ i,j\}} \tau_k u_k x_k \right| 
        \\
        &\leq \left|
        u_i x_i -u_j x_j\right| + \delta(t)\sum_{k \notin\{i,j\}}  u_k  
        \\
        &\leq |u_i -u_j| +u_i |x_i-1| +u_j|1-x_j| + \delta(t)\sum_{k \notin\{i,j\}}  u_k 
        \\
        &\leq |u_i -u_j| + \delta(t)\sum_{k =1}^{n+1}  u_k 
        \\
        &= |u_i -u_j| + t(u_i+u_j).
    \end{align}
    Finally, we pick $t$ to satisfy
    \begin{equation*}
        0<t<\frac{1}{2}\left(1-\frac{|u_i-u_j|}{u_i+u_j}\right),
    \end{equation*}
    and we use the above computations to find that
    \begin{equation*}
        |\langle \tau u,x\rangle|\leq |u_i-u_j|+t(u_i+u_j)<(1-t)(u_i+u_j)\leq |\langle \varepsilon u,x\rangle|,
    \end{equation*}
    which contradicts to the assumption of the lemma. 
    %\eqref{eq:Rogers-Shephard-ineq-p-zonotope-6-smooth}.
    As a consequence, if $i$ and $j$ are distinct indices for which $u_i$ and $u_j$ are positive, then $\varepsilon_i\neq \varepsilon_j$.

    We conclude the lemma with the following argument. Suppose that there exist three distinct indices $i$, $j$, and $k$ such that the components $u_i$, $u_j$, and $u_k$ are positive. We just proved that $\varepsilon_i\neq \varepsilon_j$, $\varepsilon_i\neq \varepsilon_k$, and $\varepsilon_j\neq \varepsilon_k$. The first two inequalities force $\varepsilon_j=\varepsilon_k=-\varepsilon_i$, which is a contradiction. As a result, no such three points exist, and we have proved the lemma.
\end{proof}

Next, we consider the particular case $ K= B_q^{n+1} $ of Proposition~\ref{prop:Projection-ineq-hyperplane} to obtain the equality conditions needed for the $L_p$-Rogers-Shephard inequality. We also make an improvement on the lower bound that depends on $q$.
\begin{prop}\label{prop:projection_hyperplane_lq_ball}
    If $1<q<\infty$ is fixed and $v\in\mathbb{S}^n$ is a direction, then
    \begin{equation}\label{eq:hyperplane_projection_lq_ball}
         c_{q,n+1}|P_{v^{\bot}}(B_q^{n+1}\cap \mathbb{R}^{n+1}_+)|_n\leq |P_{v^{\bot}}B_q^{n+1}|_n\leq 2^n|P_{v^{\bot}}(B_q^{n+1}\cap \mathbb{R}^{n+1}_+)|_n,
    \end{equation}
    where 
    \begin{equation}
        \label{const:hyperplane_projection_lq_ball_low}
        c_{q,n+1} := \frac{2^n \left|P_{\theta^\perp}B_q^{n+1}\right|_n}{\sqrt{n+1}|B_q^n|}, 
        \qquad
        \theta=\left(\frac{1}{\sqrt{n+1}},\ldots,\frac{1}{\sqrt{n+1}}\right).
        % \frac{2^{n}|P_{(\frac{1}{\sqrt{n+1}},\ldots,\frac{1}{\sqrt{n+1}})^\perp} B_q^{n+1}|_n}{\sqrt{n+1}|B_q^n|}.
    \end{equation}
    Equality holds if and only if either $v$ has exactly one nonzero component, or $v$ has exactly two nonzero components with different signs. The left inequality is sharp, with equality attained at $v = \theta$.
\end{prop}

\begin{proof}
    The inequality follows by using the same approach as in Proposition~\ref{prop:Projection-ineq-hyperplane}.
    We represent an arbitrary vector of $\mathbb{S}^n$ by $\varepsilon u$, where $u\in \mathbb{S}^n\cap\mathbb{R}_+^{n+1}$ and $\varepsilon\in \{-1,1\}^{n+1}$. It follows from a classical formula, proved for example in \cite[Lemma 5.1.1]{S-93}, that 
    % The polytope $B_q^m\cap \mathbb{R}^{m}_+$ is generated by the unit facet normals $v_1,\dots,v_N$ and $-e_1,\dots,-e_{m} \in \R^m$. Since $\langle u,v_i \rangle \geq 0$ for any $i \in [N]$ and $\langle u,-e_i \rangle \leq 0$, we get
    \begin{equation}
    \label{eq:Rogers-Shephard-ineq-p-zonotope-lp}
        |P_{u^\perp} (B_q^{n+1} \cap \R^{n+1}_+) |_{n} 
        = \int_{\s^n\cap \mathbb{R}^{n+1}_+}  \langle u, \xi \rangle \; {d}S_{B_q^{n+1}} (\xi)
        = \sum_{i=1}^{n+1}u_i |P_{e_i^{\perp}} (B_q^{n+1}\cap \R^{n+1}_+)|_{n}.
    \end{equation}
    Using Cauchy formula \eqref{eq:Cauchy-polytope-projection-hyperplane-smooth}, and the fact that $S_{B_q^{n+1}}(\sigma \xi)=S_{B_q^{n+1}}(\xi)$ for any $\sigma\in\{-1,1\}^{n+1}$ and $\xi\in \mathbb{S}^n$, we compute
    % using \eqref{eq:Cauchy-polytope-projection-hyperplane}, \eqref{eq:Weighted-Expectation-up}, and that the vectors $\varepsilon \sigma$ range over $\{-1,1\}^{m}$ as $\sigma$ ranges over $\{-1,1\}^{m}$, we compute
    %Let $F_1,\ldots,F_N$ be the sets of the form $F \cap \R^{n+1}_+$ where $F$ is a facet of $K$ whose relative interiors intersect $\mathbb{R}^{n+1}_+$, and let $v_1,\ldots,v_N$ be the corresponding outer unit normal vectors of these facet of $K$.
    %Since $K$ is unconditional, $v_1, \dots, v_N \in \mathbb{R}_{+}^{n+1}$. Let $u \in \R^{n+1}_+$ and $\varepsilon \in \{-1,1\}^{n+1} $. Using \eqref{eq:Cauchy-polytope-projection-hyperplane} and \eqref{eq:Weighted-Expectation-up}, we obtain
    \begin{align}
            |P_{(\varepsilon u)^\perp} B_{q}^{n+1}|_{n} 
            &=
            \frac{1}{2} \sum_{\sigma \in \{-1,1\}^{n+1}}  \int_{\s^{n}\cap \R^{n+1}_+} | \langle \varepsilon u , \sigma \xi \rangle | \; {d} S_{B_q^{n+1}}(\sigma\xi)\\
            &= \frac{1}{2}   \int_{\s^{n}\cap \R^{n+1}_+} \sum_{\sigma \in \{-1,1\}^{n+1}} | \langle \varepsilon u , \sigma \xi \rangle | \; {d} S_{B_q^{n+1}}(\xi)
            % \frac{1}{2}\sum_{i=1}^{N}\left(|F_i|\sum_{\sigma\in \{-1,1\}^{n+1}}|\langle a_i,\varepsilon\sigma\rangle| \right)
            % = \frac{1}{2} \sum_{i=1}^N \left(|F_i| \sum_{\sigma}  | u \cdot \sigma v_i| \right) 
            \\
            &=2^{n}\int_{\s^{n}\cap \R^{n+1}_+} \mathbb{E}_{\sigma}| \langle \varepsilon u , \sigma \xi \rangle | \; {d} S_{B_q^{n+1}}(\xi).
    \end{align}
    Using Lemma~\ref{lem:Weighted-Expectation}, equality \eqref{eq:Rogers-Shephard-ineq-p-zonotope-lp} and Cauchy formula \eqref{eq:Cauchy-polytope-projection-hyperplane-smooth}, we have
    \begin{align}
            \qquad |P_{(\varepsilon u)^\perp} B_{q}^{n+1}|_{n} 
            &\leq
            2^{n-1}\int_{\s^{n}\cap \R^{n+1}_+}  \left( \langle u, \xi \rangle +\min_{\tau \in \{-1,1\}^{n+1}} | \langle \varepsilon u, \tau \xi \rangle | \right) \; {d} S_{B_q^{n+1}}(\xi) \label{eq:Rogers-Shephard-ineq-p-zonotope-3-smooth}
            \\
            &\leq 2^{n-1}\int_{\s^{n}\cap \R^{n+1}_+} \langle u, \xi \rangle +| \langle \varepsilon u,  \xi \rangle | \; {d} S_{B_q^{n+1}}(\xi)
            \label{eq:Rogers-Shephard-ineq-p-zonotope-4-smooth}
            \\
            &\overset{}{=}
            2^{n-1}\left(\int_{\s^{n}\cap \R^{n+1}_+}  | \langle \varepsilon u,  \xi \rangle |  \; {d} S_{B_q^{n+1}}(\xi) + \sum_{i=1}^{n+1} u_i |P_{e_i^{\perp}} (B_q^{n+1}\cap \R^{n+1}_+)|_{n} \right)
            \\
            &=
            2^{n} |P_{(\varepsilon u)^{\bot}} (B_q^{n+1}\cap \mathbb{R}^{n+1}_+)|_{n}.
    \end{align}
    This proves exactly the projection bound we need, and as a result, we have verified the required upper bound for zonotopes.
    
    To characterize the equality cases, suppose that \eqref{eq:hyperplane_projection_lq_ball} is equality. Then, \eqref{eq:Rogers-Shephard-ineq-p-zonotope-4-smooth} is equality, which implies that 
    \begin{equation*}
        | \langle \varepsilon u, \xi \rangle | = \min_{\tau \in \{-1,1\}^{n+1}} | \langle \varepsilon u, \tau \xi \rangle |
    \end{equation*}
    for almost every $\xi \in \s^n\cap\R^{n+1}_+$.
    % \begin{equation}
    %     \label{eq:Rogers-Shephard-ineq-p-zonotope-5-smooth}
    %     | \langle \varepsilon u, \xi \rangle | = \min_{\tau \in \{-1,1\}^{n+1}} | \langle \varepsilon u, \tau \xi \rangle |.
    % \end{equation}
    Since the function $|\langle \varepsilon u,\cdot\rangle|$ is 1-homogeneous, we have 
    \begin{equation}
        \label{eq:Rogers-Shephard-ineq-p-zonotope-6-smooth}
        | \langle \varepsilon u, x \rangle | = \min_{\tau \in \{-1,1\}^{n+1}} | \langle \varepsilon u, \tau x \rangle |
    \end{equation}
    for almost every $x \in \R^{n+1}_+$. This is exactly the situation described in Lemma~\ref{lemma:at_most_two_postive_components}, so we conclude that either $\varepsilon u$ has exactly one nonzero component, or there exist distinct indices $i$ and $j$ such that $u_i$ and $u_j$ are positive with $\varepsilon_i\neq \varepsilon_j$. The second case is equivalent to the condition that $\varepsilon u$ has exactly two nonzero components with different signs.

    For the other direction, suppose that either $\varepsilon u$ has exactly one nonzero component, or that there exist distinct indices $i$ and $j$ such that $u_i$ and $u_j$ are positive with $\varepsilon_i\neq \varepsilon_j$. Then, from Lemma~\ref{lem:Weighted-Expectation} we see that \eqref{eq:Rogers-Shephard-ineq-p-zonotope-3-smooth} is an equality. For each $\xi\in \mathbb{S}^n\cap \mathbb{R}^{n+1}_+$, we compute
    \begin{equation}
        |\langle \varepsilon u,\xi\rangle|
        =\min\{|\varepsilon_iu_i\xi_i-\varepsilon_ju_j\xi_j|, |\varepsilon_iu_i\xi_i+\varepsilon_ju_j\xi_j|\}
        =\min_{\tau\in \{-1,1\}^{n+1}}|\langle \varepsilon u,\tau \xi\rangle|,
    \end{equation}
    which shows that \eqref{eq:Rogers-Shephard-ineq-p-zonotope-4-smooth} is equality. Putting these observations together, we see that \eqref{eq:hyperplane_projection_lq_ball} is equality.
    
    %Now observe that if \(u\) has at most two nonzero coordinates, and the corresponding entries of \(\varepsilon\) have opposite signs, then for any \(i,j\) (not necessarily distinct) with \(u_i,u_j>0\), one has \(\varepsilon_i\neq \varepsilon_j\). In this case, a direct computation shows that equality holds in \eqref{eq:Rogers-Shephard-ineq-p-zonotope-3-smooth} and \eqref{eq:Rogers-Shephard-ineq-p-zonotope-4-smooth}. Now observe that if \(u\) has at most two nonzero coordinates, and the corresponding entries of \(\varepsilon\) have opposite signs, then for any \(i,j\) (not necessarily distinct) with \(u_i,u_j>0\), one has \(\varepsilon_i\neq \varepsilon_j\). In this case, a direct computation shows that equality holds in \eqref{eq:Rogers-Shephard-ineq-p-zonotope-3-smooth} and \eqref{eq:Rogers-Shephard-ineq-p-zonotope-4-smooth}.  
    Next, we prove the reverse inequality. We claim that the minimum
    \[
        \min_{\varepsilon, u} \frac{|P_{(\varepsilon u)^\perp} B_q^{n+1}|}{|P_{(\varepsilon u)^\perp} (B_q^{n+1}\cap \R^{n+1}_+)|}
    \]
    is attained at $ u = (\frac{1}{\sqrt{n+1}}, \ldots , \frac{1}{\sqrt{n+1}})$ and $ \varepsilon = (1,\ldots , 1)$.
    Using Cauchy formula \eqref{eq:Cauchy-polytope-projection-hyperplane-smooth}, we have
    \begin{align}
        |P_{(\varepsilon u)^{\bot}} (B_q^{n+1}\cap \mathbb{R}^{n+1}_+)|_{n}
        &=
        \frac{1}{2}
        \left(\int_{\s^{n}\cap \R^{n+1}_+}  | \langle \varepsilon u,  \xi \rangle |  {d} S_{B_q^{n+1}}(\xi) + \sum_{i=1}^{n+1} u_i |P_{e_i^{\perp}} (B_q^{n+1}\cap \R^{n+1}_+)|_{n} \right)
        \\
        &\leq
        \frac{1}{2}
        \left(\int_{\s^{n}\cap \R^{n+1}_+}  | \langle  u,  \xi \rangle |  \; {d} S_{B_q^{n+1}}(\xi) + \sum_{i=1}^{n+1} u_i |P_{e_i^{\perp}} (B_q^{n+1}\cap \R^{n+1}_+)|_{n} \right)
        \\
        &\leq 
        |P_{ u^{\bot}} (B_q^{n+1}\cap \mathbb{R}^{n+1}_+)|_{n}.
    \end{align}
    Thus, $ \varepsilon = (1,\dots ,1)$. Using \eqref{eq:Rogers-Shephard-ineq-p-zonotope-lp}, it is enough to prove that
    \[
        \min_{\varepsilon, u} \frac{|P_{(\varepsilon u)^\perp} B_q^{n+1}|}{|P_{(\varepsilon u)^\perp} (B_q^{n+1}\cap \R^{n+1}_+)|} = \min_{ u } \frac{2^{n}|P_{u^\perp} B_q^{n+1}|}{\|u\|_1 |B_q^n|}.
    \]
    Using the formula proved by Barthe and Naor \cite[Theorem 3]{BN-02}, we have
\begin{align}
&|P_{u^{\perp}}B_q^{n+1}|=\frac{|B_q^{n}|}{\mathbb{E}\left|X_1\right|} \mathbb{E}\left|\sum_{j=1}^{n+1} u_j X_j\right|,
\end{align}
where $X_1, \ldots, X_{n+1}$ are i.i.d.\ random variables with density $f_q(x)=\frac{q}{2(q-1) \Gamma(1 / q)}|x|^{\frac{2 -q}{q-1}} e^{-|x|^{\frac{q}{q-1}}}$. Then,
\begin{align}
    \frac{2^{n}|P_{u^\perp} B_q^{n+1}|}{\|u\|_1|B_q^n|} = \frac{2^{n}}{\mathbb{E}\left|X_1\right|} \mathbb{E}\left|\sum_{j=1}^{n+1} \frac{u_j}{\|u\|_1} X_j\right|.
    \label{eq:expectation_lowerbound}
\end{align}
Consider the function
$$
    F(a)= \mathbb{E}\left|\sum_{j=1}^{n+1} a_j X_j\right|.
$$ 
The function $F$ is convex and invariant under any permutation of its coordinates. Hence, by the same argument as in Lemma~\ref{lem:Weighted-Expectation-lo}, for any $a$ with $\|a\|_1 = 1$, we have
\[
    F\left(\frac{1}{{n+1}}, \ldots, \frac{1}{{n+1}}\right) \leq F(a).
\]
Therefore, \eqref{eq:expectation_lowerbound} is minimized when $u = (\frac{1}{\sqrt{n+1}},\ldots,\frac{1}{\sqrt{n+1}})$.
\end{proof}

\begin{remark}
    The preceding proof technique works for smooth unconditional convex bodies, and one can check the equality cases by using the same argument. More precisely,  \eqref{eq:Rogers-Shephard-ineq-p-zonotope-4-smooth} is equality if and only if $S_K (\{u \in \R^n_+ : \text{ there exist at least } 3 \text{ nonzero components of } u\})= 0$. 
\end{remark}

\begin{corollary}
    \label{cor:Rogers-Shephard-ineq-p-zonotope_hyperplane}
    For a fixed $1<p<\infty$, let $Z\subset \R^n$ be an asymmetric $L_p$-zonotope that is the $L_p$-sum of exactly $n+1$ segments of the form $[0,u_i]$ for $u_i\in \mathbb{R}^n$. Then,
    \begin{equation}
    \label{eq:Rogers-Shephard-ineq-p-zonotope_hyperplane}
        c_{q,n+1} |Z| \leq |Z \oplus_p -Z| \leq 2^n |Z|,
    \end{equation}
    where $c_{q,n+1}$ is defined according to \eqref{const:hyperplane_projection_lq_ball_low}.
    Equality holds on the right-side bound if and only if $Z$ is contained in a hyperplane or is the $L_p$-sum of $n$ segments with the origin as an endpoint. The left inequality is sharp.
\end{corollary}

\subsection{Projections onto lines} \label{sec:proj-line}

\begin{lemma}\label{lemma:supremum_over_K}
    Let $K\subset\mathbb{R}^m$ be an unconditional convex body. Then, for any $v\in \mathbb{S}^{m-1}$ we have
    \begin{align*}
        h_{K\cap \mathbb{R}^m_+}(v)=\sup_{k\in K}\sum_{v_i>0}k_iv_i\qquad \textup{and}\qquad h_{K\cap\mathbb{R}^m_+}(-v)=\sup_{k\in K}\sum_{v_i<0}k_iv_i,
    \end{align*}
    where we use the convention that a sum taken over the empty set is equal to $0$.
\end{lemma}

\begin{proof}
    We begin with the observation that
    \begin{equation}\label{eq:absolute_value_representation}
        K\cap \mathbb{R}^m_+=\{(|k_1|,\dots,|k_m|):(k_1,\dots,k_m)\in K\}.
    \end{equation}
    To see this, we take an element $(k_1,\dots,k_m)\in K$. By the unconditionality of $K$, we have $(|k_1|,\dots,|k_m|)\in K\cap \mathbb{R}^m_+$, proving that the right side is contained in the left side. On the other hand, if $(k_1,\dots,k_m)\in K\cap \mathbb{R}^m_+$, then $(k_1,\dots,k_m)=(|k_1|,\dots,|k_m|)$, proving that the left side is contained in the right side, which establishes \eqref{eq:absolute_value_representation}.

    We will also observe that
    \begin{equation}\label{eq:sum_over_positive_terms}
        \sup_{k\in K\cap \mathbb{R}^m_+}\sum_{v_i>0}k_iv_i=\sup_{k\in K\cap \mathbb{R}^m_+}\sum_{i=1}^mk_iv_i.
    \end{equation}
    To verify this, choose $k^*\in K\cap\mathbb{R}^m_+$ so that
    \begin{equation*}
        \sum_{v_i>0}k_i^*v_i=\sup_{k\in K\cap \mathbb{R}^m_+}\sum_{v_i>0}k_iv_i.
    \end{equation*}
    Because we take the above sums over the indices $i$ for which $v_i>0$, we lose no generality in assuming that $k_i^*=0$ whenever $v_i\leq 0$. It follows that
    \begin{equation*}
        \sum_{i=1}^mk_i^*v_i=\sum_{v_i>0}k_i^*=\sup_{k\in K\cap \mathbb{R}^m_+}\sum_{v_i>0}k_iv_i,
    \end{equation*}
    from which we conclude
    \begin{equation*}
        \sup_{k\in K\cap \mathbb{R}^m_+}\sum_{i=1}^mk_iv_i\geq \sup_{k\in K\cap \mathbb{R}^m_+}\sum_{v_i>0}k_iv_i.
    \end{equation*}
    In the other direction, we have
    \begin{equation*}
        \sup_{k\in K\cap \mathbb{R}^m_+}\sum_{i=1}^mk_iv_i=\sup_{k\in K\cap \mathbb{R}^m_+}\left(\sum_{v_i>0}k_iv_i+\sum_{v_i<0}k_iv_i\right)\leq \sup_{k\in K\cap \mathbb{R}^m_+}\sum_{v_i>0}k_iv_i,
    \end{equation*}
    which completes the verification of \eqref{eq:sum_over_positive_terms}.

    To prove the identity on the left, we use \eqref{eq:absolute_value_representation} and \eqref{eq:sum_over_positive_terms} to compute
    \begin{equation*}
        \begin{split}
            \sup_{k\in K}\sum_{v_i>0}k_iv_i\leq \sup_{k\in K}\sum_{v_i>0}|k_i|v_i=\sup_{k\in K\cap \mathbb{R}^m_+}\sum_{v_i>0}k_iv_i=\sup_{k\in K\cap \mathbb{R}^m_+}\sum_{i=1}^mk_iv_i=h_{K\cap \mathbb{R}^m_+}(v).
        \end{split}
    \end{equation*}
    Next, we use \eqref{eq:sum_over_positive_terms} to get
    \begin{equation*}
        \begin{split}
            h_{K\cap \mathbb{R}^m_+}(v)=\sup_{k\in K\cap \mathbb{R}^m_+}\sum_{i=1}^mk_iv_i  =\sup_{k\in K\cap \mathbb{R}^m_+}\sum_{v_i>0}k_iv_i\leq \sup_{k\in K}\sum_{v_i>0}k_iv_i.
        \end{split}
    \end{equation*}
    From the above two computations, we conclude that
    \begin{equation*}
        h_{K\cap\mathbb{R}^m_+}(v)=\sup_{k\in K}\sum_{v_i>0}k_iv_i,
    \end{equation*}
    as we claimed. We omit the proof of the second identity because it is almost identical to the proof of the first identity.
\end{proof}

To state the equality conditions in the next result, we define the face of $K$ in direction $u\in \mathbb{S}^m$ by
\begin{align*}
    F(K;u):=\{k\in K:\langle k,u \rangle=h_K(u)\}.
\end{align*}
\begingroup
\renewcommand{\MainTheoremExtraThree}{%
Moreover, if we define 
        \begin{align*}
            v_+:=\sum_{v_i>0}v_ie_i\qquad \textup{and}\qquad v_-:=\sum_{v_i<0}v_ie_i,
        \end{align*}
        then equality holds on the right if and only if either $v\in \mathbb{R}^m_+\cup (-\mathbb{R}^m_+)$, or $v\notin \mathbb{R}^m_+\cup (-\mathbb{R}^m_+)$ and the faces $F(K;v_+)$ and $F(K;v_-)$ intersect. In the case where $K$ is the ball $B_q^m$ with $1\leq q<\infty$, equality holds if and only if $v\in \mathbb{R}^m_+\cup (-\mathbb{R}^m_+)$.  The left inequality is sharp.
}
\ProjectionOne*
\endgroup
\begin{proof}
    Because the subspace $E$ has dimension $1$, there exists a vector $v\in \mathbb{S}^{m-1}$ such that $E=\textup{span}(v)$. Then, the projection of an arbitrary vector $x\in \mathbb{R}^m$ onto the subspace $E$ is exactly $P_E(x)=v\langle x,v\rangle$, and as a result we write
    \begin{equation*}
        P_EK=\{v\langle k,v\rangle:k\in K\}=v[-h_K(-v),h_K(v)].
    \end{equation*}
    From the unconditionality of $K$, we find that
    \begin{equation*}
        |P_EK|_1=h_K(v)+h_K(-v)=2h_K(v).
    \end{equation*}
    Using the same idea as above, we compute
    \begin{equation*}
        |P_E(K\cap \mathbb{R}^m_+)|_1=h_{K\cap \mathbb{R}^m_+}(v)+h_{K\cap\mathbb{R}^m_+}(-v)=h_{K\cap \mathbb{R}^m_+-K\cap \mathbb{R}^m_+}(v).
    \end{equation*}
    Consequently, the bound we want to prove is equivalent to the bound
    \begin{equation*}
        h_K(v)\leq h_{K\cap \mathbb{R}^m_+-K\cap \mathbb{R}^m_+}(v),
    \end{equation*}
    which holds if the containment $K\subset K\cap \mathbb{R}^m_+-K\cap \mathbb{R}^m_+$ is true. 
    
    To verify this containment, let $k\in K$ and define $I:=\{i:k_i\geq 0\}$. Then, $k$ has the orthogonal decomposition $k=k_I+k_{[m]\backslash I}$, where we define
    \begin{align*}
        k_I:=\sum_{i\in I}k_ie_i\qquad\textup{and}\qquad k_{[m]\backslash I}:=\sum_{i\in [m]\backslash I}k_ie_i.
    \end{align*}
    By the unconditionality of $K$, we have that $k_I-k_{[m]\backslash I}\in K$. Then, by the convexity of $K$, we get 
    \begin{equation*}
        k_I\in [k_I-k_{[m]\backslash I},k_I+k_{[m]\backslash I}]\subset K,
    \end{equation*}
    so that $k_I\in K\cap \mathbb{R}^m_+$. Similar reasoning shows that $k_{[m]\backslash I}\in -(K\cap \mathbb{R}^m_+)$. Putting everything together, we find that $k\in K\cap \mathbb{R}^m_+-K\cap \mathbb{R}^m_+$.
    
    Suppose that equality holds on the right and that $v\notin \mathbb{R}^m_+\cup (-\mathbb{R}^m_+)$. Then, from the above computations, we see that equality holds if and only if
    \begin{equation*}
        h_K(v)=h_{K\cap \mathbb{R}^m_+}(v)+h_{K\cap \mathbb{R}^m_+}(-v).
    \end{equation*}
    From Lemma~\ref{lemma:supremum_over_K} and the definition of $v_+$ and $v_-$, this is equivalent to
    \begin{equation}\label{eq:equivelent_condition_dim_one_equality}
       \sup_{k\in K}(\langle k,v_+\rangle+\langle k,v_-\rangle)= \sup_{k\in K}\langle k,v_+\rangle+\sup_{k\in K}\langle k,v_-\rangle,
    \end{equation}
    which holds if and only if there exists $k_0\in K$ such that 
    \begin{align*}
        \langle k_0,v_+\rangle=h_K(v_+)\qquad \textup{and}\qquad \langle k_0,v_-\rangle=h_K(v_-).
    \end{align*}
    That is, $k_0\in F(K,v_+)\cap F(K,v_-)$, verifying the equality conditions in the case $v\notin \mathbb{R}^m_+\cup (-\mathbb{R}^m_+)$.

    If $v\in \mathbb{R}^m_+\cup (-\mathbb{R}^m_+)$, then either $v_+=0$, or $v_-=0$. As a result, either $\langle k,v_+\rangle=0$ for all $k\in K$, or $\langle k,v_-\rangle=0$ for all $k\in K$. Under this condition, \eqref{eq:equivelent_condition_dim_one_equality} is immediate, which is equality on the right side. If $F(K,v_+)\cap F(K,v_-)$ is non-empty, then we follow the previous step, and it is clear that equality holds.

    Now, we show that the left inequality is sharp when $K = B_1^m$ and $v=\frac{1}{\sqrt{2}}(e_i-e_j)$. Thus,
    \[
        |P_E B_1^m| = 2h_{B_1^m}(v) = 2\| v \|_\infty = \sqrt{2}.
    \]
    Using the fact that $
h_{\operatorname{conv}\left(A_1, \ldots, A_m\right)}(u)=\max _{1 \leq j \leq m} h_{A_j}(u)
$ for any compact convex sets $A_i$, we have
    \begin{align}
        |P_E (B_1^m \cap \R^m_+)| &= h_{B_1^m\cap \mathbb{R}^m_+}(v)+h_{B_1^m\cap\mathbb{R}^m_+}(-v) 
        = h_{\conv \{ 0,e_1,\ldots,e_m\}}(v)+h_{\conv \{ 0,e_1,\ldots,e_m\}}(-v) 
        \\
        &= \max_{1\leq j \leq m} \{ 0, \langle v,e_j \rangle : 1\leq j \leq m \} + \max_{1\leq j \leq m} \{ 0, \langle -v,e_j \rangle : 1\leq j \leq m \} 
        =\sqrt{2}. 
    \end{align}
    Thus, $|P_E (B_1^m \cap \R^m_+)| = |P_E B_1^m |$.
\end{proof}

\subsection{Local Maximizer} \label{sec: local maximum}

We recall a formula due to Lutwak \cite{L-93, L-96}. For any convex bodies $A,B$ containing the origin with $0 \in \operatorname{int} A$, we have
\begin{equation}
    \label{eq:lp-mixed volume}
    \lim _{\varepsilon \rightarrow 0} \frac{\left|A \oplus_p\left(\varepsilon^{\frac{1}{p}} B\right)\right|-|A|}{\varepsilon}=\frac{1}{p} \int_{\s^{n-1}} h_B^p(u) h_A^{1-p} (u)\, d S_A(u).
\end{equation}
We present below a lemma to handle \eqref{eq:lp-mixed volume} in the case where the origin lies on the boundary of the set. Our argument follows the approach from  Gardner, Hug, and Weil \cite[Lemma 8.3]{GHW-14}.

Recall from \cite[Theorem 5.1.7]{S-93} that for any compact convex sets $K_1,\ldots,K_m \subset \R^n$ and non-negative numbers $t_1,\ldots,t_m$, one has
\begin{align} \label{eq:mixed-volume}
    |t_1K_1+\ldots + t_m K_m| =\sum_{i_1,\ldots, i_m=1}^m t_{i_1} \cdots t_{i_m} V(K_{i_1}, \ldots, K_{i_n}).
\end{align}
for some non-negative numbers $V(K_{i_1}, \ldots, K_{i_n})$, which are called the mixed volumes of $K_{i_1}, \ldots, K_{i_n}$. Similarly, the mixed area measures $S_{K_1, \ldots,K_{n-1}}$ are defined to be the coefficients in the following expansion
\[
    S_{t_1K_1+\cdots+t_mK_m} (\cdot) =\sum_{i_1,\ldots,i_{n-1}=1}^{m} t_{i_1}\cdots t_{i_{n-1}} S_{K_{i_1},\ldots,K_{i_{n-1}}} (\cdot),
\]
Moreover, one has the integral representation of mixed volume via the mixed area measure
\begin{equation} \label{eq:int-representation-mixed-vol}
    V(K_1,\ldots,K_n) = \frac{1}{n} \int_{\s^{n-1}} h_{K_1}(u) \,d S_{K_2,\ldots,K_n} (u).
\end{equation}
\begin{lemma} \label{lem:LYZ-0inboundary}
    Let $L$ be a compact convex set containing the origin. If $L \subset \R^n_+$, we have
    \[
        \lim_{\varepsilon\rightarrow 0^+} \frac{ | (B_q^n \cap \R^n_+) \oplus_p \varepsilon^{\frac1p} L| - |B_q^n \cap \R^n_+|}{\varepsilon} =\frac{1}{p} \int_{\s^{n-1}\cap\R^n_+} h_{L}^p (u) h_{B_q^n \cap \R^n_+}^{1-p} (u) \, dS_{B_q^n \cap \R^n_+} (u).
    \]
    If $L\not\subset\mathbb{R}^n_+$, the left-hand side diverges to $+\infty$.
\end{lemma}

\begin{proof}
    Let $K_\varepsilon = (B_q^n \cap \R^n_+) \oplus_p \varepsilon^{\frac{1}{p}} L$. We write the telescoping sum
    \begin{equation}
        |K_\varepsilon |- |K_0| = V( K_\varepsilon [n]) - V(K_\varepsilon [n-1], K_0) + \dots + V( K_\varepsilon ,K_0[n-1]) - V(K_0[n]),
    \end{equation}
    where we use the notation $A[m]$ for a convex set $A$ repeated $m$ times.
    Using \eqref{eq:int-representation-mixed-vol}, each pair $V( K_\varepsilon [i], K_0 [n-i])- V(K_\varepsilon [i-1], K_0[n-i+1]) $ in the telescoping sum can be written as
    \begin{align}
        % V&( K_\varepsilon [i], K_0 [n-i])- V(K_\varepsilon [i-1], K_0[n-i+1]) 
        % \\
        % &=
        \frac{1}{n} \int_{\s^{n-1}} h_{K_\varepsilon} (u)- h_{K_0} (u)\,dS_{K_\varepsilon [i-1],K_0[n-i]} (u).
        % \\
        % &=\frac{1}{n} \int_{\s^{n-1}} (h_{K_0}^p + \varepsilon h_L^p)^{\frac{1}{p}} - h_{K_0} \,dS_{K_\varepsilon [i-1],K_0[n-i]}.
        % \\
        % &=\frac{1}{n} \left(\int_{\s^{n-1} \setminus -\R^n_+} (h_{K_0}^p + \varepsilon h_L^p)^{\frac{1}{p}} - h_{K_0} \,dS_{K_\varepsilon [i-1],K_0[n-i]} +\int_{\s^{n-1} \cap -\R^n_+} \varepsilon^{\frac{1}{p}} h_L \,dS_{K_\varepsilon [i-1],K_0[n-i]} \right)
    \end{align}
    Therefore,
    \begin{equation} \label{eq:tele-mixed}
       \frac{\left|K_{\varepsilon}\right|-\left|K_0\right|}{\varepsilon}=\frac{1}{n} \sum_{i=1}^n \int_{\s^{n-1}} \frac{h_{K_{\varepsilon}}(u)-h_{K_0}(u)}{\varepsilon} d S_{K_{\varepsilon}[i-1], K_0[n-i]}(u).
    \end{equation}
    % Note that $h_{K_\varepsilon} - h_K \geq 0$ since $K_0 \subset K_\varepsilon$.  
    Now, we will divide the proof into two cases, depending on whether $L \subset \R^n_+$, or $L \not \subset \R^n_+$.
    \begin{itemize}[wide =0pt, labelwidth= 1.1cm]
        \item[Case 1:] 
        % If $L \subset \R^n_+$, then $K_\varepsilon \subset \R^n_+$ and so $(h_{K_\varepsilon} - h_{K_0}) (u) = 0$ for any $ u =-e_i$ for any $i$. Thus
        If $L \subset \R^n_+$, then there exists a constant $M > 0$  such that $h_L \leq M h_{K_0}$ for all $ u \in \s^{n-1}$. Thus,
        \[
            \frac{h_{K_{\varepsilon}}(u)-h_{K_0}(u)}{\varepsilon} \rightarrow f \quad  \text{uniformly on } \s^{n-1},
        \]
        where
        \[
            f(u)= \begin{cases}\frac{1}{p} h_L^p(u) h_{K_0}^{1-p}(u), & h_{K_0}(u)>0, \\ 0, & h_{K_0}(u)=0 .\end{cases}
        \]
        % $K_\varepsilon \subset \R^n_+$ and so $(h_{K_\varepsilon} - h_{K_0}) (u) = 0$ for any $ u =-e_i$ for any $i$. Thus
        % \begin{align}
        %     % V( K_\varepsilon [i], K_0 [n-i]) - V(K_\varepsilon [i-1], K_0[n-i+1]) 
        % % &=\frac{1}{n} \int_{\s^{n-1}} h_{K_\varepsilon} - h_{K_0} \,dS_{K_\varepsilon [i-1],K_0[n-i]}
        % % % \\
        % 0 &\leq \frac{1}{n} \int_{\s^{n-1}\setminus\R^n_+} h_{K_\varepsilon} (u) - h_{K_0} (u) \,dS_{K_\varepsilon [i-1],K_0[n-i]} (u)
        % \\
        % &\leq \frac{1}{n} \int_{\s^{n-1}\setminus\R^n_+} h_{K_\varepsilon} (u) - h_{K_0} (u) \,dS_{K_\varepsilon [n-i]}(u)
        % % \\&
        % = \sum_{i=1}^n | P_{e_i^\perp} K_\varepsilon| (h_{K_\varepsilon}  - h_{K_0} )(-e_i) = 0.
        % \end{align}
        Since 
        % $\lim_{\varepsilon \to 0^+} \frac{h_{K_\varepsilon}(u) -h_{K_0} (u)}{\varepsilon} = h_{L}^p (u) h_{B^n_q \cap \R^n_+}^{1-p}(u) $ uniformly for $u\in \s^{n-1}\cap \R^n_+$ and since 
        $S_{K_{\varepsilon}[i], K[n-1-i]} $ converges $ S_K$ weakly as $\varepsilon \rightarrow 0^+$, we have that 
        \begin{align}
            &\lim_{\varepsilon \to 0^+} \frac{V( K_\varepsilon [i], K_0 [n-i]) - V(K_\varepsilon [i-1], K_0[n-i+1])}{\varepsilon} 
            = \frac{1}{n} \int_{\s^{n-1}\cap \R^n_+} f(u) \, dS_{K_0} (u)
            \\
            &= \frac{1}{np} \int_{\s^{n-1}\cap \R^n_+} h_{L}^p (u) h_{B_q^n \cap \R^n_+}^{1-p} (u) \, dS_{K_0} (u),
        \end{align}
        where in the last equality, we use the fact that $S_{K_0}$ is supported on $ \left(\mathbb{S}^{n-1} \cap \mathbb{R}_{+}^n\right) \cup\left\{-e_1, \ldots,-e_n\right\} $ and that $ h_L\left(-e_i\right)=0. $
        Using \eqref{eq:tele-mixed}, we have
        \[
            \lim_{\varepsilon \to 0^+} \frac{ |K_\varepsilon | - |K_0|}{\varepsilon} = \frac{1}{p} \int_{\s^{n-1}\cap \R^n_+} h_{L}^p (u)h_{B_q^n \cap \R^n_+}^{1-p} (u)\, dS_{K_0} (u).
        \]
        \item[Case 2:] There exists $i$ such that $h_{L}(-e_i) >0$. Note that $ h_{K_0} (-e_i) = 0$.
        Since $K_0 \subset K_\varepsilon$, we have that $h_{K_0} \leq h_{K_\varepsilon}$. Using \eqref{eq:tele-mixed}, we have
        % \begin{align}
        %     % \frac{V( K_\varepsilon [i], K_0 [n-i]) - V(K_\varepsilon [i-1], K_0[n-i+1])}{\varepsilon} 
        %     &\frac{1}{n} \int_{\s^{n-1}} h_{K_\varepsilon} (u) - h_{K_0} (u)\,dS_{K_\varepsilon [i-1],K_0[n-i]} (u)
        %     \\
        %     &
        %     \geq \frac{1}{n} \int_{\s^{n-1}\cap \R^n_-} h_{K_\varepsilon} (u)- h_{K_0} (u)\,dS_{K_\varepsilon [i-1],K_0[n-i]} (u)
        %     \\
        %     &=
        %     \frac{1}{n} \int_{\s^{n-1}} \varepsilon^{\frac{1}{p}} h_L (u) \,dS_{K_\varepsilon [i-1],K_0[n-i]} (u).
        % \end{align}
        \begin{align}
            |K_\varepsilon| - |K_0| 
            \geq 
            % V\left(K_{\varepsilon}, K_0[n-1]\right)-V\left(K_0[n]\right)
            % =
            \frac{1}{n} \int_{\mathbb{S}^{n-1}}\left(h_{K_{\varepsilon}}(u)-h_{K_0}(u)\right) d S_{K_0}(u) \geq \frac{\varepsilon^{1/p}}{n}\frac{|B_q^{n-1}|}{2^{n-1}} h_L(-e_i).
        \end{align}
        % Since $S_{K_{\varepsilon}[i], K[n-1-i]} $ converges $ S_K$ weakly as $\varepsilon \rightarrow 0+$, we have 
        % Thus,
        % \[
        %     \lim_{\varepsilon \to 0^+}\int_{\s^{n-1}} h_L (u) \,dS_{K_\varepsilon [i-1],K_0[n-i]} (u)=\lim_{\varepsilon \to 0^+}\int_{\s^{n-1}} h_L (u) \,dS_{K_0[n-1]} (u)= \sum_{i=1}^n \frac{|B_q^{n-1}|}{2^{n-1}} h_L(-e_i)
        % \]
        % is positive and finite.
        Hence, $\lim_{\varepsilon \to 0^+} \frac{ |K_\varepsilon | - |K_0|}{\varepsilon}$ diverges to $+\infty$. \qedhere
    \end{itemize}
\end{proof}
\LocalMax*
% \begin{theorem}
%     \label{thm:local-max}
%     Let $p >1$. Let $\zeta$ be an asymmetric $L_p$-zonoid and let $u_1, \ldots,u_n$ be linear independent vectors on $\R^n$. Denote $Z_{\varepsilon} = [0,u_1] \oplus_p \ldots \oplus_p [0,u_n] \oplus_p \left(\varepsilon^{\frac{1}{p}} \zeta\right) $ for $\varepsilon \geq 0$. There exists $\varepsilon_0 = \varepsilon_0 (\zeta,u_1,\ldots,u_n)$ such that for any $0 \leq \varepsilon \leq \varepsilon_0$, we have
%     \begin{equation}
%         \label{eq:local-max}
%         | Z_{\varepsilon} \oplus_p - Z_{\varepsilon} | \leq 2^n | Z_{\varepsilon} |.
%     \end{equation}
%     Equality holds if and only if $\zeta$ is of the form $\zeta = [0,a_1u_1] \oplus_p \ldots \oplus_p [0,a_nu_n]$ for some $a_i \geq 0$.
% \end{theorem}

\begin{proof}
    By applying a linear transformation, it is enough to consider the case where $u_i = e_i$ for all $i \in [n]$. We first observe that if the generating measure of $\zeta$ is supported on $\{e_1,\ldots,e_n\}$, then $\zeta = \bigoplus_{i \in [n]}^{(p)} [0,a_ie_i]$, for some $ 0 \leq a_i$. Since 
    \[
    [0,e_i] \oplus_p \varepsilon^{\frac{1}{p}}[0,a_ie_i] = [0,(1+\varepsilon a_i^p)^{1/p}e_i],
    % \begin{cases}
    %     [0,(1+\varepsilon a_i^p)^{1/p}e_i] &a_i \geq 0
    %     \\
    %     [\varepsilon^{\frac{1}{p}}a_ie_i,e_i] &a_i \leq 0
    % \end{cases}
    \]
    we see that $Z_\varepsilon $ is an $L_p$-sum of $n$-segments. It follows from Proposition \ref{prop: RS-n-seg} that $|Z_\varepsilon
    \oplus_p -Z_\varepsilon| = 2^n |Z_\varepsilon|$.
    
    Now, we assume that the generating measure $\mu$ of $\zeta$ is not supported on $\{e_1,\ldots,e_n\}$.
    Using Lemma~\ref{lem:positive_orthant_support_function}, we have $Z_{\varepsilon} = \left(B_q^n \cap \R^n_+\right) \oplus_p \left(\varepsilon^{\frac{1}{p}} \zeta\right)$. Let
    % \[
    %     F(\varepsilon) = \frac{ | Z_{\varepsilon} \oplus_p - Z_{\varepsilon} | }{| Z_{\varepsilon} |}.
    % \]
    \[
        F(\varepsilon) =  | Z_{\varepsilon} \oplus_p - Z_{\varepsilon} | -2^n| Z_{\varepsilon} |.
    \]
    We claim that $F(\varepsilon) < F(0) $ for small enough $\varepsilon > 0$. Once we verify this claim, \eqref{eq:local-max} follows since $F(0) = 0$. It is enough to prove that
    \begin{equation}
         \lim _{\varepsilon \rightarrow 0^+} \frac{F (\varepsilon ) -F(0)}{\varepsilon} < 0.
    \end{equation}
    % Note that the claim implies that $F(\varepsilon) \leq F(0)$ for small enough $\varepsilon>0$ and then \eqref{eq:local-max} follows since $F(0) = 2^n$.
    Using Lemma~\ref{lem:positive_orthant_support_function}, we have $Z_{\varepsilon} \oplus_p - Z_{\varepsilon} =B_q^n \oplus_p \left(\varepsilon^{\frac{1}{p}}(\zeta \oplus - \zeta) \right)$. Thus,
    \begin{align}
        \lim _{\varepsilon \rightarrow 0^+} \frac{F (\varepsilon ) -F(0)}{\varepsilon}
        &= 
        % \frac{1}{|Z_{p,0}|^2} \left( |Z_{p,0}| 
        \lim _{\varepsilon \rightarrow 0^+} \frac{| Z_{\varepsilon} \oplus_p - Z_{\varepsilon} |-| Z_{0} \oplus_p - Z_{0} |}{\varepsilon}  - 2^n \lim _{\varepsilon \rightarrow 0^+} \frac{| Z_{\varepsilon} |-| Z_{0} |}{\varepsilon} 
        % \right)
        \\
        &= 
        % \frac{4^n}{|B_q^n|^2} \left( \frac{|B_q^n|}{2^n} 
        \lim _{\varepsilon \rightarrow 0^+} \frac{| Z_{\varepsilon} \oplus_p - Z_{\varepsilon} |-|B_q^n |}{\varepsilon}  - 2^n \lim _{\varepsilon \rightarrow 0^+} \frac{| Z_{\varepsilon} |-| B_q^n \cap\R^n_+ |}{\varepsilon} 
        % \\
        % &= \frac{2^n}{|B_q^n|} \left(  \lim _{\varepsilon \rightarrow 0} \frac{| Z_{\varepsilon} \oplus_p - Z_{\varepsilon} |-| Z_{p,0} \oplus_p - Z_{p,0} |}{\varepsilon}  - 2^n\lim _{\varepsilon \rightarrow 0} \frac{| Z_{\varepsilon} |-| Z_{p,0} |}{\varepsilon} \right)
        . \label{pf:local-max-001}
    \end{align}
    Now, using \eqref{eq:lp-mixed volume}, we obtain
    \begin{align}
        \lim _{\varepsilon \rightarrow 0} \frac{| Z_{\varepsilon} \oplus_p - Z_{\varepsilon} |-| Z_{0} \oplus_p - Z_{0} |}{\varepsilon}
        &=  \lim _{\varepsilon \rightarrow 0}  \frac{ \left|B_q^n \oplus_p \left(\varepsilon^{\frac{1}{p}}(\zeta \oplus_p - \zeta) \right)\right| - \left|B_q^n \right|}{\varepsilon}
        \\
        &=  \frac{1}{p} \int_{\s^{n-1}} h_{\zeta \oplus_p - \zeta}^p (u) h_{B_q^n}^{1-p} (u)  \,dS_{B_q^n } (u) .
    \end{align}
    Using Lemma~\ref{lem:LYZ-0inboundary}, if $\zeta \not \subset \R^n_+$, the second limit diverges to $+\infty$.  Since there is a negative sign in front of the second limit, the claim holds in this case. On the other hand, if $\zeta \subset \R^n_+$, we have
    \begin{align}
        &\lim _{\varepsilon \rightarrow 0} \frac{| Z_{\varepsilon} |-| Z_{0}  |}{\varepsilon}
        % \\
        % &=  \lim _{\varepsilon \rightarrow 0}  \frac{ \left|(B_q^n \cap \R^n_+) \oplus_p \left(\varepsilon^{\frac{1}{p}}\zeta \right)\right| - \left|B_q^n \cap \R^n_+ \right|}{\varepsilon}
        =  \frac{1}{p} \int_{\s^{n-1} \cap \R^n_+} h_{\zeta }^p(u) h_{B_q^n\cap\R^n_+}^{1-p} (u)\,dS_{B_q^n\cap \R^n_+} (u) 
        % \\
        % &= \frac{1}{p} \left(  \int_{\s^{n-1}\cap \R^n_+} h_{\zeta }^p (u) h_{B_q^n}^{1-p} (u) \,dS_{B_q^n} (u) + \frac{|B_q^{n-1}|}{2^{n-1}}\sum_{i=1}^n h_{\zeta}^p (-e_i) h_{B_q^n \cap \R^n_+}^{1-p}(-e_i)\right)
        .
    \end{align}
    For any $\delta \in \{-1,1\}^n$, denote $R_\delta^n = \{ x: \delta_i x_i \geq 0 \text{ for all } i\}$. Using \eqref{pf:local-max-001}, it is enough to show that for any $\delta \in \{-1,1\}^n$, 
    \begin{align}
        \int_{\s^{n-1}\cap \R^n_\delta} h_{\zeta \oplus_p - \zeta}^p (u)  h_{B_q^n}^{1-p} (u) \,dS_{B_q^n } (u)\leq  \int_{\s^{n-1} \cap \R^n_+} h_{\zeta }^p (u)h_{B_q^n\cap\R^n_+}^{1-p} (u) \,dS_{B_q^n\cap \R^n_+} (u),
    \end{align}
    where the inequality is strict for at least one $\delta$. Since
    \begin{align}
        \label{pf:measure-zp}
        h_{\zeta}^p (u) = \int_{\s^{n-1}} \langle v,u \rangle_+^p \, d\mu (v), 
    \end{align}
    we obtain
    \begin{align}
        \label{pf:measure-d-zp}
        h_{\zeta\oplus_p - \zeta}^p (u)&=h_{\zeta}^p (u) +h_{-\zeta}^p (u) = h_{\zeta}^p (u) +h_{\zeta}^p (-u) 
        \\
        &= \int_{\s^{n-1}} \langle v,u \rangle_+^p \, d\mu (v) +\int_{\s^{n-1}} \langle -v,u \rangle_+^p \, d\mu (v)
        = \int_{\s^{n-1}} |\langle v,u \rangle|^p \, d\mu (v).  
    \end{align}
    Since $\zeta \subset \R^n_+$, we have $h_{\zeta} (-e_i) = 0$ for any $i$. Thus, $\mu$ is supported on the positive orthant. Therefore, using \eqref{pf:measure-zp}, \eqref{pf:measure-d-zp} together with Fubini's theorem and the triangle inequality,
    \begin{align}
        \int_{\s^{n-1}\cap \R^n_\delta} h_{\zeta \oplus_p - \zeta}^p (u) h_{B_q^n}^{1-p} (u) \,dS_{B_q^n } (u)
        &= \int_{\s^{n-1}\cap \R^n_\delta} \left(\int_{\s^{n-1}} |\langle v,u \rangle|^p \, d\mu (v)\right)  h_{B_q^n}^{1-p} (u) \,dS_{B_q^n }(u)
        \\
        % &= \int_{\s^{n-1}} \int_{\s^{n-1}\cap \R^n_\delta} |\langle v,u \rangle|^p   h_{B_q^n}^{1-p} (u) \,dS_{B_q^n }(u) \, d\mu (v)
        % \\
        &= \int_{\s^{n-1}} \int_{\s^{n-1}\cap \R^n_+} \left|\sum_{i=1}^n \delta_i v_i u_i\right|^p   h_{B_q^n}^{1-p} (u) \,dS_{B_q^n }(u) \, d\mu (v)
        \\
        &\leq \int_{\s^{n-1}} \int_{\s^{n-1}\cap \R^n_+} \left|\langle v,u \rangle\right|^p  h_{B_q^n}^{1-p} (u) \,dS_{B_q^n }(u) \, d\mu (v)
        \\
        %&= \int_{\s^{n-1}\cap \R^n_+} \int_{\s^{n-1}}  \left|\langle v,u \rangle\right|^p   h_{B_q^n}^{1-p} (u) \, d\mu (v) \,dS_{B_q^n }(u) \\
        &= \int_{\s^{n-1}\cap \R^n_+} h_{\zeta}^p  (u) h_{B_q^n}^{1-p} (u) \,dS_{B_q^n } (u).
    \end{align}
    Here, the triangle inequality is valid since $\mu$ is supported on $\R^n_+$ or $v \in \R^n_+$. Now, observe that there is exactly one inequality in the above computation. Suppose for contradiction that this inequality is in fact equality for every $\delta\in\{-1,1\}^n$. Then for every $\delta\in\{-1,1\}^n$, the equality
        \begin{equation}
    \label{pf:equality-case-local-max}
        \int_{\s^{n-1} \cap \R^n_+}\left|\sum_{i=1}^n \delta_i v_i u_i \right|^p \,d\mu(v)= \int_{\s^{n-1}\cap \R^n_+} \left(\sum_{i=1}^n v_i u_i \right)^p \,d\mu(v).
    \end{equation}
    holds for almost every $u\in \mathbb{S}^{n-1}\cap \mathbb{R}^n_+$. But then for every $\delta$ and for almost every $u$, the equality
    \begin{align}\label{eq:delta_u_v_triangle_inequality}
        \left|\sum_{i=1}^n \delta_i v_i u_i \right| = \sum_{i=1}^n v_i u_i,
    \end{align}
    holds for almost every $v\in \mathbb{S}^{n-1}$. However, the equality \eqref{eq:delta_u_v_triangle_inequality} holds for all $\delta$ if and only if $v=e_i$ for some $i$. But 
    \begin{align*}
        \mu\left(\left(\mathbb{S}^{n-1} \cap \mathbb{R}_{+}^n\right) \backslash\left\{e_1, \ldots, e_n\right\}\right)>0,
    \end{align*}
    so the requirement that for almost every $v$, $v=e_i$ for some $i$ is impossible, and therefore there must exist at least one $\delta$ so that the equality \eqref{pf:equality-case-local-max} does not hold, and so we have the required strict inequality for at least one $\delta$.

%    Now, we observe that only one inequality appears in the above computation, and equality holds if and only if for every $\delta \in \{-1,1\}^n$, for almost every $u \in \s^{n-1} \cap \R^n_+$, 
%    \begin{equation}
    %\label{pf:equality-case-local-max}
%        \int_{\s^{n-1} \cap \R^n_+}\left|\sum_{i=1}^n \delta_i v_i u_i \right|^p \,d\mu(v)= \int_{\s^{n-1}\cap \R^n_+} \left(\sum_{i=1}^n v_i u_i \right)^p \,d\mu(v).
%    \end{equation}
%    Note that 
%    \[
%        \left|\sum_{i=1}^n \delta_i v_i u_i \right| \le \sum_{i=1}^n v_i u_i \text{ for all } \delta \text{ and  for almost every } u \in \R^n_+ \cap \s^{n-1},
%    \]
%    with equality if and only if $v = e_i$ for some $i$. Since
%$
%\mu\left(\left(S^{n-1} \cap \mathbb{R}_{+}^n\right) \backslash\left\{e_1, \ldots, e_n\right\}\right)>0,
%$
%then the inequality is strict for at least one choice of $\delta$, and hence the sum over all orthants is strict.
\end{proof}

\subsection{An average version of Conjecture~\ref{conj:projection-cap-conj}} \label{sec: average}
Next, we consider a weaker version of Conjecture~\ref{conj:projection-cap-conj}, where we take the average over all $n$-dimensional subspaces. 
% Using the Cauchy-Kubota formula, the $n$-intrinsic volume of $K$ can be expressed as
%     \begin{equation}
%         V_n (K) = \frac{\tbinom{m}{n}\omega_{m}}{\omega_{m-n}\omega_n} \int_{G(m,n)} |P_E K| d \mu(E),
%     \end{equation}
%     where the integration is taken with respect to the Haar probability measure on the Grassmannian $\mathrm{G}(m, n)$ of $n$-dimensional subspaces in $\mathbb{R}^m$. 
    Motivated by this representation, we ask for the largest constant $a_{n,q}$ such that
    \[
        V_n(B_q^m\cap \R^m_+) \geq a_{n,q} V_n(B_q^m).
    \]
    We note that Conjecture~\ref{conj:projection-cap-conj} suggests one should have $a_{n,q} \geq \frac{1}{2^n}$, which is an optimal constant independent of $q$, since for $q =\infty$ the inequality is an equality (corresponding to the $\ell_\infty$-ball in Conjecture~\ref{conj:projection-cap-conj}). As a consequence of Proposition~\ref{prop:Projection-ineq-hyperplane} and Theorem~\ref{thm:Rogers-Shephard-ineq-2-zonoids}, the lower bound is verified in the case $q = 2$ and in the case $ n = m-1$ for any 1-unconditional convex body. We further confirm the bound for all $q\ge 2$ in Proposition~\ref{prop:Projection-ineq-weighted}, relying on a Loomis--Whitney type inequality for intrinsic volumes due to Campi and Gronchi~\cite{CG-11}. In particular, if $Z$ is a zonoid in $\mathbb{R}^m$, then 
    \begin{equation}
        \label{eq:aver-intrisic-vol}
        V_n(Z) \leq \frac{1}{m-n} \sum_{i=1}^m V_n(P_{e_i^\perp}Z).
    \end{equation}
    Equality holds if and only if $Z$ is a rectangular parallelotope with facets parallel to the coordinate hyperplanes. It was conjectured that \eqref{eq:aver-intrisic-vol} holds for any convex body $K$, see \cite{CGG-16}.

    We recall that a functional $\Phi$ defined on a class of all compact convex non-empty sets is called a valuation if it has the following additivity property,
    $$
        \Phi(A \cup B)=\Phi(A)+\Phi(B)-\Phi(A \cap B)
    $$
    whenever $A, B, A \cup B $ are in the class. It is well-known that intrinsic volume is a valuation; see \cite[Section 6]{S-93} and also \cite{S-18}. For any $j$, we define $e_j^{+} := \{x: \langle x,e_j \rangle \geq 0 \}$ and similarly for $e_j^{-}$.

\begin{lemma}
    Let $K$ be a 1-unconditional convex body and $1 \leq i \leq m $. Then,
    \begin{equation}
        \label{eq:quermassintegral-1-unconditional}
        2^i V_{n} \left(K \cap \bigcap_{j=1}^i e_j^+ \right) = \sum_{j=0}^i \binom{i}{j} V_{n} \left(K \cap \bigcap_{k=1}^j e_k^\perp \right).
    \end{equation}
\end{lemma}

\begin{proof}
    We will prove the lemma by induction on $i$. The case $i=1$ follows directly from the fact that the intrinsic volume is a valuation. Denote 
    $$
        K_i^+ = K \cap \bigcap_{j=1}^i e_j^+,\qquad \text{ and }\qquad   K_i = K \cap \bigcap_{j=1}^i e_j^\perp.
    $$ 
    Thus,
    % \begin{equation}
    \begin{align*}
            &V_{n} (K) 
            = 2^i V_{n} \left(K_i^+\right) - \sum_{j=1}^i \tbinom{i}{j} V_{n} \left(K_j \right)
            \\
            &= 2^i \left(2V_n \left(K_i^+ \cap e_{i+1}^+\right) - V_n \left(K_i^+ \cap e_{i+1}^\perp\right) \right) - \sum_{j=1}^i \tbinom{i}{j} V_n (K_j)
            % \\
            % &= 2^{i+1} V_n \left(\cap_{j=1}^{i+1} e_j^+ \cap K\right) - 2^iV_n \left(\cap_{j=1}^{i} e_j^+ \cap K \cap e_{i+1}^\perp\right)  - \sum_{j=1}^i \binom{i}{j} V_n (\cap_{k=1}^j e_k^\perp \cap K )
            \\
            &= 2^{i+1} V_n \left(K_{i+1}^+\right) - 
            V_n (K\cap e_{i+1}^\perp)  
            - \sum_{j=1}^i \tbinom{i}{j} V_n \left(K_j \cap e_{i+1}^\perp\right)
            - \sum_{j=1}^i \tbinom{i}{j} V_n (K_j ),
    \end{align*}
    % \end{equation}
    where the first equality follows from the inductive hypothesis applied to $K$, the second equality uses the valuation property applied to $K_i^+ \cap e_{i+1}^+$ and $K_i^+ \cap e_{i+1}^-$, and the final equality follows from the inductive hypothesis applied to $K \cap e_{i+1}^\perp$. The proof is completed by combining terms.
\end{proof}

\ProjectionIneqWeighted*
% \begin{prop}
%     \label{prop:Projection-ineq-weighted}
%     Let $m\ge n$. For any $1$-unconditional zonoid $K\subset\mathbb{R}^m$, we have
%     \begin{equation}
%         \label{eq:Projection-ineq-weighted}
%         V_n \left(K \cap \R^m_+ \right) \geq \frac{1}{2^{n}} V_n \left(K  \right).
%     \end{equation}
%     % In particular, we obtain
%     % \begin{equation}
%     %     \label{eq:Projection-ineq-weighted}
%     %         V(B_q^m [n],B_2^m[m-n]) < 2^n V \left(B_q^{m} \cap \R^m_+ [n],B_2^m[m-n]\right).
%     % \end{equation}
%     Equality holds if and only if $K$ is an $\ell_\infty$-ball.
% \end{prop}

\begin{proof}
    It is enough to prove that for any $j$,
    \begin{equation}
        \label{eq:aver-intrisic-vol-ind}
        \binom{m}{j} V_n \left(K \cap \bigcap_{k=1}^j e_k^\perp \right) \geq \binom{m-n}{j} V_n \left(K  \right).
    \end{equation}
    Indeed, using \eqref{eq:quermassintegral-1-unconditional} with $i =m$, and \eqref{eq:aver-intrisic-vol-ind}, we have
    \begin{equation}
        \begin{split}
            2^m V_n \left(K \cap \R^m_+ \right) 
            = \sum_{j=0}^{m-n} \binom{m}{j} V_n \left(K \cap \bigcap_{k=1}^j e_k^\perp \right)
            % \\
            % &
            \geq\sum_{j=0}^{m-n} \binom{m-n}{j} V_n \left(K  \right) 
            % \\
            % &
            = 2^{m-n} V_n \left(K  \right).
        \end{split}
    \end{equation}
    We prove \eqref{eq:aver-intrisic-vol-ind} by induction on $j$. The case $j=1$ is precisely \eqref{eq:aver-intrisic-vol}. Assume the statement holds for some $j\ge 1$. Since for a $1$-unconditional body, the intersection with a coordinate subspace coincides with the orthogonal projection onto that subspace, and since the orthogonal projection of a zonoid is a zonoid, we may apply \eqref{eq:aver-intrisic-vol},
    \begin{align*}
         V_n \left(K \cap \bigcap_{k=1}^{j} e_k^\perp \right) &\leq \frac{1}{m-n} \sum_{i=1}^m V_n \left(K \cap \bigcap_{k=1}^{j} e_k^\perp \cap e_i^\perp\right)
         \\
         &= \frac{j}{m-n}  V_n \left(K \cap \bigcap_{k=1}^{j} e_k^\perp \right) + \frac{m-j}{m-n}  V_n \left(K \cap \bigcap_{k=1}^{j+1} e_k^\perp \right).
    \end{align*}
    Thus, we obtain
    \begin{equation}
        V_n \left(K \cap \bigcap_{k=1}^{j+1} e_k^\perp \right) 
        \geq  
        \frac{m-n-j}{m-j} V_n \left(K \cap \bigcap_{k=1}^{j} e_k^\perp \right) 
        \geq \frac{\binom{m-n}{j}(m-n-j)}{\binom{m}{j}(m-j)} V_n \left(K  \right).
    \end{equation}
    We complete the proof of the inequality using the binomial identity.

    To verify the equality conditions, we first consider the case when $K = aB_\infty^m$. Observe that  $|P_E (a B_\infty^m)| = 2^n |P_E ( aB_\infty^m \cap \R^m_+)|.$ Using the definition of intrinsic volume, 
    \[
    V_n (K) = \frac{\tbinom{m}{n}\omega_{m}}{\omega_{m-n}\omega_n} \int_{G(m,n)} |P_E K|_n d \mu(E),
    \]
    we have that $V_n(K) = 2^nV_n(K \cap \R^m_+) $.
    Now, we assume that equality holds. That is,
    \[
         V_n(K) = 2^n V_n(K \cap \R^m_+).
    \]
    This implies that for all $j =0,\ldots,m-n$, we have 
    \[
        \binom{m}{j} V_n \left(K \cap \bigcap_{k=1}^j e_k^\perp \right) = \binom{m-n}{j} V_n \left(K  \right).
    \]
    When $j = 1$, this is precisely \eqref{eq:aver-intrisic-vol}, and equality holds in \eqref{eq:aver-intrisic-vol}  if and only if $K$ is a parallelotope with facets parallel to the coordinate hyperplanes. Since $K$ is 1-unconditional, $K = a B_\infty^m$ for some $ a \geq 0$.
\end{proof}

\begin{remark}
    Proposition~\ref{prop:Projection-ineq-weighted} gives the affirmative answer to Conjecture~\ref{conj:projection-cap-conj} in the average sense for $q \geq 2$ since the $\ell_q$-ball for $q \geq 2$ is a $1$-unconditional zonoid. We also note that the case $n=m-1$ follows from the result of Barthe and Naor~\cite{BN-02}. Indeed, one has
    \begin{equation}
        \label{eq:Vm-1-ratio}
         \frac{V_{m-1} (B_q^m \cap \R^m_+)}{V_{m-1} (B_q^m)} = \frac{|\partial (B_q^m \cap \R^m_+)|}{|\partial B_q^m |}
         = \frac{\frac{|\partial B_q^m |}{2^m} + \frac{m| B_q^{m-1} |}{2^{m-1}}}{|\partial B_q^m |} = \frac{1}{2^m} + \frac{m}{2^{m-1}} \frac{| B_q^{m-1} |}{|\partial B_q^m |}.
    \end{equation}
    Using the Cauchy-Kubota formula, we have
    \[
        \frac{|\partial B_q^m |}{| B_q^{m-1} |}=
         \frac{\omega_{m}}{\omega_{m-1}} \int_{G(m,m-1)} \frac{|P_E B_q^m|}{| B_q^{m-1} |} d \mu(E),
    \]
    and it was shown in \cite{BN-02} that for every unit vector $u$ the function $q\mapsto |P_{u^\perp}B_q^m|/|B_q^{m-1}|$ is non-decreasing on $[1,\infty)$. Consequently, the ratio \eqref{eq:Vm-1-ratio} is non-increasing in $q$. It is thus enough to verify the bound in the limit $q\to\infty$, which is an equality.
\end{remark}

\section{Symmetric bodies} \label{sec:sym-bodies}

Inspired by the Mahler conjecture for the symmetric case (see \cite{FMZ-23}), we present in this section our study of the maximum of $|K\oplus_p(-K)|/|K|$ in the plane when we restrict to symmetric convex bodies, and more generally for asymmetric $L_1$-zonoids in $\R^n$. 
% We recall that a convex body $K$ is origin-symmetric if $K = -K$. A convex body is symmetric if there exists a point $a \in K$ such that $K-a$ is origin symmetric.
We follow the idea from \cite{CG-06} by approximating a zonoid with a zonotope and then using shadow systems to reduce the number of segments.

Now, we recall the basic knowledge of shadow systems. A shadow system $X_t$ along a vector $v\in \R^n$ of points from $\R^n$ is a family of sets represented as follows:
\[
    X_t = \{ x_i + t a_i v \}_{i\in I},
\]
where $ t\in [t_0,t_1]$, $x_i,v \in \R^n$, $a_i \in \R$ and $I$ is an arbitrary set of indices. It follows from the definition that the shadow systems of convex bodies are convex hulls of the shadow systems of points.

% \begin{theorem}\label{th:center-sym}
%     Let $p\geq 1$. For any convex body $K \subset \R^2$ having a center of symmetry and containing the origin, we have
%     \begin{equation}
%         |K \oplus_p -K | \leq \left(\frac{2\Gamma(1+\frac{1}{q})^2}{\Gamma(1+\frac{2}{q})} +2 \right)|K|,
%     \end{equation}
%     where $q$ satisfies $\frac{1}{p} + \frac{1}{q} = 1$.
%     Equality holds if $K$ is a parallelogram with vertex at the origin.
% \end{theorem}

% \begin{conj}
%     We conjecture that among planar convex bodies having a center of symmetry, for $p>1$, there is equality in Theorem~\ref{th:center-sym} only for parallelograms with a vertex at the origin.
% \end{conj}
     
% \begin{conj}
%  Let $K$ be a convex body in $\R^n$ having a center of symmetry and containing the origin and let $p>1$. Then
%   \[
%   |K \oplus_p -K|\le \kappa^{(s)}_{n,p}|K|,\quad\hbox{where}\quad  \kappa^{(s)}_{n,p}=\sum_{i=0}^n\frac{\binom{n}{i}}{\binom{n/q}{i/q}}.
%   \]
%   with equality for $K=[0,1]^n$.
% \end{conj}
% The fact that equality holds  for $K=[0,1]^n$ follows from formula \eqref{eq:lemAB-003} applied to $K_\delta=[0,1]^n$ and $K_{-\delta}=[-1,0]^n$ and get
% \[
% |[0,1]^n\oplus_p[-1,0]^n|=\sum_{i=0}\sum_{E } \binom{n/q}{i/q}^{-1}      \left| P_E[0,1]^n\right| \left| P_{E^\perp} [-1,0]^n \right|=\sum_{i=0}^n\frac{\binom{n}{i}}{\binom{n/q}{i/q}}.
% \]

% \subsection{Another proof maybe}
\RGsymZonoid*
% \begin{theorem}
% \label{thm:RG-sym-zonoid}
%     Let $p \ge 1$, $n\ge2$ and $K$ be an asymmetric zonoid  in $\R^n$ containing the origin. Then
%     \begin{equation} 
%         \label{eq:RG-sym-zonoid}
%         | K \oplus_p -K| \leq \kappa^{(s)}_{n,p} |K|, \quad \text{where } \kappa^{(s)}_{n,p} :=\sum_{i=0}^n\frac{\binom{n}{i}}{\binom{n/q}{i/q}}
%     \end{equation}
%     where $q$ satisfies $\frac{1}{p}+\frac{1}{q}=1$. Equality holds if $K$ is a parallelotope with vertex at the origin.
% \end{theorem}

\begin{proof}
    By continuity and a standard density argument, it is enough to prove this when $Z$ is a zonotope. We are going to prove that any parallelotope with a vertex at the origin is a maximizer. Now, we write $Z$ as the sum of one-sided segments, that is,
    \[
        Z = \sum_{i=1}^m [0,v_i] 
    \]
    for some $v_i \in \R^n \setminus \{0\}$. Denote $\Lambda = \{ v_1, \ldots.v_m\}$. We may assume that there exists an index $i$ such that $\Lambda \setminus \{v_i\}$ spans $\R^n$. If not, then $\Lambda$ consists of exactly $n$ vectors. Indeed, since $\Lambda$ spans $\R^n$, there exist $\{v_{i_1},\ldots,v_{i_n} \}$ spanning $\R^n$ and if there is another vector $v_{i_{n+1}}$, then $\Lambda \setminus \{v_{i_{n+1}}\}$ spans $\R^n.$
    
    There exists a positive real number $a$ such that $\Lambda_t$ consists of
    \begin{equation}
        w_1 (t)= (1+at)v_1, \quad w_i (t)= v_i - tv_1 \frac{\langle v_1,v_i\rangle}{\|v_1\|_2^2}, 
    \end{equation}
    where $i =2,\ldots,m$ and $t \in[-\frac{1}{a},1]$, see \cite{CG-06} for precisely expression of $a$. Here, $\Lambda_0 = \Lambda$, at $t=1$, $w_1(1)$ is orthogonal to the other vectors in $\Lambda_1$ and at $ t = -\frac{1}{a}$, $w_{1}(-\frac{1}{a}) = 0$.
    % Suppose $K$ is not an affine image of unit cube. There exists
    % We may assume that there exist indices $i \neq j$ such that $v_j \not\in v_i^\perp$. Otherwise, for distinct indices $i,j$, we have $v_i \perp v_j$. Hence, there are at most $n$ vectors. Since $K$ is full-dimensional, it follows that $K$ is parallelotope. Without of loss of generality, we assume that $i=1$.
    Let
    $$
        Z(\Lambda_t) =[0,w_1(t)] + \ldots + [0,w_m(t)]
    $$ 
    be an asymmetric zonotope generated by $\Lambda_t$.
    It was proved by Campi and Gronchi in \cite{CG-06} that the volume $| Z(\Lambda_t) |$ is constant. Weberndorfer proved that  $Z(\Lambda_t)$ is a shadow system in \cite[Theorem 3.2]{W-13}.
    It was proved in \cite[Theorem 2.3]{BC-08} that for any shadow systems $\{X_t\}_{t\in I}$ and $\{Y_t\}_{t\in I}$, one has $\{X_t \oplus_p Y_t\}_{t\in I}$ is also a shadow system. Thus, $Z(\Lambda_t) \oplus_p - Z(\Lambda_t)$ is a shadow system. Rogers and Shephard proved in \cite{RS-58-2} that the volume of a shadow system of convex sets is a convex function in $t$. Therefore, for any $t$, we have that the function
    $$
        R:t \mapsto \frac{|Z(\Lambda_t) \oplus_p -Z(\Lambda_t)|}{|Z(\Lambda_t)|}
    $$ 
    is convex for $t \in [-\frac{1}{a},1]$. Thus, the maximum is attained at either $ t = -\frac{1}{a}$ or $t =1$. Now, we repeat the process finitely many times, and we arrive at an affine image of a unit cube. Note that the process ends since the number of segments is either reduced by 1 or, after at most $n$ steps, the number of vectors is reduced. 

    If the origin is in the interior $Z$, then there exists a shadow system $\{Z_t = Z+tu\}_{t \in [-d,c]}$ for some $0 <c,d$ such that $Z_0 = Z$ and the origin belongs to the boundaries of $Z_{-d}$ and $Z_{c}$. If the origin is not at the vertex, then there exists a shadow system (family of translations of $Z$ along the face) whose endpoints are zonotopes with one vertex at the origin. By the same reasoning as above, we arrive at a parallelotope. 
    Next, by applying a linear transformation, we may assume that $Z=[0,1]^n$. Using \cite[Lemma 15]{MNZ-25}, we get
    \[
    |[0,1]^n\oplus_p[-1,0]^n|=\sum_{i=0}\sum_{E } \binom{n/q}{i/q}^{-1}      \left| P_E[0,1]^n\right| \left| P_{E^\perp} [-1,0]^n \right|=\sum_{i=0}^n\frac{\binom{n}{i}}{\binom{n/q}{i/q}},
    \]
    where the sum is taken over all coordinate subspaces.
\end{proof}

Observe that in $\R^2$, the symmetric convex bodies are just the asymmetric zonoids. Thus, we obtain the following corollary. The result does not apply beyond dimension $3$ since the class of centered zonoids is not dense in the class of origin-symmetric convex bodies; see \cite{GW-93} for further discussion. We note that the following corollary can also be proved using the approach of Bianchini and Colesanti \cite{BC-08}, which is a different shadow system:  approximating the convex bodies with a polygon and using the shadow system to reduce the number of vertices in such a way that the bodies are symmetric with respect to a point.
\begin{corollary}\label{cor:center-sym}
    Let $K \subset \mathbb{R}^2$ be a convex body with a center of symmetry and containing the origin, and let $p>1$. Then
    % Let $p\geq 1$. For any convex body with a  $K \subset \R^2$ containing the origin, we have
    \begin{equation}
        |K \oplus_p -K | \leq \left(\frac{2\Gamma(1+\frac{1}{q})^2}{\Gamma(1+\frac{2}{q})} +2 \right)|K|,
    \end{equation}
    where $q$ satisfies $\frac{1}{p} + \frac{1}{q} = 1$.
    Equality holds if $K$ is a parallelogram with a vertex at the origin.
\end{corollary}

\begin{proof}[Second proof of Corollary~\ref{cor:center-sym}]
    By continuity and a standard density argument, it is enough to prove the result in the case where $K$ is a polygon. We are going to prove that parallelograms with a vertex at the origin are maximizers. Since $K$ has a center of symmetry, $K$ has an even number of vertices. Suppose that $K$ has at least 6 vertices. If the origin is in the interior of $K$, then there exists a shadow system $\{K_t = K+tu\}_{t \in [-d,c]}$ for some $0 <c,d$ and $ u \in \s^1$ such that $K_0 = K$, and the origin belongs to the boundaries of $K_{-d}$ and $K_{c}$. See the picture below.
%     \begin{tikzpicture}
%     \definecolor{myblue}{RGB}{20,120,230}

%     % =========================================================
%     % Parameters: vertices of an origin-symmetric polygon
%     % If v is a vertex, then -v is also a vertex.
%     % =========================================================
%     \coordinate (A1) at (-3.2,-0.2);
%     \coordinate (A2) at (-2.8,-1.4);
%     \coordinate (A3) at (-1.2,-2.0);
%     \coordinate (A4) at ( 2.0,-1.6);
%     \coordinate (A5) at ( 3.2, 0.2);
%     \coordinate (A6) at ( 2.8, 1.4);
%     \coordinate (A7) at ( 1.2, 2.0);
%     \coordinate (A8) at (-2.0, 1.6);

%     \begin{scope}[shift={(0,0)}]

%         % =====================================================
%         % Axes
%         % =====================================================
%         \draw[line width=1.1pt] (-4.4,-1.4) -- (4.4,-1.4);
%         \draw[line width=1.1pt] (-2.8,-3.0) -- (-2.8,3.0);

%         % =====================================================
%         % Origin-symmetric polygon K
%         % =====================================================
%         \fill[gray!15,opacity=.45]
%             (A1) -- (A2) -- (A3) -- (A4) --
%             (A5) -- (A6) -- (A7) -- (A8) -- cycle;

%         \draw[line width=1.2pt]
%             (A1) -- (A2) -- (A3) -- (A4) --
%             (A5) -- (A6) -- (A7) -- (A8) -- cycle;

%         % =====================================================
%         % Label
%         % =====================================================
%         \node[text=black] at (0.4,0.6) {$K$};

%     \end{scope}
% \end{tikzpicture}
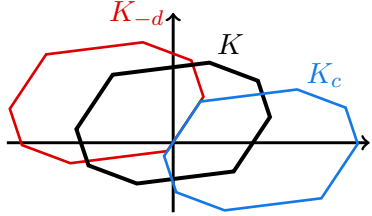
\begin{center}
\begin{tikzpicture}[scale =.4]
    \definecolor{myred}{RGB}{220,0,0}
    \definecolor{myblue}{RGB}{20,120,230}

    % =========================================================
    % Base origin-symmetric polygon P
    % =========================================================
    \coordinate (P1) at (-3.2,-0.35);
    \coordinate (P2) at (-2.8,-1.55);
    \coordinate (P3) at (-1.2,-2.15);
    \coordinate (P4) at ( 2.0,-1.75);
    \coordinate (P5) at ( 3.2, 0.05);
    \coordinate (P6) at ( 2.8, 1.25);
    \coordinate (P7) at ( 1.2, 1.85);
    \coordinate (P8) at (-2.0, 1.45);
    % \coordinate (P1) at (-3.0,-0.6);
    % \coordinate (P2) at (-1.2,-2.0);
    % \coordinate (P3) at ( 1.2,-2.0);
    % \coordinate (P4) at ( 3.0,-0.6);
    % \coordinate (P5) at ( 3.0, 0.6);
    % \coordinate (P6) at ( 1.2, 2.0);
    % \coordinate (P7) at (-1.2, 2.0);
    % \coordinate (P8) at (-3.0, 0.6);

    % =========================================================
    % Translations
    % Left copy: origin is the right endpoint of the bottom edge
    % Middle copy: origin is the midpoint of the bottom edge
    % Right copy: origin is the left endpoint of the bottom edge
    % =========================================================
    \pgfmathsetmacro{\txR}{2.9}
    \pgfmathsetmacro{\tyR}{ -.0875}

    \pgfmathsetmacro{\txC}{ 0.0}
    \pgfmathsetmacro{\tyC}{ 0.8}

    \pgfmathsetmacro{\txL}{ -2.2}
    \pgfmathsetmacro{\tyL}{ 1.4732}

    \begin{scope}[shift={(0,0)}]

        % =====================================================
        % Axes
        % =====================================================
        \draw[->,line width=1.1pt] (-5.5,0) -- (6.5,0);
        \draw[->,line width=1.1pt] (0,-2.3) -- (0,4.3);

        % =====================================================
        % Left translated polygon  K_{-d}
        % =====================================================
        \draw[myred,line width=1pt]
            ($ (P1)+(\txL,\tyL) $) --
            ($ (P2)+(\txL,\tyL) $) --
            ($ (P3)+(\txL,\tyL) $) --
            ($ (P4)+(\txL,\tyL) $) --
            ($ (P5)+(\txL,\tyL) $) --
            ($ (P6)+(\txL,\tyL) $) --
            ($ (P7)+(\txL,\tyL) $) --
            ($ (P8)+(\txL,\tyL) $) -- cycle;

        % =====================================================
        % Middle translated polygon  K
        % =====================================================
        \draw[line width=1.5pt]
            ($ (P1)+(\txC,\tyC) $) --
            ($ (P2)+(\txC,\tyC) $) --
            ($ (P3)+(\txC,\tyC) $) --
            ($ (P4)+(\txC,\tyC) $) --
            ($ (P5)+(\txC,\tyC) $) --
            ($ (P6)+(\txC,\tyC) $) --
            ($ (P7)+(\txC,\tyC) $) --
            ($ (P8)+(\txC,\tyC) $) -- cycle;

        % =====================================================
        % Right translated polygon  K_c
        % =====================================================
        \draw[myblue,line width=1pt]
            ($ (P1)+(\txR,\tyR) $) --
            ($ (P2)+(\txR,\tyR) $) --
            ($ (P3)+(\txR,\tyR) $) --
            ($ (P4)+(\txR,\tyR) $) --
            ($ (P5)+(\txR,\tyR) $) --
            ($ (P6)+(\txR,\tyR) $) --
            ($ (P7)+(\txR,\tyR) $) --
            ($ (P8)+(\txR,\tyR) $) -- cycle;

        % =====================================================
        % Labels
        % =====================================================
        \node[text=myred]  at (-1.2,4.25) {$K_{-d}$};
        \node[text=black]  at ( 1.9,3.25) {$K$};
        \node[text=myblue] at ( 5,2.25) {$K_c$};

    \end{scope}
\end{tikzpicture}
\end{center}
    The family $\{ K_t \oplus_p -K_t\}_{t \in [-d,c]}$ is also a shadow system (family of translations of $K$), and therefore, since $|K_t|=|K|$ for any $t$, we have that the function
    $$
        R:t \mapsto \frac{|K_t \oplus_p -K_t|}{|K_t|}
    $$ 
    is convex for $t \in [-d,c]$. Thus, the maximum is attained at either $t = c$ or $t=-d$. Hence, we can assume that $K$ is a convex body with the origin on its boundary. 

    If the origin belongs to the interior of an edge of $K$, then there exists a shadow system (family of translations of $K$ along the edge) whose endpoints are polygons with one vertex at the origin. By the same reasoning as above, we may assume that $K$ is a polygon with one vertex at the origin.
    
    For $N\geq 3$, let $v_1=0,\ldots,v_{2N}$ be the vertices of $K$, listed in counterclockwise order. Thus, $v_2$ and $v_{2N}$ are the vertices adjacent to $v_1$.
    % Denote by $A_i$ a triangle with vertices are $v_{i-1},v_i$ and $v_{i+1}$ where $v_0 := v_{2N}$. 
    There exists a shadow system $\{K_t\}_{t \in [t_0, t_1]}$ obtained by moving the vertices $v_{2}$ and $ v_{N+2}$ in such a way that $K_t$ remains symmetric with respect to the center of symmetry of $K$, and moreover that the polygons $K_0 = K, K_{t_0}$ and $K_{t_1}$ have at most $2N-2$ vertices. Note that along the movement, the volume of $K_t$ is fixed and the origin belongs to $K_t$ for all $t \in [t_0,t_1]$. We can then use the same argument as above to show that the maximizer along this movement is reached at one of the endpoints, so it is a polygon with two fewer vertices. By repeated application of this method, we reduce the polygon to a parallelogram. 
    
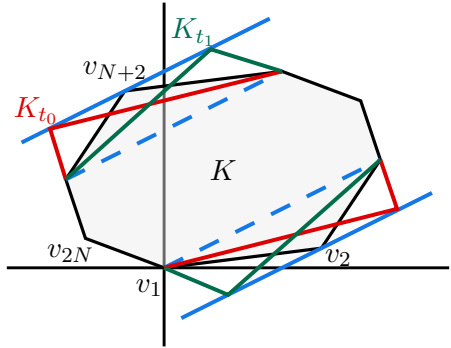
\begin{center}
    \begin{tikzpicture}[scale=0.65]
    \definecolor{myblue}{RGB}{20,120,230}
    \definecolor{myred}{RGB}{220,0,0}
    \definecolor{mygreen}{RGB}{0,110,75}

    % =========================================================
    % Parameters: vertices of K
    % DO NOT CHANGE BLACK K
    % =========================================================
    \coordinate (A1) at (-3.2,-0.2);
    \coordinate (A2) at (-2.8,-1.4);
    \coordinate (A3) at (-1.2,-2.0);
    \coordinate (A4) at ( 2.0,-1.6);
    \coordinate (A5) at ( 3.2, 0.2);
    \coordinate (A6) at ( 2.8, 1.4);
    \coordinate (A7) at ( 1.2, 2.0);
    \coordinate (A8) at (-2.0, 1.6);
    \node at (-3.1,-1.7) {$v_{2N}$};
    \node at (-2.2,2) {$v_{N+2}$};
    \node at (2.35,-1.8) {$v_{2}$};
    \node at (-1.5,-2.4) {$v_{1}$};
    % \node at (-3.7,-.3) {$v_{2N}$};
    % =========================================================
    % Blue direction:
    % A1A7 and A3A5 are parallel.
    % Direction vector = (4.4,2.2), slope = 1/2.
    % Solid blue lines are parallel translates through A8 and A4.
    % =========================================================

    % solid blue line through A8
    \coordinate (Uleft)  at (-4.10,0.55);
    \coordinate (Uright) at ( 0.95,3.075);

    % solid blue line through A4
    \coordinate (Lleft)  at (-0.85,-3.025);
    \coordinate (Lright) at ( 4.25,-0.475);

    % red outside vertices
    \coordinate (P) at (-3.52,0.83);
    \coordinate (Q) at ( 3.54,-0.8);

    % lower green outside vertex
    \coordinate (B) at (0.10,-2.55);
    \coordinate (R) at (-.25,2.45);
    \begin{scope}[shift={(1.6,-.6)}]

        % =====================================================
        % Axes -- unchanged
        % =====================================================
        \draw[line width=1.1pt] (-6,-1.4) -- (3,-1.4);
        \draw[line width=1.1pt] (-2.8,-3.0) -- (-2.8,4.0);

        % =====================================================
        % Black polygon K -- unchanged
        % =====================================================
        \fill[gray!15,opacity=.45]
            (A1) -- (A2) -- (A3) -- (A4) --
            (A5) -- (A6) -- (A7) -- (A8) -- cycle;

        \draw[line width=1.2pt]
            (A1) -- (A2) -- (A3) -- (A4) --
            (A5) -- (A6) -- (A7) -- (A8) -- cycle;

        % =====================================================
        % Solid blue lines
        % Upper passes through A8.
        % Lower passes through A4.
        % Both are parallel to A1A7 and A3A5.
        % =====================================================
        \draw[myblue,line width=1.5pt]
            (Uleft) -- (Uright);

        \draw[myblue,line width=1.5pt]
            (Lleft) -- (Lright);

        % =====================================================
        % Dashed blue lines
        % Correct pairs: A1A7 and A3A5
        % =====================================================
        \draw[myblue,line width=1.5pt,dash pattern=on 8pt off 7pt]
            (A1) -- (A7);

        \draw[myblue,line width=1.5pt,dash pattern=on 8pt off 7pt]
            (A3) -- (A5);

        % =====================================================
        % Red broken lines
        % =====================================================
        \draw[myred,line width=1.5pt]
            (A1) -- (P) -- (A7);

        \draw[myred,line width=1.5pt]
            (A3) -- (Q) -- (A5);

        % =====================================================
        % Green broken lines
        % =====================================================
        \draw[mygreen,line width=1.5pt]
            (A1) -- (R) -- (A7);

        \draw[mygreen,line width=1.5pt]
            (A3) -- (B) -- (A5);
        \node[mygreen] at ( -2.2, 3.4) {$K_{t_1}$};
        \node[myred] at ( -5.4, 1.8) {$K_{t_0}$};
        % =====================================================
        % Label
        % =====================================================
        \node[text=black] at (-1.6,.6) {$K$};
    \end{scope}
\end{tikzpicture}
\end{center}
% Next, by applying a linear transformation, we may assume that $K$ is a rectangle in the first quadrant containing the origin and so $|K| =1$. 
% Using the formula \cite[Lemma 15]{MNZ-25}, we have
% \[
%     | K \oplus_p -K |= \frac{2\Gamma(1+\frac{1}{q})^2}{\Gamma(1+\frac{2}{q})} +2. \qedhere 
% \]
\end{proof}

% Inequality \eqref{eq:RG-sym-zonoid} is obvious when $p = 1$ since under Minkowski sum, the volume is translation invariant, and then inequality is indeed always equality. Also, for another endpoint, when $p = \infty$, this inequality was proved by Rogers and Shephard \cite[Theorem 3]{RS-58-2} where the equality holds if and only if $K$ is a pseudo double pyramid.

\begin{remark}
   Note that this inequality in $\R^2$ can also be established using the Orlicz sum instead of the $L_p$-sum as in the paper by Jin and Yuan \cite{JY-14}. The proof proceeds via the same arguments as in the second proof.
\end{remark}
\begin{conj}
    We conjecture that among planar convex bodies with a center of symmetry and containing the origin, for $p>1$, there is equality in Corollary~\ref{cor:center-sym} only for parallelograms with a vertex at the origin.
\end{conj}

\newpage
% \renewcommand*{\bibfont}{\footnotesize}
% \printbibliography
\bibliographystyle{siam}
% \bibliography{references}

\bigskip
\noindent Matthieu Fradelizi
\\
LAMA, Univ Gustave Eiffel, Univ Paris Est Creteil, 77447 Marne-la-Vall\'ee, France.
\\
E-mail address: matthieu.fradelizi@univ-eiffel.fr
\vspace{2mm}
\\
\noindent Auttawich Manui 
\\
Department of Mathematical Sciences, Kent State University, Kent, OH 44242, USA.
\\
E-mail address: amanui@kent.edu
\vspace{2mm}
\\
\noindent Mark Meyer
\\
LAMA, Univ Gustave Eiffel, Univ Paris Est Creteil, 77447 Marne-la-Vall\'ee, France.
\\
E-mail address: mark.meyer@univ-eiffel.fr
\vspace{2mm}
\\
\noindent Cheikh Saliou Ndiaye
\\
LAMA, Univ Gustave Eiffel, Univ Paris Est Creteil, 77447 Marne-la-Vall\'ee, France.
\\
E-mail address: cheikh-saliou.ndiaye@univ-eiffel.fr
\vspace{2mm}
\\
\end{document}